\documentclass[a4paper,11pt]{amsart}
\usepackage{amsmath,amsthm,amssymb,amsfonts,enumerate,color,esint}
\usepackage[pdftex]{graphicx}
\usepackage{float}
\usepackage[shortlabels]{enumitem}

\usepackage[colorlinks=true, allcolors=blue]{hyperref}
\usepackage{amsrefs}

\newcommand{\N}{\mathbb{N}}

\newcommand{\R}{\mathbb{R}}
\newcommand{\A}{\mathcal{A}}
\newcommand{\norm}[1]{\left\Vert#1\right\Vert}
\newcommand{\abs}[1]{\left\vert#1\right\vert}
\def\({\left(}
\def\){\right)}
\newcommand{\loc}{\operatorname{loc}}
\newcommand{\trace}{\operatorname{trace}}
\newcommand{\Dom}{\operatorname{Dom}}
\newcommand{\eps}{\varepsilon}

\newtheorem{thm}{Theorem}[section]
\newtheorem{prop}[thm]{Proposition}
\newtheorem{cor}[thm]{Corollary}
\newtheorem{lem}[thm]{Lemma}

\theoremstyle{definition}
\newtheorem{defn}[thm]{Definition}
\newtheorem{rem}[thm]{Remark}

\newtheorem{notation}[thm]{Notation}

\numberwithin{equation}{section}

\allowdisplaybreaks

\author[P. R. Stinga]{\href{https://pabloraulstinga.github.io/}{Pablo Ra\'ul Stinga}}

\address{Department of Mathematics\\
396 Carver Hall\\
 Ames, IA 50011\\ The United States of America}
\email{stinga@iastate.edu}

\author[M. Vaughan]{\href{https://maryvaughan.github.io/}{Mary Vaughan}}

\address{Department of Mathematics and Statistics\\
The University of Western Australia\\
35 Stirling Hwy\\ 
Crawley WA 6009\\ Australia}
\email{mary.vaughan@uwa.edu.au}



\begin{document}

\title[Schauder estimates for fractional elliptic equations]{Interior Schauder estimates for fractional elliptic equations in nondivergence form}

\begin{abstract}
We obtain sharp interior Schauder estimates for solutions to nonlocal Poisson problems driven by fractional powers of nondivergence form elliptic operators $(-a^{ij}(x) \partial_{ij})^s$, for $0<s<1$, in bounded domains under minimal regularity assumptions
on the coefficients $a^{ij}(x)$.  Solutions to the fractional problem are characterized by a local degenerate/singular extension problem.
We introduce a novel notion of viscosity solutions for the extension problem and implement Caffarelli's perturbation methodology in the corresponding degenerate/singular Monge--Amp\`ere geometry to prove Schauder estimates in the extension. This in turn implies interior Schauder estimates for solutions to the fractional nonlocal equation. 
Furthermore, we prove a new Hopf lemma, the interior Harnack inequality and H\"older regularity in the Monge--Amp\`ere geometry for viscosity solutions to the extension problem. 
\end{abstract}

\maketitle

\section{Introduction}

We prove interior Schauder estimates for solutions to nonlocal Poisson problems driven by fractional powers of nondivergence form elliptic operators
\begin{equation}\label{eq:Ls}
L^s = (-a^{ij}(x) \partial_{ij})^s \quad \hbox{in}~\Omega \quad \hbox{for}~0 < s < 1
\end{equation}
in bounded domains $\Omega \subset \R^n$, $n\geq1$, under minimal regularity assumptions on the coefficients $a^{ij}(x)$ and the domain. 

Equations involving fractional power operators as in \eqref{eq:Ls} in the minimal regularity regime arise naturally in probabilistic models 
of random jump processes in heterogeneous media and stochastic games with jumps \cite{Oksendal-Sulem},
finance \cite{Cont},  the theory of semipermeable membranes and the Signorini problem in elasticity \cite{Duvaut}, 
and in relation to the fractional Monge--Amp\`ere equation of Caffarelli--Charro \cite{Caffarelli-Charro,Jhaveri-Stinga}. 
See \cite{Stinga-Vaughan} for a detailed presentation of these applications. 

Despite the numerous applications, regularity of solutions remained an open question until the work initiated in \cite{Stinga-Vaughan},
where the Harnack inequality and H\"older regularity of solutions to the fractional nonlocal problem
\begin{equation}\label{eq:fractionalPDE}
\begin{cases}
(-a^{ij}(x) \partial_{ij})^s u = f & \hbox{in}~\Omega \\
u =  0 & \hbox{on}~\partial \Omega
\end{cases}
\end{equation}
were proved. In this paper, we continue the regularity analysis program by establishing interior Schauder 
estimates for solutions $u \in \Dom(L^s)$ to \eqref{eq:fractionalPDE}. 

Before presenting the results, let us briefly recall the definition of the fractional power operators \eqref{eq:Ls}. 
Assume that $\Omega$ satisfies a uniform exterior cone condition. 
The coefficients $a^{ij}(x):\Omega \to \R$ are symmetric $a^{ij}(x) = a^{ji}(x)$, $1\leq i,j \leq n$,
$a^{ij}(x) \in C(\Omega) \cap L^\infty(\Omega)$
and uniformly elliptic, meaning that there exist constants $0 < \lambda \leq \Lambda$ such that
\begin{equation}\label{eq:ellipticity}
\lambda \abs{\xi}^2 \leq a^{ij}(x) \xi_i \xi_j \leq \Lambda \abs{\xi}^2 \quad \hbox{for all}~\xi \in \R^n~\hbox{and}~x \in \Omega. 
\end{equation}
In this setting, we consider the nondivergence form elliptic operator 
\begin{equation}\label{eq:L-intro}
L = -a^{ij}(x)\partial_{ij} \equiv - \sum_{i,j=1}^n a^{ij}(x) \partial_{x_ix_j}, \quad x \in \Omega. 
\end{equation}
Now, it is not immediately obvious how to define fractional powers of $L$. 
Indeed, the Fourier transform method of defining fractional power operators is not the most adequate tool in our setting, particularly in bounded domains. 
On the other hand, the spectral method (like the the one used to define fractional powers of divergence form operators $(-\partial_i(a^{ij}(x) \partial_{ij}))^s$, see \cite{Caffarelli-Stinga,Stinga-Zhang}) is unsuitable since $L$ has no natural Hilbert space structure and, moreover, cannot be written in divergence form. 
Instead, we use the method of semigroups to define $L^s$ by
\begin{equation}\label{eq:semi-defn}
L^su =  \lim_{\eps \to 0}\frac{1}{\Gamma(-s)} \int_\eps^\infty (e^{-tL} u - u) \frac{dt}{t^{1+s}}
\end{equation}
where $0 < s < 1$, $\Gamma$ denotes the Gamma function, and $\{e^{-tL}\}_{t \geq 0}$ is the uniformly bounded $C_0$-semigroup generated by \eqref{eq:L-intro}.  
It is known that 
\begin{equation}\label{eq:dom}
u \in \Dom(L^s) \quad \hbox{if and only if} \quad \hbox{the limit in \eqref{eq:semi-defn} exists}
\end{equation}
and, in this case, the resulting limit is precisely $L^su$, see \cite{Berens}. 
Precise definitions and details are given in Section \ref{sec:fractional-powers}. 

We now present our main result regarding regularity of solutions to \eqref{eq:fractionalPDE}. 

\begin{thm}[Schauder estimates]\label{thm:schauder-intro}
Assume that $\Omega \subset \R^n$ is a bounded domain satisfying the uniform exterior cone condition, 
$a^{ij}(x) \in C(\Omega) \cap L^\infty(\Omega)$ are symmetric and satisfy \eqref{eq:ellipticity}, and $f \in C_0(\Omega) \cap C^{0,\alpha}(\Omega)$ for some $0 < \alpha < 1$. 
Let $u \in \Dom(L^s)$ be a solution to  \eqref{eq:fractionalPDE}. 
\begin{enumerate}[$(1)$]
\item If $0 < \alpha +2s < 1$, then $u \in C_{\loc}^{0,\alpha+2s}(\Omega)$ and, for any subdomain $\Omega' \subset \subset \Omega$, 
\[
\|u\|_{C^{0,\alpha+2s}(\Omega')} \leq C(\|u\|_{L^{\infty}(\Omega)} +\|f\|_{C^{0,\alpha}(\Omega)}). 
\]
\item If $1 < \alpha +2s < 2$, then $u \in C_{\loc}^{1,\alpha+2s-1}(\Omega)$ and, for  any subdomain $\Omega' \subset \subset \Omega$, 
\[
\|u\|_{C^{1,\alpha+2s-1}(\Omega')} \leq C(\|u\|_{L^{\infty}(\Omega)} +\|f\|_{C^{0,\alpha}(\Omega)}). 
\]
\item If $2 < \alpha +2s < 3$ and $a^{ij}(x) \in C^{0,\alpha+2s-2}(\Omega)$, then $u \in C_{\loc}^{2,\alpha+2s-2}(\Omega)$ and, for  any subdomain $\Omega' \subset \subset \Omega$, 
\[
\|u\|_{C^{2,\alpha+2s-2}(\Omega')} \leq C(\|u\|_{L^{\infty}(\Omega)} + \|f\|_{C^{0,\alpha}(\Omega)}). 
\]
\end{enumerate}
The constants $C$ above depend only on $n$, $s$, $\lambda$, $\Lambda$, $\alpha$,  the modulus of continuity of $a^{ij}$, and the distance between $\Omega'$ and $\partial \Omega$. 
\end{thm}

The description of $\Dom(L^s)$ in \eqref{eq:dom} is rather obscure and not very useful for our scope. 
Even so, Theorem \ref{thm:schauder-intro} is sharp in that we only assume $u \in \Dom(L^s)$. 
Furthermore, we prove the sharp interior Harnack inequality and H\"older regularity for solutions $u \in \Dom(L^s)$, see Remark \ref{rem:nonlocal-harnack}.

Our proof of Theorem \ref{thm:schauder-intro} is based on the extension problem characterization of fractional power operators in general Banach spaces \cite{Gale}, see also \cite{Biswas-Stinga}. In particular,
we consider the solution $U  = U(x,z): \overline{\Omega} \times [0,\infty) \to \R$ to the following local equation in nondivergence form and in one additional dimension:
\begin{equation*}\label{eq:fullequationforU}
\begin{cases}
a^{ij}(x) \partial_{ij}U + z^{2-\frac{1}{s}} \partial_{zz}U = 0 & \hbox{in}~\Omega \times \{z>0\}\\
U = u & \hbox{on}~\Omega \times \{z=0\}\\
U= 0 & \hbox{on}~\partial \Omega \times \{z \geq 0\}. 
\end{cases}
\end{equation*}
It was recently established in \cite{Biswas-Stinga} that $u \in \Dom(L^s)$ if and only if 
\[
\lim_{z \to 0} \frac{U(x,z) - U(x,0)}{z}  \equiv \partial_{z} U(x,0)  = -d_s (-a^{ij}(x)\partial_{ij})^su(x)
\]
where the constant $d_s>0$ is explicit and depends only on $0 < s < 1$. 
See Theorem \ref{thm:extension} for the precise statement. 
Therefore, to prove Theorem \ref{thm:schauder-intro}, we show that solutions $U$ to 
\begin{equation}\label{eq:extension-intro}
\begin{cases}
a^{ij}(x) \partial_{ij}U + z^{2-\frac{1}{s}} \partial_{zz}U = 0 & \hbox{in}~\Omega \times \{z>0\}\\
\partial_{z}U(x,0) = f(x)& \hbox{on}~\Omega \times \{z=0\}
\end{cases}
\end{equation}
are $C^{\alpha+2s}$ on the set $\{z=0\}$. 
The corresponding result then holds for the solution $u(x) = U(x,0)$ to \eqref{eq:fractionalPDE}.

While \eqref{eq:extension-intro} is now a local PDE problem, 
there are still many difficulties to overcome. 
For instance, the equation is not translation invariant in the $z$-variable, and the coefficient $z^{2-\frac{1}{s}}$ is singular when $0 < s < \frac{1}{2}$ and degenerate when $\frac{1}{2}<s<1$ as $z \to 0$. 
We also have to deal with the Neumann condition. 
Even in the case $s=\frac{1}{2}$, the problem \eqref{eq:extension-intro} can formally be written as a single equation
 in $\Omega \times [0,\infty)$ with a right hand side that is a singular measure with density $f(x)$ supported on $\{z=0\}$.

An essential observation for the study of \eqref{eq:extension-intro} is that the PDE can be recast as an equation comparable to a linearized Monge--Amp\`ere equation. 
To see this, consider the even reflection of $U$ in the variable $z$ given by $\tilde{U}(x,z) = U(x,|z|)$ for $x \in \Omega$, $z \in \R$. We continue to use $U$ instead of $\tilde{U}$ for ease and notice that it satisfies
\begin{equation}\label{eq:extensionPDE}
a^{ij}(x) \partial_{ij}U + |z|^{2-\frac{1}{s}} \partial_{zz}U = 0 \quad \hbox{in}~\Omega \times \{z\not=0\}. 
\end{equation}
Next, define the convex function $\Phi = \Phi(x,z)$ by
\[
\Phi(x,z) = \frac{1}{2}|x|^2 + \frac{s^2}{1-s} |z|^{\frac{1}{s}}, \quad (x,z) \in \R^{n+1}. 
\]
Since the Hessian of $\Phi$ is
\[
D^2\Phi(x,z) = \begin{pmatrix} I & 0 \\ 0 & |z|^{\frac{1}{s}-2}\end{pmatrix},
\]
where $I$ is the identity matrix acting on $\R^n$, 
the linearized Monge--Amp\`ere equation associated to $\Phi$ with zero right hand side is 
\begin{equation}\label{eq:linMA}
\trace((D^2\Phi)^{-1}D^2U) = \Delta_xU + |z|^{2-\frac{1}{s}} \partial_{zz}U = 0 \quad \hbox{for}~z\not= 0.
\end{equation}
Being that the coefficients $a^{ij}(x)$ satisfy \eqref{eq:ellipticity}, it follows that the coefficients in the nondivergence form equation \eqref{eq:extensionPDE} are comparable to the coefficients in the linearized Monge--Amp\`ere equation \eqref{eq:linMA}. 

There is an intrinsic geometry associated to linearized Monge--Amp\`ere equations, as discovered by Caffarelli--Guti\'errez \cite{Caffarelli-Gutierrez}. 
We showed in \cite{Stinga-Vaughan} that the geometry for the degenerate/singular equation \eqref{eq:extension-intro} is the linearized Monge--Amp\`ere geometry associated to $\Phi$, see also \cite{Maldonado-Stinga} for a study of the fractional nonlocal linearized Monge--Amp\`ere equation.  
Specifically, there is a quasi-metric measure space associated with $\Phi$, and all our results regarding solutions to \eqref{eq:extension-intro} are in this setting. 
See Section \ref{sec:MA} for definitions and details. 

Regularity estimates for linearized Monge--Amp\`ere equations associated to smooth, convex functions $\psi$ were first studied by Caffarelli--Guti\'errez \cite{Caffarelli-Gutierrez} who proved the Harnack inequality and, later on, by Guti\'errez--Nguyen \cite{Gutierrez--Nguyen} who considered Schauder estimates. They worked
under the assumption that $\det D^2\psi$ is continuous and  bounded away from zero and infinity. 
For our function $\Phi$, we have that $D^2\Phi$ either degenerates or blows up at $\{z=0\}$ when $s\not=\frac{1}{2}$, so our problem does not fit into their setting. 
On the other hand, in their studies of Monge--Amp\`ere equations, Daskalopoulos--Savin \cite{Daskalopoulos-Savin}  and  Le--Savin \cite{Le-Savin} prove Schauder estimates for singular equations and degenerate equations, respectively, like \eqref{eq:linMA}. 
Maldonado has also studied regularity of solutions to degenerate elliptic equations associated to $\psi(x) = |x|^p$, $p \geq 2$, see \cite{MaldonadoI,MaldonadoIII}. 
However, our results are not contained in and do not follow from any of the aforementioned works. 
Not only are our techniques different than in \cite{Daskalopoulos-Savin, Le-Savin,MaldonadoI,MaldonadoIII}, 
their results are for Dirichlet problems and do not include the Neumann condition on the boundary $\{z=0\}$.  

Another significant difference with respect to the existing literature is that we consider viscosity solutions rather than strong solutions. 
Indeed, previous regularity estimates for linearized Monge--Amp\`ere equations are for classical solutions or $W^{2,n}_{\loc}$ solutions.  
As suggested by Caffarelli--Silvestre in \cite{Caffarelli-Silvestre}, one might try to use the $L^p$-viscosity theory. 
Instead, we introduce a new notion of continuous viscosity solution that is adapted to the degeneracy of \eqref{eq:extension-intro}. 
We feel that this might give a clue on how to build a viscosity solutions theory for linearized Monge--Amp\`ere equations in the
classical Caffarelli--Guti\'errez setting. 

As first observed in \cite[Remark 4.3] {Caffarelli-Silvestre}, 
the usual choice of $C^2$ test functions at the boundary $\{z=0\}$ is insufficient in the degenerate case $\frac{1}{2}<s<1$. For example, uniqueness does not hold.   
We define a new class of test functions to be 
the set of $\phi \in C^2_x \cap C^1_z$ whose weighted second derivative $z^{2- \frac{1}{s}} \partial_{zz} \phi$ is continuous up to the boundary $\{z=0\}$. 
We denote this set of test functions by $C_s$ and show that it is the correct class for dealing with the degeneracy and Neumann condition in \eqref{eq:extension-intro}. 
Definitions and preliminary results are given in Section \ref{sec:viscosity}. 
 
We prove that if $a^{ij}, f \in C^{0,\alpha}$, then viscosity solutions to \eqref{eq:extension-intro} are $(\alpha+2s)$-H\"older continuous with respect to the quasi-distance $\delta_\Phi$ associated to $\Phi$ at points on the boundary $\{z=0\}$. 
More specifically, if $\Omega'\subset \subset \Omega$ and $x_0 \in \Omega'$, we show that there is a  
Monge--Amp\`ere polynomial (namely, a polynomial associated to $\Phi$) such that
\[
\|U - P\|_{L^{\infty}(S_{r^2}(x_0,0)^+)} \leq Cr^{\alpha+2s} \quad \hbox{for $r$ small}.
\]
See Theorem \ref{thm:schauder-ext} for the precise statement. 
We will see that the scaling is different when $2<\alpha+2s<3$, since in this case $\frac{1}{2}<s<1$ and the equation is degenerate. 

For the proof, we implement a nontrivial adaptation of Caffarelli's perturbation argument of \cite{Caffarelli-Annals} for uniformly elliptic equations. In this regard, we need to study viscosity solutions $H= H(x,z)$ to 
\begin{equation}\label{eq:harmonic}
\begin{cases}
\Delta_xH + z^{2-\frac{1}{s}} \partial_{zz}H = 0 & \hbox{in}~S_1 \cap \{z>0\}\\
\partial_{z}H(x,0) = 0 & \hbox{on}~S_1 \cap \{z=0\}.
\end{cases}
\end{equation}
Throughout the paper, we will say that $H$ is \textbf{harmonic} if it satisfies \eqref{eq:harmonic}. 
We will show that viscosity solutions to \eqref{eq:harmonic} are in fact classical up to the boundary. 
Toward this end, we prove a new Hopf lemma for viscosity solutions by constructing new explicit barriers in the Monge--Amp\`ere geometry that can handle both the Neumann condition and the degeneracy of the equation.  

Furthermore, we need a Harnack inequality for viscosity solutions to \eqref{eq:extension-intro}. 
Recall that the a priori estimates in \cite[Theorem 1.3]{Stinga-Vaughan} are for classical solutions and thus are not sufficient. 
Our next result is the Harnack inequality and H\"older regularity for viscosity solutions to the extension equation with an extra nonzero right hand side $F$. 
For notation, see Section \ref{sec:MA}. 

\begin{thm}\label{thm:F-harnack}
Let $\Omega \subset \R^n$ be a bounded domain, 
$a^{ij}(x): \Omega \to \R$ be bounded, measurable and satisfy \eqref{eq:ellipticity}.
There exist positive constants $C_H = C_H(n,\lambda,\Lambda,s)>1$ and $\kappa = \kappa(n,s) < 1$
such that for every section $S_R = S_R(\tilde{x}, \tilde{z})\subset \subset \Omega \times \R$, 
 every $f \in L^{\infty}(S_{R}\cap \{z=0\})$, $F \in L^{\infty}(S_{R})$,
  and every nonnegative $C_s$-viscosity solution $U$, symmetric across $\{z=0\}$, to 
 \begin{equation}\label{eq:F-extension}
 \begin{cases}
 a^{ij}(x) \partial_{ij}U + |z|^{2-\frac{1}{s}} \partial_{zz}U = F & \hbox{in}~S_{R} \cap \{z\not=0\} \\
 \partial_{z} U(x,0) = f & \hbox{on}~S_{R}  \cap \{z=0\},
 \end{cases}
 \end{equation}
 we have that
 \begin{equation*}\label{eq:harnack-inf}
 \sup_{S_{\kappa R} }U \leq C_H \left(  \inf_{S_{\kappa R} }U
 	+ \|f\|_{L^{\infty}(S_{R} \cap \{z=0\})} R^s 
	+\|F\|_{L^{\infty}(S_{R})} R \right).
 \end{equation*}
Consequently, there exist constants $0 < \alpha_1  = \alpha_1(n,\lambda,\Lambda,s)< 1$ and $\hat{C} = \hat{C}(n,\lambda,\Lambda,s)>1$ such that, 
for every $C_s$-viscosity solution  $U$, symmetric across $\{z=0\}$, to \eqref{eq:F-extension}
it holds that
\begin{align*}
|U(\tilde{x},&\tilde{z}) - U(x,z)|\\
	&\leq \frac{\hat{C}}{R^{\frac{\alpha_1}{2}}} [\delta_\Phi((\tilde{x},\tilde{z}),(x,z))]^{\frac{\alpha_1}{2}}
	\left(  \sup_{S_{ R} }|U| +  \|f\|_{L^{\infty}(S_{R} \cap \{z=0\})} R^s 
	+\|F\|_{L^{\infty}(S_{R})} R \right)
\end{align*}
for every $(x,z) \in S_{R}$. 
\end{thm}

\begin{rem}[Harnack inequality and H\"older regularity for the fractional problem]\label{rem:nonlocal-harnack}
We recall that the Harnack inequality and H\"older regularity results in \cite[Theorem 1.1]{Stinga-Vaughan} for the nonlocal problem \eqref{eq:fractionalPDE} were established for solutions $u \in \Dom(L)$ 
under the extra assumption that $a^{ij}(x) \in C^{0,\alpha}(\Omega)$ for some $0 <\alpha < 1$. 
With Theorem \ref{thm:F-harnack} and Theorem \ref{thm:extension} now in hand, 
\cite[Theorem 1.1]{Stinga-Vaughan} holds under the sharp assumption that $u \in \Dom(L^s)$ and without the additional hypothesis that $a^{ij}(x)$ are H\"older continuous, but only continuous and bounded.
\end{rem}

For the proof of Theorem \ref{thm:F-harnack}, we implement Savin's method of sliding paraboloids that was first used in the uniformly elliptic setting in \cite{Savin} (see also \cite[Chapter 10]{Stinga-book} for a presentation for nondivergence form elliptic equations). 
In \cite{Stinga-Vaughan}, we developed the method of sliding paraboloids in the Monge--Amp\`ere geometry for classical solutions. 
Our main novelty here is the proof for viscosity solutions. 
Since the equation in \eqref{eq:extension-intro} is not translation invariant in the $z$-variable, it is not clear how to regularize with $\inf$/$\sup$-convolutions. 
Indeed, one might be tempted to use the Monge--Amp\`ere quasi-distance or regularize only in the horizontal direction like in \cite{DeSilva-Ferrari-Salsa}. 
However, we successfully adapt $\inf$/$\sup$-convolutions for the extension equation by carefully analyzing the degeneracy of the equation, see Section \ref{sec:harnack}. 

The rest of the paper is organized as follows. 
First, in Section \ref{sec:fractional-powers}, we precisely define the fractional operators $(-a^{ij}(x)\partial_{ij})^s$ and state the extension characterization. 
Then, in Section \ref{sec:MA}, we provide the necessary background on the Monge--Amp\`ere geometry associated to $\Phi$. 
We define $C_s$-viscosity solutions and prove preliminary results in Section \ref{sec:viscosity}. 
Section \ref{sec:harmonic} is devoted to proving a new Hopf lemma and establishing regularity of $C_s$-viscosity solutions to the harmonic equation \eqref{eq:harmonic}. 
We prove Theorem \ref{thm:F-harnack} in Section \ref{sec:harnack}.
In Section \ref{sec:approx}, we show an approximation lemma.
Finally, in Section \ref{sec:schauder}, we prove Schauder estimates for the extension equation on the set $\{z=0\}$ and obtain Theorem \ref{thm:schauder-intro}. 

\section{Fractional power operators and extension problem}\label{sec:fractional-powers}

In this section, we give the definition of the fractional power operator $L^s = (-a^{ij}(x) \partial_{ij})^s$ in \eqref{eq:fractionalPDE} and state the extension problem characterization. 
For this, we first present some general definitions and results regarding fractional powers and the method of semigroups. 

A family of bounded, linear operators $\{T_t\}_{t \geq 0}$ on a Banach space $X$ is a semigroup on $X$ if $T_0=I$ (the identity operator on $X$) and $T_{t_1} \circ T_{t_2} = T_{t_1+t_2}$ for every $t_1,t_2 \geq 0$. 
If, in addition, $T_tu \to u$ as $t \to 0$ for all $u \in X$, then $\{T_t\}_{t \geq 0}$ is a $C_0$-semigroup. 
A semigroup  $\{T_t\}_{t\geq0}$ is a uniformly bounded if there is $M \geq 1$ such that $\|T_t\|\leq M$ for all $t \geq 0$. 

The infinitesimal generator $A$ of a semigroup $\{T_t\}_{t \geq 0}$ is the closed linear operator 
\[
-Au := \lim_{t \to 0} \frac{T_tu - u}{t}
\]
in the domain $\Dom(A) = \{u \in X : -Au~\hbox{exists}\}$. 
In this case, we write $T_t = e^{-tA}$. 
On the other hand, a linear operator $(A, \Dom(A))$ on $X$ is said to generate a semigroup if there is a semigroup $\{T_t\}_{t\geq0}$ for which $A$ is its infinitesimal generator, that is, $T_t = e^{-tA}$. 
See \cite{Pazy,Yosida} for more on the theory of semigroups.   

If $A$ is the generator of a uniformly bounded $C_0$-semigroup $\{e^{-tA}\}_{t\geq0}$ on $X$, 
then Berens--Butzer--Westphal proved in \cite{Berens}
 that $u \in \Dom(A^s)$, $0 < s < 1$, if and only if
\begin{equation}\label{eq:As-domain}
w := \lim_{\eps \to 0} \frac{1}{\Gamma(-s)} \int_\eps^\infty (e^{-tA}u-u) \frac{dt}{t^{1+s}} \quad \hbox{exists in}~X,
\end{equation}
and in this case, the fractional power operator is precisely $w = A^su$.

Now, in our setting, we assume that the bounded domain $\Omega$ satisfies the uniform exterior cone condition, namely, there is a right circular cone $\mathcal{C}$ such that for all $x \in \partial \Omega$, there is a cone $\mathcal{C}_x$ with vertex $x$ that is congruent to $\mathcal{C}$ and such that $\overline{\Omega} \cap \mathcal{C}_x = \{x\}$.  
We consider the Banach space
\[
C_0 (\Omega) = \{u \in C(\overline{\Omega}) : u=0~\hbox{on}~\partial \Omega \} 
\]
endowed with the $L^{\infty}(\Omega)$ norm. 
Let $L$ be the linear operator on $C_0(\Omega)$ given by 
\begin{equation*}\label{eq:L-defn}
L= -a^{ij}(x) \partial_{ij}, \quad \Dom(L) = \{u \in C_0(\Omega) \cap W_{\loc}^{2,n}(\Omega) : Lu \in C_0(\Omega)\}
\end{equation*}
where the coefficients $a^{ij}(x)\in C(\Omega) \cap L^\infty(\Omega)$ are symmetric and satisfy \eqref{eq:ellipticity}. 
Under these hypotheses, it was established in \cite[Proposition 4.7]{Arendt} that $L$ generates a uniformly bounded $C_0$-semigroup $\{e^{-tL}\}_{t \geq 0}$ on $C_0(\Omega)$.
Consequently, we can define the fractional power operator $L^s= (-a^{ij}(x)\partial_{ij})^s: \Dom(L^s) \to C_0(\Omega)$
as in \eqref{eq:As-domain} with $L$ in place of $A$.

See \cite{Stinga-Vaughan} for further remarks on pointwise formulas for $(-a^{ij}(x)\partial_{ij})^su(x)$
and the definition of the negative fractional powers $(-a^{ij}(x)\partial_{ij})^{-s}f(x)$. 

Fractional powers of infinitesimal generators of uniformly bounded $C_0$-semigroups can be characterized by extension problems. 
See \cite{Caffarelli-Silvestre} for the fractional Laplacian on $\R^n$,
\cite{Stinga-Torrea} for Hilbert spaces and \cite{Gale} for general Banach spaces. 
We will use the recent sharp results of \cite{Biswas-Stinga} as they provide a full characterization of $\Dom(L^s)$ in terms of the extension problem, which we find to be more practical than \eqref{eq:As-domain}. 
After a change of variables as in the proof of Proposition \ref{prop:H regularity},
we obtain the following particular case of  \cite[Theorem 1.1]{Biswas-Stinga}. 

\begin{thm}[Particular case of \cite{Biswas-Stinga}]\label{thm:extension}
Assume that the bounded domain $\Omega \subset \R^n$ satisfies the uniform exterior cone condition and $a^{ij}(x) \in C(\Omega) \cap L^\infty(\Omega)$ are symmetric and satisfy \eqref{eq:ellipticity}. 
If $u \in C_0(\Omega)$ then a solution $U \in C^{\infty}((0,\infty);\Dom(L)) \cap C([0,\infty);C_0(\Omega))$ to the extension problem
\[
\begin{cases}
a^{ij}(x) \partial_{ij} U + z^{2-\frac{1}{s}} \partial_{zz} U = 0 & \hbox{in}~\Omega \times \{z>0\}\\
U(x,0) = u(x) & \hbox{on}~\Omega \times \{z=0\}\\
U= 0 & \hbox{on}~\partial\Omega \times \{z \geq 0\}
\end{cases}
\]
is given by
\[
U(x,z) = \frac{s^{2s}z}{\Gamma(s)} \int_0^\infty e^{-s^2z^{1/s}/t} e^{-tL} u(x) \frac{dt}{t^{1+s}}
\]
and satisfies $\|U(\cdot, z)\|_{L^\infty(\Omega)} \leq M \|u\|_{L^{\infty}(\Omega)}$ for some $M >0$. 
Moreover, $u \in \Dom(L^s)$ if and only if
\[
\lim_{z \to 0} \frac{U(x,z) - U(x,0)}{z} \equiv \partial_{z}U(x,0) \quad \hbox{exists in}~C_0(\Omega)
\]
and, in this case,
\[
 \partial_{z}U(x,0)  = -d_s L^su
\]
where $d_s = \frac{s^{2s}\Gamma(1-s)}{\Gamma(1+s)}>0$ and, furthermore,  
 $U$ is the unique classical solution to the initial boundary value extension problem
\[
\begin{cases}
a^{ij}(x) \partial_{ij} U + z^{2-\frac{1}{s}} \partial_{zz} U = 0 & \hbox{in}~\Omega \times \{z>0\}\\
U(x,0) = u(x) & \hbox{on}~\Omega \times \{z=0\}\\
\partial_{z}U(x,0)  = -d_s L^su & \hbox{on}~\Omega \times \{z=0\}\\
U= 0 & \hbox{on}~\partial\Omega \times \{z \geq 0\}.
\end{cases}
\]
\end{thm}

\section{Monge--Amp\`ere setting}\label{sec:MA}

In this section, we present background and preliminaries on the Monge--Amp\`ere geometry associated to $\Phi$ 
and set notation for the rest of the article. 
We refer the reader to \cite{Forzani,Gutierrez} for more details on the Monge--Amp\`ere geometry associated to general convex functions.

\subsection{Monge--Amp\`ere geometry}

For $0 < s < 1$, define the functions $\varphi:\R^n \to \R$ and $h:\R \to \R$ by 
\begin{equation}\label{eq:phi-h}
\varphi(x) = \frac{1}{2} |x|^2 \quad \hbox{and} \quad h(z) = \frac{s^2}{1-s} |z|^{\frac{1}{s}}. 
\end{equation}
Observe that $\varphi \in C^{\infty}(\R^n)$ and $h \in C^1(\R) \cap C^2(\R \setminus \{0\})$ are strictly convex functions. 
Define next the strictly convex function $\Phi:\R^{n+1} \to \R$ by
\begin{equation}\label{eq:Phi}
\Phi(x,z) = \varphi(x) + h(z).
\end{equation}

The Monge--Amp\`ere measure associated to a strictly convex function $\psi \in C^1(\R^n)$ is the Borel measure given by
\[
\mu_\psi(E) = |\nabla \psi(E)| \quad \hbox{for every Borel set}~E \subset \R^n,
\]
where $|A|$ denotes the Lebesgue measure of a measurable set $A \subset \R^n$. For Borel sets $I \subset \R$, $A \subset \R^n$, and $E \subset \R^{n+1}$, we have that
\[
\mu_h(I) = \int_I h''(z) \, dz, \quad \mu_\varphi(A) = |A|, \quad \hbox{and} \quad \mu_\Phi(E)= \int_E h''(z) \, dz \, dx,
\]
see \cite[Lemma 4.1]{Stinga-Vaughan}. 

The Monge--Amp\`ere quasi-distance associated to a strictly convex function $\psi \in C^1(\R^n)$ is
\[
\delta_{\psi}(x_0,x) = \psi(x) - \psi(x_0) - \langle \nabla \psi(x_0), x-x_0 \rangle. 
\]
By convexity, $\delta_{\psi} \geq 0$ and $\delta_{\psi}(x_0,x) = 0$ if and only if $x=x_0$. 
We use the term quasi-distance when there is a constant $K \geq 1$ such that
\[
\delta_{\psi}(x_1,x_2) \leq K (\min\{\delta_{\psi}(x_1,x_3),\delta_{\psi}(x_3,x_1)\} + \min\{\delta_{\psi}(x_2,x_3),\delta_{\psi}(x_3,x_2)\})
\]
for any $x_1,x_2,x_3 \in \R^n$. 
In the particular case of $\phi$, $h$, and $\Phi$ given above,
we note that
\begin{equation}\label{eq:distances}
\begin{aligned}
\delta_{\varphi}(x_0,x) &= \tfrac{1}{2} |x-x_0|^2\\
\delta_{h}(z_0,z) &= h(z) - h(z_0) - h'(z_0)(z-z_0) \\[.25em]
 \delta_{\Phi}((x_0,z_0),(x,z)) &= \delta_{\varphi}(x_0,x) + \delta_{h}(z_0,z).
\end{aligned}
\end{equation}
By \cite[Corollary 4.7]{Stinga-Vaughan}, $\delta_\varphi$, $\delta_h$, and $\delta_\Phi$ are indeed quasi-distances with constant $K$ depending only on $n$ (for $\delta_\varphi$ and $\delta_\Phi$) and $s$ (for $\delta_h$ and $\delta_\Phi$). 

The Monge--Amp\`ere section of radius $R>0$, centered at $x_0 \in \R^n$, associated to a strictly convex function $\psi \in C^1(\R^n)$ is given by
\[
S_\psi(x_0,R) = \{x \in \R^n : \delta_\psi(x_0,x)<R\}. 
\]
Since we are concerned specifically with $\varphi$, $h$, and $\Phi$, we adopt the following notation. 

\begin{notation}
Unless otherwise stated, we always use the following notation. 
\begin{itemize}
\item $x = (x_1,x_2,\dots, x_n) \in \R^n$, $z \in \R$.
\item $S_R(x) \subset \R^n$ is a section of radius $R>0$ associated to $\varphi$, centered at $x$. 
\item $S_R(z) \subset \R$ is a section of radius $R>0$ associated to $h$, centered at $z$.
\item $S_R(x,z) \subset \R^{n+1}$ is a section of radius $R>0$ associated to $\Phi$, centered at $(x,z)$. 
\end{itemize}
\end{notation}

Sections of radius $R>0$ associated to $\varphi$ are equivalent to Euclidean balls of radius $\sqrt{R}$ in the following way:
\begin{equation}\label{eq:x-MA-euclidean}
S_R(x_0) 
	= \{x \in \R^n : \tfrac{1}{2} |x- x_0|^2 < R \} 
	= B_{\sqrt{2R}}(x_0). 
\end{equation}

Sections of radius $R>0$ associated to $h$ are intervals in $\R$. 
Since $h''(z) = |z|^{\frac{1}{s}-2}$ is singular/degenerate near $z=0$ when $s \not= \frac{1}{2}$, in general, we cannot provide a precise relationship between the radius/center of the section in the Monge--Amp\`ere geometry and the radius/center of the interval in the Euclidean geometry. 
Nevertheless, we make note of two special cases.  
First, when the section is centered at the origin $z_0=0$, it is an interval of radius comparable to $R^s$:
\begin{equation*}\label{eq:z-MA-euclidean}
S_R(0) = \{ z \in \R : h(z) < R\} = \{ z \in \R : |z| < q_s R^s \} = B_{q_sR^s}(0), \quad q_s := \left( \frac{1-s}{s^2}\right)^s. 
\end{equation*} 
On the other hand, when separated from the set $\{z=0\}$, sections are comparable to intervals of radius $\sqrt{R}$: 

\begin{lem}\label{lem:balls-to-sections-away-from-0} 
Let $R>0$ and $z_0 \in \R \setminus \{z=0\}$. 
If $B_R(z_0) \subset \subset \{z\not=0\}$, then
\begin{equation}\label{eq:sec-awayfrom0}
B_R(z_0) \subset S_{\frac{\sigma}{2} R^2}(z_0) \quad \hbox{where}~\sigma := \sup_{B_R(z_0)} h''.  
\end{equation}
If $S_R(z_0) \subset \subset \{z\not=0\}$, then
\[
S_R(z_0) \subset \subset B_{\sqrt{2R/\tilde{\sigma}}}(z_0) \quad \hbox{where}~\tilde{\sigma} := \inf_{S_R(z_0)} h''.
\]
\end{lem}

\begin{rem}
For $\frac{1}{2}<s<1$, the function $h''$ is singular at the origin, so 
if $0 \in\overline{B_R(z_0)}$, then $\sigma = +\infty$. 
For $0 < s<\frac{1}{2}$, the function $h''$ is instead degenerate at the origin, so if $0 \in\overline{S_R(z_0)}$, then $\tilde{\sigma} = 0$. 
In both of these cases, Lemma \ref{lem:balls-to-sections-away-from-0} is ineffectual. 
Of course, when $s = \frac{1}{2}$, sections are equivalent to Euclidean balls since $h(z) = \frac{1}{2}|z|^2$. 
\end{rem}

\begin{proof}[Proof of Lemma \ref{lem:balls-to-sections-away-from-0}]
Suppose first that $B_R(z_0) \subset \subset \{z\not=0\}$.
If $z \in B_R(z_0)$, then by Taylor's theorem, 
\[
\delta_h(z_0,z)
	= h(z) - h(z_0) - h'(z_0)(z-z_0)
	\leq  \frac{1}{2} \|h''\|_{L^{\infty}(B_R(z_0))}(z-z_0)^2 \leq \frac{\sigma}{2}R^2
\]
which shows that $z \in S_{\frac{\sigma}{2} R^2}(z_0)$.

Now suppose that $S_R(z_0) \subset \subset \{z\not=0\}$. 
If $z \in S_R(z_0)$, then by Taylor's theorem, there is some $\xi$ between $z$ and $z_0$ such that
\[
R > \delta_h(z_0,z)
	= h(z) - h(z_0) - h'(z_0)(z-z_0)
	= \frac{1}{2} h''(\xi)(z-z_0)^2 \geq \frac{\tilde{\sigma}}{2} (z-z_0)^2.
\]
It follows that $z \in B_{\sqrt{2R/\tilde{\sigma}}}(z_0)$. 
\end{proof}

\begin{rem}
From the proof of Lemma \ref{lem:balls-to-sections-away-from-0}, we see that the Monge--Amp\`ere distance $\delta_h$ is comparable to the Euclidean distance away from $\{z=0\}$. 
\end{rem}

There are often times  
when it is necessary to use cubes or cylinders instead of Euclidean balls, or in our case, Monge--Amp\`ere sections. 
To this end, we define a Monge--Amp\`ere cube of radius $R>0$ centered at $x \in \R^n$ associated to $\varphi$ by
\[
Q_R(x) = S_{\varphi_1}(x_1,R) \times \dots \times S_{\varphi_n}(x_n,R)
\]
where  $x=(x_1,\dots, x_n)$ and $\varphi_i:\R \to \R$ is defined by $\varphi_i(x) = \frac{1}{2} |x_i|^2$ for $i=1,\dots,n$. 
A Monge--Amp\`ere cube of radius $R>0$, centered at $(x,z) \in \R^{n+1}$ associated to $\Phi$ is given by
\[
Q_R(x,z) :=Q_R(x) \times S_R(z). 
\]
With this, we adopt the following notation for Monge--Amp\`ere cubes, cylinders, and rectangles, and other set related notation
that will be used throughout the rest of the paper. 

\begin{notation}\label{not:sets2}
Unless otherwise stated, we always use the following notation. 
\begin{itemize}
\item $Q_R(x) \subset \R^n$ is a Monge--Amp\`ere cube of radius $R>0$, centered at $x$.  
\item $Q_R(x,z) \subset \R^{n+1}$ is a Monge--Amp\`ere cube of radius $R>0$, centered at $(x,z)$.
\item $Q_R(x) \times S_r(z) \subset \R^{n}\times \R$ is a Monge--Amp\`ere rectangle of radius $R>0$, height $r>0$, centered at $(x,z)$.
\item $S_R(x) \times S_r(z) \subset \R^{n}\times \R$ is a Monge--Amp\`ere cylinder of radius $R>0$, height $r>0$, centered at $(x,z)$. 
\item If no center is specified, the center is the origin,~e.g.~$S_R \times S_R = S_R(0) \times S_R(0) \subset \R^{n} \times \R$. 
\item $T_R := S_R \times \{z=0\}$. 
\item $E^+ := E \cap \{z>0\}$ for a set $E \subset \R^{n+1}$ or $E \subset \R$.
\item $E^- := E \cap \{z<0\}$ for a set $E \subset \R^{n+1}$ or $E \subset \R$. 
\end{itemize}
\end{notation}

Note that sections, cylinders, and cubes are related in the following way
\begin{equation}\label{eq:section-cube}
S_R(x,z) \subset S_R(x) \times S_R(z) \subset Q_R(x) \times S_R(z) = Q_R(x,z),
\end{equation}
and similarly for cylinders and rectangles, see for example \cite[Lemma 10]{Forzani}.

We refer the interested reader to \cite[Section 4]{Stinga-Vaughan} for more foundational properties of the Monge--Amp\`ere geometry associated to $\varphi$, $h$, and $\Phi$ (especially Corollary 4.7 there). 
Here, we just recall two properties for sections associated to $h$ needed for our analysis and another on Monge--Amp\`ere cubes. 

First, since $h''(z) = |z|^{\frac{1}{s}-2}$ is a Muchenhoupt $A_\infty(\R)$ weight, we have the following. See \cite[Section 9.3]{Grafakos} for definitions and properties of the class $A_\infty(\R)$. 

\begin{lem}
\label{lem:A-infty}
Given $0 < \varepsilon<1$, there is $0 < \varepsilon_0<1$, depending only on $\varepsilon$ and $0 < s < 1$, such that for any section $S_R(z)$ and any measurable set $E \subset S_R(z)$,
\[
\frac{\abs{E}}{\abs{S_R(z)}} < \varepsilon_0 \quad \hbox{implies} \quad \frac{\mu_h(E)}{\mu_h(S_R(z))} < \varepsilon.
\]
\end{lem}

The next result is a consequence of \cite[Theorem 5]{Forzani} (see \cite[Corollary 4.7]{Stinga-Vaughan}).

\begin{lem}
\label{lem:doubling}
There exist constants constants $C,c>0$, depending only on $s$, such that
\[
cR \leq |S_R(z)|\mu_h(S_R(z)) \leq CR
\]
for all sections $S_R(z)$.
\end{lem}

Lastly, we have the following version of \cite[Theorem 3.3.10]{Gutierrez} adapted
to our setting (see also \cite[Corollary 4.7]{Stinga-Vaughan}). 

\begin{lem}
\label{lem:Guti} \mbox{}
\begin{enumerate}[$(1)$]
\item Let $x_0 \in \R^n$. 
There exist constants $C_0>0$, $p_0 \geq 1$, depending on $n$, such that for $0 < r_1< r_2 \leq 1$, $t>0$ and $x_1 \in Q_{r_1t}(x_0)$, we have that
\[
Q_{C_0(r_2-r_1)^{p_0}t}(x_1) \subset Q_{r_2t}(x_0).
\]
\item Let $z_0 \in \R$. 
There exist constants $C_1>0$, $p_1 \geq 1$, depending on $s$, such that for $0 < r_1< r_2 \leq 1$, $t>0$ and $z_1 \in S_{r_1t}(z_0)$, we have that
\[
S_{C_1(r_2-r_1)^{p_1}t}(z_1) \subset S_{r_2t}(z_0).
\]
\end{enumerate}
\end{lem}

\subsection{Monge--Amp\`ere H\"older spaces}\label{sec:Holder}

Now, we introduce H\"older spaces in the Monge--Amp\`ere geometry associated to $\varphi$ and $\Phi$ given in \eqref{eq:phi-h} and \eqref{eq:Phi}, respectively. 

Fix $0 < \alpha <1$. For a strictly convex function $\psi \in C^1(\R^n)$, we say that a function $u:\R^n \to \R$ is $\alpha$-H\"older continuous with respect to $\psi$ in a set $A \subset \R^n$ if
\[
|u(x) - u(x_0)| \leq C [ \delta_{\psi}(x_0,x)]^{\frac{\alpha}{2}} \quad \hbox{for all}~x,x_0 \in A. 
\]
where $\delta_{\psi}$ is the Monge--Amp\`ere quasi-distance associated to $\psi$. 
In this case, we write $u \in C_{\psi}^{\alpha}(A)$ and define the seminorm
\[
[u]_{C_{\psi}^{\alpha}(A)} := \sup_{\substack{x,x_0 \in A\\ x\not=x_0}} \frac{|u(x) - u(x_0)|}{[ \delta_{\psi}(x_0,x)]^{\frac{\alpha}{2}}}.
\]

Recalling \eqref{eq:distances}, the class $C_{\varphi}^{\alpha}(A)$ is the usual class of H\"older continuous functions, so we drop the $\varphi$ notation and simply write 
\[
C^{\alpha}(A) := C_{\varphi}^{\alpha}(A).
\] 
For $k \in \N \cup \{0\}$, the space $C^{k,\alpha}(A)$ is the H\"older space endowed with the norm
\[
\|u\|_{C^{k,\alpha}(A)} :=\|u\|_{C^k(A)} + \max_{|\beta|=k} [D^\beta u]_{C^{\alpha}(A)}. 
\]
We say that $u \in C^{k,\alpha}(x_0)$ for a point $x_0 \in A$ if there is a polynomial $P_{x_0}$ of degree $k$ such that, in the domain of $u$, 
\[
u(x) = P_{x_0}(x) + O(|x-x_0|^{k+\alpha}). 
\]

From the definition above, we have that $U \in C^{\alpha}_{\Phi}(E)$ for $E\subset \R^{n+1}$ if  
\[
|U(x,z) - U(x_0,z_0)| \leq C [ \delta_{\Phi}((x_0,z_0),(x,z))]^{\frac{\alpha}{2}} \quad \hbox{for all}~(x,z),(x_0,z_0) \in E.  
\]
\begin{rem}
As a consequence of Theorem \ref{thm:F-harnack}, we have that $C_s$-viscosity solutions, symmetric across $\{z=0\}$, to \eqref{eq:extension-intro} are in the class $C_{\Phi}^{\alpha_1}(E)$ for any subdomain $E \subset \subset \Omega \times \R$. 
\end{rem}

\begin{defn} \label{defn:MA-poly} 
We define Monge--Amp\`ere polynomials $P=P(x,z)$ with respect to $\Phi$ of order $k=0,1, 2$ in the following way. 
\begin{enumerate}
\item If $k=0$, then $P(x,z)$ is constant.
\item If $k=1$, then $P(x,z)$ is an affine function of $(x,z)$.
\item If $k=2$, then 
\[
P(x,z) = \frac{1}{2} \langle \mathcal{A}x,x \rangle  + \langle b,x \rangle z+ d h(z) + \ell(x,z)
\]
for some $n \times n$ matrix $\mathcal{A}$, vector $b \in \R^n$, constant $d \in \R$ and affine function $\ell(x,z)$. 
\end{enumerate}
\end{defn}

For $k=0,1,2$, we say that $U \in C^{k,\alpha}_\Phi(x_0,z_0)$ at a point $(x_0,z_0) \in E$
if there is a Monge--Amp\`ere polynomial $P_{(x_0,z_0)}$ with respect to $\Phi$ of order $k$ such that, in the domain of $U$,
\[
U(x,z) = P_{(x_0,z_0)}(x,z) + O(\delta_{\Phi}((x_0,z_0),(x,z))^{\frac{k+\alpha}{2}}). 
\]

\subsection{Scaling in the Monge--Amp\`ere geometry}\label{sec:scaling}

Lastly, we highlight how  Monge--Amp\`ere cylinders and the extension equation \eqref{eq:extensionPDE} scale. This is an important point for the proof of Schauder estimates. 

\begin{lem}\label{lem:scaling sections}
For any $(x_0,z_0) \in \R^{n+1}$ and any $R, r, \rho >0$, 
\[
(x,z) \in S_R(x_0) \times S_r(z_0) \quad \hbox{if and only if}\quad (\rho x, \rho^{2s} z) \in S_{\rho^2R}(\rho x_0) \times S_{\rho^2 r} (\rho^{2s}z_0),
\]
and similarly for Monge--Amp\`ere sections, cubes, and rectangles. 
Consequently, for Monge--Amp\`ere cylinders centered at the origin, namely $(x_0,z_0) = (0,0)$, 
\begin{equation}\label{eq:scaling0}
(x,z) \in S_{R}\times S_{r}  \quad \hbox{if and only if} \quad (\rho x, \rho^{2s} z) \in S_{\rho^2R} \times S_{\rho^2r}. 
\end{equation}
\end{lem}

\begin{proof}
Observe that $(x,z) \in S_R(x_0) \times S_r(z_0)$ if and only if
\begin{equation}\label{eq:Phi-sec}
\frac{1}{2} \abs{x-x_0}^2 <R \quad \hbox{and} \quad h(z) - h(z_0) - h'(z_0)(z-z_0) < r.
\end{equation}
It is a simple computation to check that $\rho^2h(z) = h(\rho^{2s}z)$ and 
$\rho^{2-2s}h'(z) =h'(\rho^{2s}z)$. 
With this, we multiply \eqref{eq:Phi-sec} on both sides by $\rho^2$ to equivalently write
\[
\frac{1}{2} \abs{\rho x-\rho x_0}^2 <\rho^2 R \quad \hbox{and} \quad h(\rho^{2s} z) - h(\rho^{2s} z_0) - h'(\rho^{2s} z_0)(\rho^{2s} z-\rho^{2s} z_0) < \rho^2 r,
\]
which means that $(\rho x, \rho^{2s} z) \in S_{\rho^2R} (\rho x_0) \times S_{\rho^2 r}(\rho^{2s} z_0)$. 
\end{proof}

Consequently, the equation scales as follows. 

\begin{lem} \label{lem:scaling equation}
Let $R, r,\rho>0$. 
A function $U =U(x,z)$ is a solution to 
\[
\begin{cases}
a^{ij}(x)\partial_{ij}U + z^{2-\frac{1}{s}} \partial_{zz} U= 0 & \hbox{in}~S_{\rho^2R} \times S_{\rho^2r}^+\\
\partial_{z}U(x,0) = f(x) & \hbox{on}~T_{\rho^2R}
\end{cases}
\]
if and only if $V(x,z)  =U(\rho x, \rho^{2s} z)$ solves
\[
\begin{cases}
a^{ij}(\rho x)\partial_{ij}V + z^{2-\frac{1}{s}} \partial_{zz} V= 0 & \hbox{in}~S_{R} \times S_{r}^+\\
\partial_{z}V(x,0) = \rho^{2s}f(\rho x) & \hbox{on}~T_{R}.
\end{cases}
\]
\end{lem}

\section{Viscosity solutions to the extension problem}\label{sec:viscosity}

In this section, we define the correct notion of viscosity solutions to the degenerate/singular extension problem
\eqref{eq:extension} and present some fundamental properties. 

For simplicity, we present the notions and results of this section only in $S_1^+ \cup T_1$ where we recall from Notation \ref{not:sets2} that $S_1^+ = S_1(0,0)^+$ and $T_1 = S_1(0,0) \cap \{z=0\}$. 
Nevertheless, we remark that the everything holds in more general subdomains of $\R^{n+1}$, such as Monge--Amp\`ere sections, cylinders, cubes, and rectangles, that may intersect $\{z=0\}$. 

\subsection{Definitions and preliminary results}

We say that a continuous function $\phi$ touches $U$ from above (below) at a point $(x_0,z_0) \in S_1^+$ if there is an open convex set $E \subset S_1^+$ such that $(x_0,z_0) \in E$, 
\begin{equation}\label{eq:touches}
\phi(x_0,z_0) = U(x_0,z_0) \quad \hbox{and} \quad \phi\geq U\quad (\phi \leq U) \quad \hbox{in}~E. 
\end{equation}
Similarly, we say that $\phi$ touches $U$ from above (below) at a point $(x_0,0) \in T_1$ if there is an open convex set $E\subset S_1^+ \cup T_1$ such that $(x_0,0) \in E$ and \eqref{eq:touches} holds.

\begin{defn}[Class $C_s$]
We define the class $C_s$ by
\[
C_s = \{ \phi \in C^2(S_1^+) \cap C_x^2(\overline{S_1^+}) \cap C_z^1(\overline{S_1^+}): z^{2-\frac{1}{s}} \partial_{zz}\phi \in C(\overline{S_1^+})\}.
\]
\end{defn}

For example, the Monge--Amp\`ere polynomials of Definition \ref{defn:MA-poly} are in the class $C_s$. 

\begin{defn}[$C_s$-viscosity solutions]\label{defn:viscosity-sub}
Let $a^{ij}(x)$ be bounded, measurable functions satisfying \eqref{eq:ellipticity} and let $f \in C(\overline{T_1})$, $F \in C(\overline{S_1^+})$. 
We say that $U \in C(\overline{S_1^+})$ is a $C_s$-viscosity subsolution (supersolution) to
\begin{equation}\label{eq:sub}
\begin{cases}
a^{ij}(x) \partial_{ij} U + z^{2 -\frac{1}{s}} \partial_{zz} U = F & \hbox{in}~S_1^+\\
\partial_{z} U = f & \hbox{on}~T_1
\end{cases}
\end{equation}
if the following conditions hold.
\begin{enumerate}[start=1,label={(\roman*)}]
	\item \label{item:visc-i}
	If $(x_0,z_0) \in S_1^+$ and $\phi \in C^2(S_1^+)$ touches $U$ from above (below) at $(x_0,z_0)$, then
	\[
	a^{ij}(x) \partial_{ij} \phi(x_0,z_0)+ |z_0|^{2 -\frac{1}{s}} \partial_{zz} \phi(x_0,z_0)  \geq F(x_0,z_0)\quad (\leq F(x_0,z_0)).
	\]
	\item \label{item:visc-ii}
	If $(x_0,0) \in T_1$ and $\phi \in C_s$ touches $U$ from above (below) at $(x_0,0)$, then
	\[
	\partial_z \phi(x_0,0) \geq f(x_0) \quad (\leq f(x_0)).
	\]
\end{enumerate}
We say that $U$ is a $C_s$-viscosity solution if it is both a $C_s$-viscosity subsolution and a $C_s$-viscosity supersolution. 
\end{defn}

We now describe some basic properties of the class $C_s$, beginning with the regularity in $z$.

\begin{lem}\label{lem:expansion}
If $\phi \in C_s$, then $\partial_z\phi \in C^{\eta}(\overline{S_1^+})$ for $\eta = \min(1, \frac{1}{s}-1)$. In particular, 
\[
\partial_z\phi \in \begin{cases}
C^1_z(\overline{S_1^+}) & \hbox{if}~0 < s \leq 1/2\\
 C^{\frac{1}{s}-1}_z(\overline{S_1^+})& \hbox{if}~1/2 < s < 1.
 \end{cases}
 \]
Moreover, for $(x_0,0) \in T_1$, we have
\[
\phi(x,z) \leq \phi(x_0,0) +  A \cdot (x-x_0)+\partial_z\phi(x_0,0)z + B |x-x_0|^2 + Cz^{1+\eta}
\]
where $\|\nabla_x \phi \|_{L^{\infty}} \leq |A|$, $\|D^2_x\phi \|_{L^{\infty}} \leq 2B$, and $C= C(\phi,s)>0$. 
\end{lem}

\begin{proof}
Since $z^{2 - \frac{1}{s}} \partial_{zz}\phi$ is a continuous function in $\overline{S_1^+}$, we have that
\[
|\partial_{zz}\phi(x,z)| \leq \frac{C}{z^{2-\frac{1}{s}}} = C h''(z).
\]
Consequently,
\[
|\partial_z\phi(x,z)- \partial_z\phi(x,0)| \leq C \int_0^z h''(\xi) \, d\xi = C h'(z) = C z^{\frac{1}{s}-1}.
\]
This shows that $\partial_z\phi \in C^{\eta}_z(\overline{S_1^+})$ for $\eta = \min(1, \frac{1}{s}-1)$. 

By Taylor expanding $\phi(x,z)$ in $x$ around $x_0$, we write
\begin{equation}\label{eq:taylorx}
\phi(x,z) = \phi(x_0,z) + \nabla_x \phi(x_0,z) \cdot (x-x_0) + \frac{1}{2} D^2_x \phi(\xi,z)(x-x_0) \cdot (x-x_0)
\end{equation}
for some $\xi$ between $x$ and $x_0$. On the other hand, since $\phi \in C^{1,\eta}_z(\overline{S_1^+})$, 
\begin{equation}\label{eq:holderz}
\phi(x_0,z)
	= \phi(x_0,0) + \partial_z\phi(x_0,0)z + O(z^{1+\eta}).
\end{equation}
The result follows by combining \eqref{eq:taylorx} and \eqref{eq:holderz}. 
\end{proof}

Next, we prove two useful characterizations of \ref{item:visc-ii} in Definition \ref{defn:viscosity-sub}. 

\begin{lem}[Characterization 1]\label{lem:either-or}
Condition \ref{item:visc-ii} is equivalent to the following.
\begin{enumerate}[start=2,label={(\roman*)'}]
	\item \label{item:visc-ii'}
	 If $(x_0,0) \in T_1$ and $\phi \in C_s$ touches $U$ from above at $(x_0,0)$, then either
	\[
	(a^{ij}(x_0) \partial_{ij} \phi+ z^{2 -\frac{1}{s}} \partial_{zz} \phi )\big|_{(x_0,0)} \geq F(x_0,0)
	\quad \hbox{or}\quad
	\partial_z \phi(x_0,0) \geq f(x_0). 
	\]
\end{enumerate}
\end{lem}

\begin{proof}
It is clear that \ref{item:visc-ii} implies \ref{item:visc-ii'}. 
Conversely, assume that \ref{item:visc-ii'} holds. 
Suppose $\phi \in C_s$ touches $U$ from above at $(x_0,0) \in T_1$. 
Assume, by way of contradiction, that
\[
\partial_z\phi(x_0,0) < f(x_0). 
\]
By \ref{item:visc-ii'}, it must be that
\[
(a^{ij}(x) \partial_{ij}  \phi+ z^{2 -\frac{1}{s}} \partial_{zz} \phi )\big|_{(x_0,0)} \geq F(x_0,0).
\]
Define the function $\psi = \psi(x,z)$ by
\[
\psi(x,z) = \phi(x,z) + \eta z - C h(z) \quad \hbox{in}~\overline{S_\tau(x_0,0)^+}
\]
for $\eta, \tau >0$ small and $C>0$ large, to be determined. Notice that, for $z>0$, 
\[
\eta z - C h(z)  > 0 
\quad \hbox{if and only if} \quad
0 < z < \left( \frac{\eta(1-s)}{Cs^2}\right)^{s/(1-s)}. 
\]
Take $\tau>0$ such that $\{z: 0 < h(z) < \tau \} \subset (0,(\eta(1-s)/(Cs^2)))^{s/(1-s)})$. 
We have that $\psi$ touches $\phi$ from above at $(x_0,0)$ in $\overline{S_\tau(x_0,0)^+}$. 
Since $\phi \in C_s$ and $\eta z -Ch(z) \in C_s$, it follows that $\psi \in C_s$. 
By \ref{item:visc-ii'}, either
	\[
	(a^{ij}(x) \partial_{ij}  \psi+ z^{2 -\frac{1}{s}} \partial_{zz} \psi )\big|_{(x_0,0)} \geq F(x_0,0)
	\quad \hbox{or}\quad
	\partial_z \psi(x_0,0) \geq f(x_0). 
	\]
Since $\partial_z\phi(x_0,0) < f(x_0)$, we can find $\eta>0$ sufficiently small to guarantee that 
\[
\partial_z\psi(x,0) = \partial_z\phi(x,0) + \eta < f(x_0).
\]
Therefore, it must be that
\[
(a^{ij}(x) \partial_{ij}  \psi+ z^{2 -\frac{1}{s}} \partial_{zz} \psi )\big|_{(x_0,0)} \geq F(x_0,0).
\]
However, if we take $C$ large enough to guarantee that
\[
(a^{ij}(x) \partial_{ij}  \phi+ z^{2 -\frac{1}{s}} \partial_{zz} \phi )\big|_{(x_0,0)} <C + F(x_0,0),
\]
then
\begin{align*}
(a^{ij}(x) \partial_{ij}  \psi+z|^{2 -\frac{1}{s}} \partial_{zz} \psi )\big|_{(x_0,0)}
	&= (a^{ij}(x) \partial_{ij}  \phi+ z^{2 -\frac{1}{s}} \partial_{zz} \phi )\big|_{(x_0,0)} -C <F(x_0,0),
\end{align*}
which is a contradiction. Thus, it must be that $\partial_z \phi(x_0,0) \geq  f(x_0)$, so that \ref{item:visc-ii} holds. 
\end{proof}

\begin{lem}[Characterization 2]\label{lem:characterization2}
Condition \ref{item:visc-ii}
is equivalent to the following.
\begin{enumerate}[start=2,label={(\roman*)''}]
	\item \label{item:visc-ii''} If $(x_0,0) \in T_1$ and $\phi(x,z) = P(x) + az$ touches $U$ from above at $(x_0,0)$
	where $P$ is a polynomial of degree $2$ in $x$ and $a \in \R$, then 	
	\[
	\partial_z \phi(x_0,0) \geq f(x_0). 
	\]
\end{enumerate}
\end{lem}

\begin{proof}
It is clear that \ref{item:visc-ii}  implies \ref{item:visc-ii''} in the $C_s$-class. Conversely, assume that \ref{item:visc-ii''} holds.
Let $\phi \in C_s$ touch $U$ from above at $(x_0,0)$. 
By Lemma \ref{lem:expansion}, 
\begin{equation}\label{eq:eta-expand}
\psi(x,z) =  \phi(x_0,0) +  A \cdot (x-x_0)+\partial_z\phi(x_0,0)z + B |x-x_0|^2 + C z^{1+\eta}
\end{equation}
touches $\phi$, and hence $U$, from above at $(x_0,0)$ in $\overline{S_\tau^+(x_0,0)}$ for $\tau>0$ small. 
For any $\eps>0$, 
\begin{equation}\label{eq:bound-by-linear}
\varepsilon z - C z^{1+\eta} > 0 \quad \hbox{as long as}~0 < z < \left(\frac{\eps}{C}\right)^{1/\eta}.
\end{equation}
Taking $\tau$ smaller if necessary, it follows that $\{z : 0 < h(z) < \tau\} \subset (0, (\eps/C)^{1/\eta})$. Then, in $\overline{S_\tau^+(x_0,0)}$, we have
\[
\psi(x,z) \leq P(x) + a z
\]
where
\[
P(x) =  \phi(x_0,0) +  A \cdot (x-x_0)+B |x-x_0|^2  \quad \hbox{and} \quad a = \partial_z\phi(x_0,0) + \varepsilon.
\]
Since $P(x) +az$ touches $\psi$, and hence $U$, from above at $(x_0,0)$ and \ref{item:visc-ii''} holds, we have that
\[
 \partial_z\phi(x_0,0) + \varepsilon = a=\partial_z(P(x) +az)\big|_{(x_0,0)} \geq f(x_0).
\]
Taking $\varepsilon \to 0$ gives $\partial_z\phi(x_0,0) \geq f(x_0)$, so that \ref{item:visc-ii} holds. 
\end{proof}

As a consequence of the proof of Lemma \ref{lem:characterization2}, we have the following Corollary. 

\begin{cor}\label{lem:if C2 then Cs}
Assume that $\phi \in C^2(\overline{S_1^+})$. 
Given $\varepsilon>0$, there is $\psi \in C_s$ and $\tau>0$ such that 
$\psi$ touches $\phi$ from above at $(x_0,0)$ in $\overline{S_\tau(x_0,0)^+}$ and satisfies
\begin{equation}\label{eq:psi-phi-e}
\partial_z\psi(x_0,0)=\partial_z\phi(x_0,0) + \varepsilon.
\end{equation}
\end{cor}

\begin{proof}
Since $\phi \in C^2(\overline{S_1^+})$, we use the expansion \eqref{eq:eta-expand} with $\eta = 1$ to instead write
\[
\phi(x,z) \leq \phi(x_0,0) +  A \cdot (x-x_0)+\partial_z\phi(x_0,0)z + B |x-x_0|^2 + C z^{2}.
\]
Given $\varepsilon>0$, we apply \eqref{eq:bound-by-linear} with $\eta =1$ to find $\tau>0$ small enough so that in $\overline{S_\tau(x_0,0)^+}$, 
\[
\phi(x,z)
	\leq \phi(x_0,0) + \partial_z\phi(x_0,0) z + A \cdot (x-x_0) + B|x-x_0|^2 +  \varepsilon z =: \psi(x,z). 
\]
Notice that $\psi$ touches $\phi$ from above at $(x_0,0)$  in $S_\tau(x_0,0)^+$ and satisfies \eqref{eq:psi-phi-e}.
\end{proof}

The next lemma validates the expected relationship between classical solutions and $C_s$-viscosity solutions. 

\begin{lem}[Classical solutions and viscosity solutions] \label{lem:classical to visc}
If $U \in C^2(S_1^+) \cap C^1(\overline{S_1^+})$ is a classical subsolution (supersolution) to \eqref{eq:sub}, that is,
$$\begin{cases}
a^{ij}(x) \partial_{ij} U + z^{2 -\frac{1}{s}} \partial_{zz} U \geq (\leq) F & \hbox{in}~S_1^+\\
\partial_{z} U \geq(\leq) f & \hbox{on}~T_1
\end{cases}$$
then $U$ is a $C_s$-viscosity subsolution  (supersolution).

Conversely, if $U \in C^2(S_1^+) \cap C^1(\overline{S_1^+})$ is a $C_s$-viscosity subsolution (supersolution) to \eqref{eq:sub}, then $U$ is a classical subsolution  (supersolution). 
\end{lem}

\begin{proof}
We only present the proof for subsoutions. 
Since the equation is uniformly elliptic in any $S_r(x_0,z_0)\subset \subset  S_1^+$, the result holds in $S_1^+$. 
We only check the Neumann condition. 

It is easy to see that if $U$ is a classical subsolution on $T_1$, then $U$ is a $C_s$-viscosity subsolution on $T_1$. 
Conversely, suppose that $U$ is a smooth, $C_s$-viscosity subsolution on $T_1$.  
Let $(x_0,0) \in T_1$ and $\eps>0$. Since $U \in C^1_z(\overline{S_1^+})$, there is $\tau>0$ such that
\[
U(x_0,z) \leq U(x_0,0) + \partial_zU(x_0,0)z + \eps z \quad \hbox{whenever}~0 < h(z) < \tau.
\]
With this and expanding $U$ as in \eqref{eq:taylorx}, we have that 
\begin{equation*}\label{eq:classical-expansion}
\phi(x,z) = U(x_0,0) +  A \cdot (x-x_0)+\partial_zU(x_0,0)z + B |x-x_0|^2 +\varepsilon z  \in C_s,
\end{equation*}
with $
\|\nabla_x U \|_{L^{\infty}} \leq |A|$ and $\|D^2_xU \|_{L^{\infty}} \leq 2B$, touches $U$ from above at $(x_0,0)$ 
in $\overline{S_\tau(x_0,0)^+}$, for $\tau$ perhaps smaller. 
Using Definition \ref{defn:viscosity-sub}\ref{item:visc-ii} and sending $\eps \to 0$, we get
$\partial_zU(x_0,0) \geq f(x_0)$.  
\end{proof}

The following result is easy to verify.

\begin{lem}\label{lem:comparison}
If $U$ is a $C_s$-viscosity solution to \eqref{eq:sub}  and $V\in C^2(S_1^+) \cap C^1(\overline{S_1^+})$ is a classical solution, then $W=U-V$ is a $C_s$-viscosity solution to \[
\begin{cases}
a^{ij}(x) \partial_{ij} W + z^{2 -\frac{1}{s}} \partial_{zz} W = 0 & \hbox{in}~S_1^+\\
\partial_{z} W = 0 & \hbox{on}~T_1.
\end{cases}
\]
\end{lem}

Lastly, for a positive definite symmetric matrix $M$, recall that the Pucci extremal operators with ellipticity
constants $0<\lambda\leq\Lambda$ are given by
\[
\mathcal{P}^-(M) = \lambda \sum_{e_i>0} e_i + \Lambda \sum_{e_i<0} e_i
\quad \hbox{and} \quad
\mathcal{P}^+(M) = \Lambda \sum_{e_i>0} e_i + \lambda \sum_{e_i<0} e_i
\]
where $e_i$ are the eigenvalues of $M$. 

\begin{rem}\label{rem:pucci}
If $\|F\|_{L^{\infty}(S_1^+)} \leq a$ and $U$ is a $C_s$-viscosity subsolution (supersolution)  to \eqref{eq:sub}, then $U$ is a $C_s$-viscosity subsolution (supersolution) to
\[
\begin{cases}
\mathcal{P}^+(D^2_xU) + z^{2 -\frac{1}{s}} \partial_{zz} U \geq -a \qquad (\mathcal{P}^-(D^2_xU) + z^{2 -\frac{1}{s}} \partial_{zz} U \leq a) & \hbox{in}~S_1^+\\
\partial_{z} U \geq  f  \hspace{1.54in} (\partial_{z} U\leq f)
& \hbox{on}~T_1.
\end{cases}
\]
\end{rem}

\subsection{A stability result}

We now prove that $C_s$-viscosity solutions are closed under uniform limits. Here is one instance in which we use that $z^{2-\frac{1}{s}} \partial_{zz}\phi$ is continuous up to $\{z=0\}$ for $\phi \in C_s$ to overcome the degeneracy of the equation. 

\begin{lem}\label{lem:stability}
Consider sequences $a^{ij}_k:T_1\to \R$ of continuous functions satisfying \eqref{eq:ellipticity}, $f_k \in C(T_1) \cap L^\infty(T_1)$, and $F_k \in C(S_1^+ \cup T_1) \cap L^\infty (S_1^+ \cup T_1) $. 
Let $U_k \in C(\overline{S_1^+})$ be a sequence of $C_s$-viscosity (sub/super)solutions to 
\[
\begin{cases}
a^{ij}_k(x) \partial_{ij}U_k + z^{2- \frac{1}{s}} \partial_{zz}U_k = F_k & \hbox{in}~S_1^+ \\
\partial_z U_k (x,0) = f_k(x) & \hbox{on}~T_1.
\end{cases}
\]
Assume that there are $a^{ij}:T_1 \to \R$ satisfying \eqref{eq:ellipticity}, $f \in C(T_1) \cap L^\infty(T_1)$, $F \in C(S_1^+ \cup T_1) \cap L^\infty (S_1^+ \cup T_1) $ and $U \in C(\overline{ S_1^+})$ such that, as $k \to \infty$,
\begin{itemize}
\item $a^{ij}_k \to a^{ij}$ uniformly on $T_1$ 
\item $f_k \to f$ uniformly on $T_1$ 
\item $F_k \to F$ uniformly on $S_1^+\cup T_1$
\item $U_k \to U$ uniformly on compact subsets $ S_1^+ \cup T_1$.
\end{itemize}
Then, $U$ is a $C_s$-viscosity (sub/super)solution to 
\[
\begin{cases}
a^{ij}(x)\partial_{ij}U + z^{2- \frac{1}{s}} \partial_{zz}U = F & \hbox{in}~S_1^+ \\
\partial_z U (x,0) = f(x) & \hbox{on}~T_1.
\end{cases}
\]
\end{lem}

\begin{proof}
We present only the proof that $U$ is a $C_s$-viscosity subsolution. We only need to check the Neumann condition. 

Suppose, by way of contradiction, that $\partial_z U (x,0) \geq f(x)$ does not hold on $T_1$ in the viscosity sense. 
By Lemma \ref{lem:either-or}, there is a point $(x_0,0) \in T_1$ and a test function $\phi \in C_s$ that touches $U$ from above at $(x_0,0)$ and both
\begin{equation}\label{eq:bwoc}
\partial_z\phi(x_0,0) < f(x_0) \quad \hbox{and} \quad
(a^{ij}(x) \partial_{ij} \phi + z^{2-\frac{1}{s}}\partial_{zz}\phi) \big|_{(x_0,0)} < F(x_0,0)
\end{equation}
hold. We may assume that $\phi$ touches $U$ strictly from above at $(x_0,0)$.
Otherwise, we replace $\phi$ with $\tilde{\phi} = \phi + \varepsilon(|x-x_0|^2 + h(z))$ for $\varepsilon$ small. 

Since $U_k \to U$ uniformly on compact subsets of $S_1^+ \cup T_1$, for $r>0$ small and $k$ sufficiently large, let $\varepsilon_k$ be such that
\[
\overline{S_r(x_0,0)^+} \subset S_1^+ \cup T_1 \quad \hbox{and} \quad
\varepsilon_k := \|U_k - U\|_{L^{\infty}(S_r(x_0,0)^+)}.
\] 
Note that $\varepsilon_k \to 0$ as $k \to \infty$ and that
$U_k \leq \phi + \varepsilon_k$ in $\overline{S_r(x_0,0)^+}$. 
Let $0 < r_k < r$ with $r_k \searrow 0$ and define
\[
d_k = \inf_{S_{r_k}(x_0,0)^+} (\phi + \varepsilon_k - U_k) \geq 0.
\]
Let $(x_k, z_k) \in \overline{S_{r_k}(x_0,0)^+}$ be a point where the previous infimum is attained and note that $(x_k,z_k) \to (x_0,0)$ as $k \to \infty$. 
Set $c_k=\varepsilon_k - d_k$, so that $c_k \to 0$ as $k \to \infty$. 
Since
\[
U_k(x_k,z_k) = \phi(x_k,z_k) + c_k \quad \hbox{and} \quad U_k \leq \phi + c_k \quad \hbox{in}~\overline{S_{r_k}(x_0,0)^+},
\]
we have that $\phi + c_k \in C_s$ touches $U_k$ from above at $(x_k,z_k)$. 
We now use that $U_k$ is a $C_s$-viscosity subsolution to arrive at a contradiction. 
Indeed, if $z_k >0$ for all $k$, we have that
\[
a^{ij}_k(x_k) \partial_{ij}\phi(x_k,z_k) + z_k^{2-\frac{1}{s}} \partial_{zz}\phi(x_k,z_k) \geq F(x_k,z_k).
\]
Sending $k \to \infty$ and using that $\phi \in C_s$, 
\[
(a^{ij}(x) \partial_{ij}\phi+ z^{2-\frac{1}{s}} \partial_{zz}\phi)\big|_{(x_0,0)} \geq F(x_0,0),
\]
contradicting \eqref{eq:bwoc}.
If instead for all $k_0 \in \N$, there is a $k \geq k_0$ such that $z_k = 0$, then, at such points, 
\[
\partial_z\phi(x_k,0) \geq f_k(x_k).
\]
Passing to the limit also contradicts \eqref{eq:bwoc}. 
Therefore, $U$ is a $C_s$-viscosity solution. 
\end{proof}

\section{Analysis of harmonic functions}\label{sec:harmonic}

Here, we show that $C_s$-viscosity solutions to the harmonic equation \eqref{eq:harmonic} are classical. 

\begin{prop}\label{lem:H-classical}
If $H \in C(\overline{S_1 \times S_1^+})$ is a $C_s$-viscosity solution to
\begin{equation}\label{eq:harmonic-2}
\begin{cases}
\Delta_x H + z^{2- \frac{1}{s}} \partial_{zz}H = 0 & \hbox{in}~S_1 \times S_1^+ \\
\partial_z H(x,0) = 0 & \hbox{on}~T_1,
\end{cases}
\end{equation}
then $H$ is a classical solution that satisfies the following estimates. 
\begin{enumerate}[$(1)$]
\item For each integer $k  \geq 0$ and each $S_r(x_0)   \subset S_1 \subset \R^n$, 
\begin{equation}\label{eq:Hx-estimates}
\sup_{S_{r/4}(x_0) \times (S_{c_sr/4}^+ \cup \{0\})} |D_x^k H| \leq \frac{C}{r^{k/2}} \underset{S_r(x_0) \times (S_{c_sr}^+ \cup \{0\})}{\operatorname{osc}} H
\end{equation}
where $C = C(n,k,s)>0$ and $c_s = 1/[2(1-s)]$. 
\item For $z \in S_1^+ $, it holds that
\begin{equation}\label{eq:Hz-decay}
\sup_{x \in S_{1/4}(0)} \abs{\partial_zH(x,z)} \leq C z^{\frac{1}{s}-1}\underset{S_1(0) \times (S_{c_s}^+ \cup \{0\})}{\operatorname{osc}} H
\end{equation}
where $C = C(n,s)>0$. 
\end{enumerate}
If, in addition, we prescribe $H= g$ on $\partial (S_1 \times S_1^+) \cap \{z>0\}$ for a given $g \in C(\overline{S_1 \times S_1^+})$, then
the solution $H$ to \eqref{eq:harmonic-2} is unique.
\end{prop}

The proof is at the end of this section and relies on a new Hopf lemma and regularity estimates for classical solutions to \eqref{eq:harmonic-2}. 

\subsection{Explicit barriers}

For the proof of the Hopf lemma, we first construct explicit barriers in the Monge--Amp\`ere geometry to handle the degeneracy of \eqref{eq:harmonic-2}.

\begin{lem}\label{lem:barrier}
Fix $(x_0,0) \in T_1$. 
Let $z_0$ and $R$ be such that $S_R(x_0,z_0) \subset S_1 \times S_1^+$ and $R=\delta_{\Phi}((x_0,z_0),(x_0,0))$. 
Fix $0 < \rho < R$. 
Then there is a function 
\begin{equation*}\label{eq:phi-class}
\phi \in \begin{cases}
C^2(\overline{S_R(x_0,z_0)}) & \hbox{when}~0 < s \leq 1/2\\
C_s & \hbox{when}~1/2 < s <1
\end{cases}
\end{equation*}
satisfying 
\[
\begin{cases}
\Delta_x \phi + z^{2-\frac{1}{s}} \partial_{zz}\phi >0 & \hbox{in}~S_R(x_0,z_0)  \setminus S_\rho (x_0,z_0) \\
\partial_z \phi(x_0,0) >0 &\\
 \phi(x_0,0) =0. &
\end{cases}
\]
Moreover, $\phi \leq 0$ on $\partial S_R(x_0,z_0)$ and $c\leq \phi \leq C$ on $\partial S_{\rho}(x_0,z_0)$ for some $C,c>0$.  
\end{lem}

\begin{proof}
First note that $z_0>0$ and $R = \delta_h(z_0,0)$. For ease in notation, we let
\[
A:= S_R(x_0,z_0)  \setminus S_\rho(x_0,z_0).
\]
We split into cases based on whether $0 < s\leq 1/2$ or $1/2<s<1$.

\medskip

\noindent
\underline{\bf Case 1}.  Assume that $0 < s \leq 1/2$. 

\smallskip

Begin by considering the function 
\[
\tilde{\phi}(x,z) = e^{-\alpha \delta_{\Phi}((x_0,z_0),(x,z))} \quad \hbox{for}~(x,z) \in A,
\]
where $\alpha>0$ is to be determined. 
For $(x,z) \in A$, we have
\begin{align*}
\Delta_x \tilde{\phi}(x,z) + &z^{2-\frac{1}{s}}\partial_{zz} \tilde{\phi}(x,z)\\
	&= \alpha e^{-\alpha \delta_{\Phi}((x_0,z_0),(x,z))}
		\bigg[
		\alpha \bigg(2 \delta_{\varphi}(x_0,x) + \frac{(h'(z) - h'(z_0))^2}{h''(z) \delta_h(z_0,z)} \delta_h(z_0,z) \bigg)
		- (n+1)
		\bigg].
\end{align*}
It can be checked (see \cite[Lemma 8.1]{Stinga-Vaughan}) that
\begin{equation}\label{eq:quotient}
\mathcal{Q}(z):= 
\frac{(h'(z) - h'(z_0))^2}{h''(z) \delta_h(z_0,z)} \geq 1 \quad \hbox{for}~z>0.
\end{equation}
Therefore, for $(x,z) \in A$, 
\begin{align*}
\Delta_x \tilde{\phi}(x,z) + z^{2-\frac{1}{s}}\partial_{zz} \tilde{\phi}(x,z)
	&\geq \alpha e^{-\alpha \delta_{\Phi}((x_0,z_0),(x,z))}
		\big[\alpha \big( \delta_{\varphi}(x_0,x) +\delta_h(z_0,z) \big)
			- (n+1)\big]\\
	&\geq \alpha e^{-\alpha \delta_{\Phi}((x_0,z_0),(x,z))}
		\big[\alpha \rho- (n+1)\big]
	>0
\end{align*}
by choosing $\alpha = \alpha(\rho,n)>0$ such that $\alpha > (n+1)/\rho$. 
Note also that
\[
\tilde{\phi}(x_0,0) = e^{-\alpha R} \quad \hbox{and} \quad
\partial_z\tilde{\phi}(x_0,0)
	= \alpha h'(z_0) e^{-\alpha R} >0.
\]

The lemma holds with $\phi$ given by 
\[
\phi(x,z) :=\tilde{\phi}(x,z) -\tilde{\phi}(x_0,0)
	=  e^{-\alpha \delta_{\Phi}((x_0,z_0),(x,z))}  -  e^{-\alpha R}.
\]

\noindent
\underline{\bf Case 2}. Assume that $1/2<s<1$. 

\medskip

In this case the function $\mathcal{Q}$ in \eqref{eq:quotient} satisfies $\mathcal{Q}(0) = 0$, so we cannot control the equation for $\tilde{\phi}$ defined in Case 1. 
Nevertheless,  $\mathcal{Q}'(z)>0$ for $z>0$, so we only need to bypass the points near $\{z=0\}$. 
For this, 
let $\varepsilon = \varepsilon (n,s,z_0,R)>0$ small, to be determined,  and 
let $0 < \varepsilon_0<1$ be as in Lemma \ref{lem:A-infty}. 
Define the set 
\[
H_\varepsilon = \bigg\{ z \in S_R(z_0) : 0 < z^{2- \frac{1}{s}} \leq \varepsilon_0 \frac{|S_R(z_0)|}{\mu_h(S_R(z_0))} \bigg\}. 
\]
Since
\[
|H_\eps| = \int_{H_\eps} \, dz \leq \int_{H_\eps} \eps_0 \frac{|S_R(z_0)|}{\mu_h(S_R(z_0))} z^{\frac{1}{s}-2} \, dz \leq \eps_0 |S_{R}(z_0)|,
\]
we can apply Lemma \ref{lem:A-infty} to get
$\mu_h(H_\varepsilon) \leq \varepsilon \mu_h(S_R(z_0))$. 
Let $\tilde{H}_\varepsilon$ be an open interval satisfying
\[
H_\varepsilon \subset \tilde{H}_\varepsilon \subset S_R(z_0), \quad \mu_h(\tilde{H}_\varepsilon \setminus H_\varepsilon) \leq \varepsilon \mu_h(S_R(z_0))
\]
and $\psi_\varepsilon(z)$ be a smooth function satisfying
\[
\psi_\varepsilon = 1~\hbox{in}~H_\varepsilon, \quad
	\psi_\varepsilon = \varepsilon~\hbox{in}~S_R(z_0) \setminus \tilde{H}_\varepsilon, \quad
	\varepsilon \leq \psi_\varepsilon \leq 1~\hbox{in}~S_R(z_0).
\]
One can check, as in the proof of \cite[Lemma 8.2]{Stinga-Vaughan}, that
\begin{equation}\label{eq:psi-integral}
\int_{S_R(z_0)} \psi_{\varepsilon} \, d\mu_h
	\leq 3 \varepsilon \mu_h(S_R(z_0)). 
\end{equation}

Let $h_{\varepsilon}(z)$ be the strictly convex solution to 
\[
\begin{cases}
h_{\varepsilon}'' = 2(n+1) \psi_\varepsilon h'' & \hbox{in}~S_R(z_0) \\
h_{\varepsilon} = 0 & \hbox{on}~\partial S_R(z_0). 
\end{cases}
\]
We remark that $h_{\varepsilon}\in C^{\infty}(S_R(z_0))$, and since $h \in C^1(\R)$, we have $h_{\varepsilon} \in C^1(\overline{S_R(z_0)})$.
Since $h_{\varepsilon}$ is strictly convex  in $S_R(z_0)$ and zero at the endpoints, $h_\eps<0$ in $S_R(z_0)$ and $h_{\varepsilon}$ attains a unique minimum at some $z_m \in S_R(z_0)$. In particular, $h_\eps'(z_m) = 0$. 

For any $z \in S_R(z_0)$, we use the equation for $h_\eps$ and \eqref{eq:psi-integral} to estimate
\begin{align*}
|h_{\varepsilon}'(z)|
	&= \bigg| \int_{z_m}^z h_{\varepsilon}''(w) \, dw \bigg|\\
	&= 2(n+1)\bigg| \int_{z_m}^z  \psi_\varepsilon(w) h''(w)\, dw \bigg|\\
	&\leq 2(n+1) \int_{S_R(z_0)} \psi_\varepsilon \, d\mu_h\\
	&\leq 6(n+1)\varepsilon \mu_h(S_R(z_0)) = C_1 \varepsilon  \mu_h(S_R(z_0)).
\end{align*}
Since $h_{\varepsilon} \in C^1(\overline{S_R(z_0)})$, we can further deduce that
\begin{equation}\label{eq:he-deriv-0}
- h_{\varepsilon}'(0) = |h_{\varepsilon}'(0)|
	= \lim_{z \to 0} |h_{\varepsilon}'(z)| \leq 
	C_1 \varepsilon  \mu_h(S_R(z_0)).
\end{equation}
With this, we can show, as in the proof of \cite[Lemma 8.2]{Stinga-Vaughan}, that there is $C_2 = C_2(n,s)>0$ such that, for any $z \in S_R(z_0)$,
\begin{equation}\label{eq:he-estimate}
- h_\eps(z) = |h_{\varepsilon}(z)| \leq C_2 \varepsilon R. 
\end{equation}
Using the estimates on $h_\eps$ and $h_\eps'$ given above, we can follow the steps in the proof of \cite[Lemma 8.2]{Stinga-Vaughan} to show that, for small $\varepsilon = \varepsilon(\rho, n, s)>0$, there is $C_4 = C_4(\rho,R,n,s)>0$ such that
\begin{equation}\label{eq:below}
(h'(z) - h'(z_0) - h_\varepsilon'(z))^2 \geq C_4[\mu_h(S_R(z_0))]^2 \quad \hbox{when}~\rho/2 \leq \delta_h(z_0,z) < R
\end{equation}
and
\[
\rho \leq \delta_{\Phi}((x_0,z_0), (x,z)) - h_\varepsilon(z) < (1+C_2\varepsilon) R \quad \hbox{when}~(x,z) \in A. 
\]

We are now ready to proceed with the construction of the barrier. 
Define $\tilde{\phi}(x,z)$ by
\[
\tilde{\phi}(x,z)
	= e^{-\alpha [ \delta_\Phi((x_0,z_0),(x,z))- h_\varepsilon(z) ]}.
\]
For $(x,z) \in A$, we have
\begin{align*}
\Delta_x &\tilde{\phi}(x,z) + z^{2-\frac{1}{s}}\partial_{zz} \tilde{\phi}(x,z)\\
	&=\alpha e^{-\alpha [ \delta_\Phi((x_0,z_0),(x,z))- h_\varepsilon(z) ]}\\
	&\quad \times\bigg[
	\alpha \big(2\delta_\varphi(x_0,x) + z^{2-\frac{1}{s}} (h'(z) - h'(z_0) - h_\varepsilon'(z))^2 \big)
	- (n+1)\big(1-2\psi_\varepsilon(z) \big)
	\bigg]
\end{align*}
where we have used the equation for $h_\varepsilon$. 

Suppose now that $z \in H_{\varepsilon}$. Using that $\psi_\varepsilon(z) = 1$, we have
\begin{align*}
\Delta_x &\tilde{\phi}(x,z) + z^{2-\frac{1}{s}}\partial_{zz} \tilde{\phi}(x,z)\\
	&= \alpha e^{-\alpha [ \delta_\Phi((x_0,z_0),(x,z))- h_\varepsilon(z) ]} \bigg[
	\alpha \big(2\delta_\varphi(x_0,x) + z^{2-\frac{1}{s}} (h'(z) - h'(z_0) - h_\varepsilon'(z))^2 \big)
	+ (n+1) \big)\bigg]\\
	&\geq  \alpha e^{-\alpha [ \delta_\Phi((x_0,z_0),(x,z))- h_\varepsilon(z) ]} (n+1)>0. 
\end{align*}
On the other hand, suppose that $z \notin H_\varepsilon$, so that $z^{2-\frac{1}{s}} >\varepsilon_0 |S_R(z_0)| / \mu_h(S_R(z_0))$. Since $\psi_\varepsilon(z) >0$, we have
\begin{align*}
\Delta_x &\tilde{\phi}(x,z) + z^{2-\frac{1}{s}}\partial_{zz} \tilde{\phi}(x,z)\\
	&\geq \alpha e^{-\alpha [ \delta_\Phi((x_0,z_0),(x,z))- h_\varepsilon(z) ]}\\
	&\quad \times\bigg[
	\alpha \bigg(2\delta_\varphi(x_0,x) + \varepsilon_0 \frac{|S_R(z_0)|}{\mu_h(S_R(z_0))} (h'(z) - h'(z_0) - h_\varepsilon'(z))^2 \bigg)
	- (n+1) \big)
	\bigg].
\end{align*}
Since $\delta_\Phi((x_0,z_0),(x,z)) \geq \rho$, it must be that either $\delta_\varphi(x_0,x) \geq \rho/2$ or $\delta_h(z_0,z) \geq \rho/2$. 
Suppose first that $\delta_\varphi(x_0,x) \geq \rho/2$. Then
\[
2\delta_\varphi(x_0,x) + \varepsilon_0 \frac{|S_R(z_0)|}{\mu_h(S_R(z_0))} (h'(z) - h'(z_0) - h_\varepsilon'(z))^2 
	\geq 2\delta_\varphi(x_0,x) \geq \rho.
\]
Choosing $\alpha = \alpha(\rho,n)$ such that $\alpha > (n+1)/\rho$
gives
\begin{align*}
\Delta_x \tilde{\phi}(x,z) + z^{2-\frac{1}{s}}\partial_{zz} \tilde{\phi}(x,z)
	\geq \alpha e^{-\alpha [ \delta_\Phi((x_0,z_0),(x,z))- h_\varepsilon(z) ]} \big[
	\alpha \rho
	- (n+1) \big)
	\big] >0.
\end{align*}
Now suppose that $\delta_h(z_0,z) \geq \rho/2$. Then, by \eqref{eq:below} and Lemma \ref{lem:doubling}, we have
\begin{align*}
2\delta_\varphi(x_0,x) + \varepsilon_0 \frac{|S_R(z_0)|}{\mu_h(S_R(z_0))} (h'(z) - h'(z_0) - h_\varepsilon'(z))^2 
	&\geq \varepsilon_0 \frac{|S_R(z_0)|}{\mu_h(S_R(z_0))} (h'(z) - h'(z_0) - h_\varepsilon'(z))^2 \\
	&\geq  C_4 \varepsilon_0 \frac{|S_R(z_0)|}{\mu_h(S_R(z_0))}   [\mu_h(S_R(z_0))]^2\\
	&= C_4 \varepsilon_0 |S_R(z_0)|  \mu_h(S_R(z_0)) \\
	&\geq C_5 \varepsilon_0 R
\end{align*}
for $C_5 = C_5(\rho, R, n , s)>0$. 
Choose $\alpha = \alpha (n,s, z_0,\rho, R)>0$ larger to guarantee that $\alpha C_5 \varepsilon_0 R \geq n+1$. 
Then
\begin{align*}
\Delta_x  \tilde{\phi}(x,z) + z^{2-\frac{1}{s}}\partial_{zz} \tilde{\phi}(x,z)
	&\geq \alpha e^{-\alpha [ \delta_\Phi((x_0,z_0),(x,z))- h_\varepsilon(z) ]} \big[
	\alpha C_5 \varepsilon_0 R	
	- (n+1) \big)
	\big]>0.
\end{align*}
In summary,
\[
\Delta_x  \tilde{\phi}(x,z) + z^{2-\frac{1}{s}}\partial_{zz} \tilde{\phi}(x,z)>0 \quad \hbox{for all}~(x,z) \in A. 
\]

We claim that  the lemma holds with $\phi(x,z)$ given by 
\begin{align*}
\phi(x,z) 
	&:=\tilde{\phi}(x,z) -\tilde{\phi}(x_0,0)
	= e^{-\alpha [\delta_{\Phi}((x_0,z_0),(x,z))- h_\varepsilon(z)]}  -  e^{-\alpha R}. 
\end{align*}
Indeed, since $\tilde{\phi}$ satisfies the equation in $A$, so does $\phi$. 
At $(x_0,0)$, we have
\[
\phi(x_0,0) = \tilde{\phi}(x_0,0) -\tilde{\phi}(x_0,0) = 0,
\]
and, by \eqref{eq:he-deriv-0},
\begin{align*}
\partial_z\phi(x_0,0)
	&=  \alpha e^{-\alpha [R- h_\varepsilon(0)]} \big[h'(z_0) + h_\varepsilon'(0)\big] \\
	&\geq  \alpha e^{-\alpha R} \big[h'(z_0) -C_1 \varepsilon  \mu_h(S_R(z_0)) \big]
	>0
\end{align*}
for $\varepsilon = \varepsilon (n,s, z_0, R)>0$ sufficiently small. 
On $\partial S_R(x_0,z_0)$, we use that $h_\eps \leq 0$ to get
\[
\phi(x,z)
	= e^{-\alpha [R- h_\varepsilon(z)]}  -  e^{-\alpha R} 
	= e^{-\alpha R} \big[ e^{\alpha h_\varepsilon(z)} -1 \big] \leq 0.
\]
On $\partial S_\rho (x_0,z_0)$, 
we again use that $h_\eps \leq 0$ to obtain
\[
\phi(x,z) \leq  e^{-\alpha \rho}  -  e^{-\alpha R} =: C,
\]
and then apply \eqref{eq:he-estimate} to find 
\begin{align*}
\phi(x,z)
	&= e^{-\alpha [\rho - h_\varepsilon(z)]}  -  e^{-\alpha R}
	\geq e^{-\alpha [\rho +  C_2\varepsilon R]}  -  e^{-\alpha R} =: c>0
\end{align*}
when $\eps>0$ is small enough to guarantee that $\rho +  C_2\varepsilon R < R$.
We conclude that $0 < c \leq \phi \leq C$ on $\partial S_\rho(x_0,z_0)$.
Lastly, since
\begin{align*}
z^{2-\frac{1}{s}}& \partial_{zz} \phi(x,z)\\
	&= \alpha e^{-\alpha [ \delta_\Phi((x_0,z_0),(x,z))- h_\varepsilon(z) ]} 
	\big[\alpha z^{2-\frac{1}{s}} (h'(z) - h'(z_0) - h_\varepsilon'(z))^2 
	- (1  - 2(n+1)) \psi_\varepsilon(z)\big]
\end{align*}
is continuous in $\overline{S_R(x_0,z_0)}$, we have that $\phi \in C_s$. 
\end{proof}

\begin{rem}\label{rem:barrier-aij}
By using ellipticity, the proof of Lemma \ref{lem:barrier} can be readily modified to prove the existence of barriers $\phi$
satisfying $a^{ij}(x) \partial_{ij} \phi + z^{2 - \frac{1}{s}} \partial_{zz} \phi >0$ for bounded,
measurable coefficients $a^{ij}(x)$ satisfying \eqref{eq:ellipticity}. 
\end{rem}

\subsection{A Hopf lemma}

Our following Hopf lemma states that harmonic functions attain their extrema on the curved boundary. 

\begin{lem}\label{lem:Hopf}
If $H \in C(\overline{S_1 \times S_1^+})$ is a $C_s$-viscosity solution to
\[
\begin{cases}
\Delta_x H + z^{2- \frac{1}{s}} \partial_{zz}H = 0 & \hbox{in}~S_1 \times S_1^+ \\
\partial_z H (x,0) = 0 & \hbox{on}~T_1,
\end{cases}
\] 
then $H$ attains its maximum and minimum on $\partial (S_1 \times S_1^+) \cap \{z>0\}$. 
\end{lem}

\begin{proof}
We present only the proof for the maximum. Assume that $H$ is not constant, otherwise there is nothing to show.
By interior Schauder estimates, $H$ is a classical solution in the interior of $S_1 \times S_1^+$. 
Moreover, by the weak maximum principle \cite[Theorem 3.1]{GT}, $H$ attains its maximum on the boundary $\partial (S_1 \times S_1^+)$. 
Suppose, by way of contradiction, that $H$ attains its maximum at a point $(x_0,0) \in T_1$. 
We may assume that 
\[
H(x_0,0) >  H(x,z) \quad \hbox{for all}~(x,z) \in S_1 \times S_1^+. 
\]
Indeed, set $M = H(x_0,0)$ and suppose that there is a point  $(x_1,z_1) \in S_1 \times S_1^+$ with $H(x_1,z_1) = M$. 
By the strong maximum principle, $H \equiv M$ in $(S_1 \times S_1^+) \cap \{z>\frac{z_1}{2}\}$. 
In particular, $H(x_1, \frac{z_1}{2}) = M$. Iterating this process, we find that $H \equiv M$ in $(S_1 \times S_1^+) \cap \{z> \frac{z_1}{2^k}\}$ for all $k \in \N$. Consequently, $H \equiv M$ in $S_1 \times S_1^+$, a contradiction
to the assumption that $H$ is not constant.

Let $z_0$ and $R$ be such that $S_R(x_0,z_0) \subset S_1 \times S_1^+$ and $R=\delta_{\Phi}((x_0,z_0),(x_0,0))$. 
Fix $0 < \rho < R$. 
In this setting, consider the barrier $\phi$ constructed in Lemma \ref{lem:barrier}. 

For $\varepsilon>0$, to be determined, define
\[
\psi(x,z) =  H(x_0,0) -\varepsilon \phi(x,z) \quad \hbox{on}~S_R(x_0,z_0) \setminus  S_\rho(x_0,z_0).
\]
By Lemma \ref{lem:barrier}, 
\[
\begin{cases}
\Delta_x \psi + z^{2-\frac{1}{s}} \partial_{zz} \psi <0 & \hbox{in}~S_R(x_0,z_0) \setminus  S_\rho(x_0,z_0)\\
\partial_z \psi(x_0,0) < 0 \\
\psi(x_0,0) = H(x_0,0). 
\end{cases}
\]
On $\partial S_R(x_0,z_0) \setminus \{(x_0,0)\}$, we have
\[
\psi \geq H(x_0,0)  - \varepsilon \cdot 0  > H.
\]
On $\partial S_\rho(x_0,z_0)$, we use that $H< H(x_0,z_0)$ and that $c \leq \phi \leq C$ to find $\varepsilon>0$ such that
\[
H - \psi  = H - H(x_0,0) + \varepsilon \phi \leq 0. 
\]
Consequently, we have that $\psi \geq H$ on $\partial [S_R(x_0,z_0) \setminus S_\rho(x_0,z_0)]$. 
By the weak maximum principle, $\psi \geq H$ in $S_R(x_0,z_0) \setminus S_\rho(x_0,z_0)$. 
Since $\psi(x_0,0) = H(x_0,0)$, we have that $\psi$ touches $H$ from above at $(x_0,0)$. 
Moreover, by Lemma \ref{lem:barrier},
\begin{equation}\label{eq:neumann-frombarrier}
\partial_z \psi (x_0,0)  = -\varepsilon \partial_z\phi(x_0,0)<0.
\end{equation}

If $1/2 < s < 1$, we know that $\psi \in C_s$. Since $H$ is a $C_s$-viscosity subsolution,
\[
\partial_z \psi (x_0,0) \geq 0,
\]
contradicting \eqref{eq:neumann-frombarrier}. If $0 < s \leq 1/2$, we have $\psi \in C^2$. 
Nevertheless, by Corollary \ref{lem:if C2 then Cs},
for all $\delta>0$, there is a function $\psi_{\delta} \in C_s$ that touches $\psi$, and hence $H$, from above at $(x_0,0)$, and satisfies
\[
\partial_z \psi_{\delta}(x_0,0) = -\varepsilon \partial_z \phi (x_0,0)+ \delta < 0
\]
for $\delta>0$ sufficiently small. 
However, since $H$ is a $C_s$-viscosity subsolution, we have that
\[
\partial_z \psi_\delta (x_0,0) \geq 0,
\]
a contradiction. Hence, $U$ attains its maximum on $\partial (S_1 \times S_1^+) \cap \{z>0\}$. 
\end{proof}

\begin{rem}
In view of Remark \ref{rem:barrier-aij}, we also have that the conclusion of Lemma \ref{lem:Hopf}
remains valid for $C_s$-viscosity solutions to $a^{ij}(x) \partial_{ij} U + z^{2- \frac{1}{s}} \partial_{zz}U=0$ where $a^{ij}(x)$ are bounded, measurable coefficients satisfying \eqref{eq:ellipticity}. 
\end{rem}

\subsection{Regularity estimates for classical solutions}

We now rewrite the known regularity estimates for classical solutions in our setting.

\begin{prop}[Proposition 3.5 in \cite{Caffarelli-Stinga}]\label{prop:H regularity}
If $H \in C^2(S_1 \times S_1^+) \cap C^1(S_1 \times \overline{S_1^+})$ is a classical solution to 
\begin{equation*}\label{eq:harmonic-eqn}
\begin{cases}
\Delta_xH + z^{2-\frac{1}{s}} \partial_{zz} H = 0 & \hbox{in}~S_1 \times S_1^+\\
\partial_{z}H(x,0) = 0 & \hbox{on}~T_1,
\end{cases}
\end{equation*}
then $H$ satisfies \eqref{eq:Hx-estimates} and \eqref{eq:Hz-decay}. 
\end{prop}

\begin{proof}
Consider the change of variables $z \mapsto (y/2s)^{2s}$. 
Notice that
\[
h(z) 
	= \frac{s^2}{1-s} \abs{z}^{\frac{1}{s}} 
	= \frac{s^2}{1-s}\(\frac{y}{2s}\)^2
	=\frac{c_s}{2} y^2 \quad \hbox{where}~c_s = \frac{1}{2(1-s)}. 
\]
Therefore, for any $r>0$, 
\begin{equation}\label{eq:sections-zy}
z \in S_{c_sr}^+ \quad \hbox{if and only if} \quad y \in B_{\sqrt{2r}}^+ = (0, \sqrt{2r}). 
\end{equation} 
In $x$, recall from \eqref{eq:x-MA-euclidean} that $S_1 = B_{\sqrt{2}} \subset \R^n$. 

Define the function $W(x,y) := H(x, (y/2s)^{2s})$. 
One can check that $W$ is a classical solution to the divergence form equation
\[
\begin{cases}
\Delta_xW + \frac{1-2s}{y} \partial_yW + \partial_{yy}W = \operatorname{div}_{x,y}(y^{1-2s} \nabla W) = 0 & \hbox{in}~B_{\sqrt{2}} \times B_{\sqrt{2/c_s}}^+\\
y^{1-2s} \partial_{y}W(x,y)\big|_{y=0} = 0 & \hbox{on}~B_{\sqrt{2}} \times \{y=0\}.
\end{cases}
\]
Consider a section $S_r(x_0) \subset S_1$. Recalling \eqref{eq:x-MA-euclidean}, we apply
 \cite[Proposition 3.5(1)]{Caffarelli-Stinga}, for each $k \geq 0$, to obtain
\begin{align*}
\sup_{S_{r/4}(x_0) \times [0, \sqrt{2r}/2)} |D^k_xW|
	&= \sup_{B_{\sqrt{2r}/2}(x_0) \times [0, \sqrt{2r}/2)} |D^k_xW| \\
	&\leq \frac{C}{r^{k/2}} \underset{B_{\sqrt{2r}}(x_0) \times [0, \sqrt{2r})}{\operatorname{osc}}  W
	= \frac{C}{r^{k/2}} \underset{S_r(x_0) \times [0, \sqrt{2r}) }{\operatorname{osc}}  W.
\end{align*}
Since $D^k_xH(x,z) = D^k_x W(x, y)$, we use \eqref{eq:sections-zy} and change variables to find
\begin{align*}
\sup_{S_{r/4}(x_0) \times (
S_{c_sr/4} \cup \{0\})} \abs{D^k_xH(x,z)}
	&=  \sup_{S_{r/4}(x_0) \times [0, \sqrt{2r}/2)} \abs{D^k_xW(x,y)} \\
	&\leq \frac{C}{r^{k/2}} \underset{S_r(x_0) \times[0, \sqrt{2r}) }{\operatorname{osc}} W(x,y)
	= \frac{C}{r^{k/2}}\underset{S_r(x_0) \times (S_{c_sr} \cup \{0\})}{\operatorname{osc}} H(x,z),
\end{align*}
which proves \eqref{eq:Hx-estimates}.

Similarly, by \cite[Proposition 3.5(3)]{Caffarelli-Stinga}, if $y \in [0, \sqrt{2})$, then
\[
\sup_{B_{\sqrt{2}/2}(0)} \abs{W_y(x,y)} \leq C y\underset{B_{\sqrt{2}}(0) \times [0,\sqrt{2})}{\operatorname{osc}} W.
\]
Let $z \in S_{c_s}^+ \cup \{0\}$. 
Since $\partial_zH(x,z)=y^{1-2s}\partial_yW(x,y)$, we use \eqref{eq:x-MA-euclidean} and \eqref{eq:sections-zy} to change variables and obtain
\begin{align*}
\sup_{x \in S_{1/4}(0)} \abs{\partial_zH(x,z)}
	&=\sup_{x \in B_{\sqrt{2}/2}(0)}y^{1-2s}  \abs{\partial_yW(x,y)}\\
	&\leq Cy^{2(1-s)}\underset{B_{\sqrt{2}}(x_0) \times [0,\sqrt{2})}{\operatorname{osc}} W
	= C z^{\frac{1}{s}-1} \underset{S_1(0) \times (S_{c_s}^+ \cup \{0\})}{\operatorname{osc}} H,
\end{align*}
which proves \eqref{eq:Hz-decay}. 
\end{proof}

\subsection{Proof of Proposition \ref{lem:H-classical}}

Let $\tilde{H}(x,z) = H(x,|z|)$ denote the even reflection of $H$ across $T_1$ and continue to denote $\tilde{H}$ by $H$. 
The reflected function $H$ is a $C_s$-viscosity solution to 
\begin{equation}\label{eq:reflected}
\begin{cases}
\Delta_x H + |z|^{2- \frac{1}{s}} \partial_{zz}H = 0 & \hbox{in}~(S_1 \times S_1) \setminus \{z=0\}\\
\partial_{z} H (x,0) = 0 & \hbox{on}~T_1.
\end{cases}
\end{equation}
By \cite[Lemma 4.2]{Caffarelli-Silvestre}, 
there is a unique classical solution $V$ to \eqref{eq:reflected} satisfying $V = H$ on $\partial (S_1 \times S_1)$. 
By Proposition \ref{prop:H regularity}, $V$ satisfies the regularity estimates \eqref{eq:Hx-estimates} and \eqref{eq:Hz-decay}. 
To prove the statement, it is enough to show that $H=V$. 
By Lemma \ref{lem:comparison}, $W = H-V$ is a $C_s$-viscosity solution to
\[
\begin{cases}
\Delta_x W + z^{2- \frac{1}{s}} \partial_{zz}W = 0 & \hbox{in}~S_1 \times S_1^+\\
\partial_{z} W (x,0) = 0 & \hbox{on}~T_1\\
W = 0 &\hbox{on}~\partial (S_1 \times S_1^+) \cap \{z>0\}. 
\end{cases}
\]
By Lemma \ref{lem:Hopf}, $W$ attains both its maximum and minimum on $\partial (S_1 \times S_1^+) \cap \{z>0\}$. 
Consequently, $W \equiv 0$, and we have that $H=V$, as desired.

If, in addition, one prescribes $H= g$ on $\partial (S_1 \times S_1^+) \cap \{z>0\}$, then the solution $H$ is unique by \cite[Lemma 4.2]{Caffarelli-Silvestre}. \qed

\section{Harnack inequality and H\"older regularity for viscosity solutions \\ to the extension problem}\label{sec:harnack}

In this section we prove Theorem \ref{thm:F-harnack}. 
By rescaling, it is enough to prove the following normalized version
(see, for example, \cite[Section 5]{Stinga-Vaughan}, where the main difference here is the additional normalization $\|F\|_{L^\infty}\leq a$ for the right hand side $F$). 

\begin{thm}\label{thm:F-harnack3}
Fix $a>0$. 
Let $\Omega \subset \R^n$ be a bounded domain, 
$a^{ij}(x): \Omega \to \R$ be bounded, measurable and satisfy \eqref{eq:ellipticity}.
There exist positive constants $C_H = C_H(n,\lambda,\Lambda,s)>1$, $\kappa = \kappa(n,s) < 1$,  and $K_0 = K_0(n,s)$ such that, for every cube $Q_R = Q_R(\tilde{x}, \tilde{z})$ such that $Q_{R}  \subset\subset \Omega \times \R$,
 every nonpositive $f \in L^{\infty}(Q_{R}\cap \{z=0\})$, every $F \in L^{\infty}(Q_{R})$,
  and every nonnegative $C_s$-viscosity solution $U$, symmetric across $\{z=0\}$, to 
 \[
 \begin{cases}
 a^{ij}(x) \partial_{ij}U + |z|^{2-\frac{1}{s}} \partial_{zz}U = F & \hbox{in}~Q_{R}  \cap \{z\not=0\} \\
\partial_{z} U(x,0) = f & \hbox{on}~Q_{R}  \cap \{z=0\},
 \end{cases}
 \]
 if
 \[
 U(\tilde{x},\tilde{z}) \leq \frac{aR}{2K_0},
\quad 
\|f\|_{L^{\infty}(Q_{R}(\tilde{x},\tilde{z}) \cap \{z=0\})}  \leq a \mu_h(S_{R}(\tilde{z})), \quad 
\|F\|_{L^{\infty}(Q_{R})} \leq a, 
 \]
then
 \[
U \leq C_H aR \quad \hbox{in}~Q_{\kappa R}(x_0,z_0). 
 \]
\end{thm}

The proof of Theorem \ref{thm:F-harnack3} is at the end of the section. First, we review the notion of Monge--Amp\`ere paraboloids and prove the point-to-measure estimate for $C_s$-viscosity solutions. 

\subsection{Paraboloids associated to $\Phi$}\label{sec:polynomials}

Let us briefly review the definition and present some basic properties of Monge--Amp\`ere paraboloids associated to $\Phi$. 

A Monge--Amp\`ere paraboloid $P$ of opening $a>0$ (associated to $\Phi$) in $\R^{n+1}$ is a function of the form
\[
P(x,z) = -a\Phi(x,z) + \ell(x,z), \quad (x,z) \in \R^{n+1}
\]
where $\ell(x,z)$ is an affine function in $(x,z)$. If $(x_v,z_v)$ is the unique solution to $\nabla P(x_v,z_v) = 0$, we say that $(x_v,z_v)$ is the vertex of $P$, and we can write $P$ as 
\[
P(x,z) = -a \delta_{\Phi}((x_v,z_v),(x,z)) + c 
\]
for some constant $c \in \R$. 
Moreover, if $P$ coincides with a continuous function $U: \R^{n+1} \to \R$ at a point $(x_0,z_0)$, then we can further write
\[
P(x,z) = -a \delta_{\Phi}((x_v,z_v),(x,z)) +a \delta_{\Phi}((x_v,z_v),(x_0,z_0)) + U(x_0,z_0).  
\]
See \cite[Section 6]{Stinga-Vaughan} for these and more properties. 

Our next result relates touching Monge--Amp\`ere paraboloids to classical ones when $z_0\not=0$. 

\begin{lem}\label{eq:MA-classical}
Let $U:\R^{n+1} \to \R$ be a continuous function and let $(x_0,z_0) \in \R^{n+1}$ with $z_0 \not=0$. 
If there is a Monge--Amp\`ere paraboloid $P$ opening $a>0$ that touches $U$ from below at $(x_0,z_0)$, 
then $U$ can be touched from below by a classical paraboloid $P_c$ at $(x_0,z_0)$ such that $D^2P_c(x_0,z_0) = D^2P(x_0,z_0)$. 
\end{lem}

\begin{proof}
Assume, without loss of generality, that $z_0>0$. 
Begin by writing $P$ as
\begin{align*}
P(x,z) 
	&= -a \delta_{\Phi}((x_v,z_v),(x,z)) + \delta_{\Phi}((x_v,z_v),(x_0,z_0)) + U(x_0,z_0)\\
	&= -\frac{a}{2} \left( |x-x_v|^2- |x_0-x_v|^2\right)
		-a \left( \delta_h(z_v,z) - \delta_h(z_v,z_0)\right)
		+U(x_0,z_0).
\end{align*}
Expanding the $z$-component, we find that
\begin{align*}
\delta_h(z_v,z) &- \delta_h(z_v,z_0)\\
	&= \big(h(z) - h(z_v) - h'(z_v)(z-z_v)\big)-\big(h(z_0) - h(z_v) - h'(z_v)(z_0-z_v)\big)\\
	&= \big[h(z)  -h(z_0) \big]- h'(z_v)(z-z_0)\\
	&= \big[h'(z_0)(z-z_0) + \frac{1}{2} h''(z_0)(z-z_0)^2 + \frac{1}{6} h'''(\xi)(z-z_0)^3\big]- h'(z_v)(z-z_0)
\end{align*}
for some $\xi$ between $z$ and $z_0$. Since $z_0>0$, there is a neighborhood $S_\tau(z_0) \subset \subset \R^+$ in which we can bound $|h'''(\xi)| \leq C$ uniformly in $S_\tau(z_0)$. Consequently, 
\begin{align*}
P_c(x,z)
	&= -\frac{a}{2} \left( |x-x_v|^2- |x_0-x_v|^2\right)\\
	&\quad-a \left((h'(z_0) - h'(z_v))(z-z_0) +\frac{1}{2} h''(z_0)(z-z_0)^2 + \frac{1}{6} C(z-z_0)^3\right)
		+U(x_0,z_0)
\end{align*}
 is a classical paraboloid that touches $P$, and hence $U$, from below at $(x_0,z_0)$. It is clear that $D^2P_c(x_0,z_0) = D^2P(x_0,z_0)$.  
\end{proof}

Recalling Definition \ref{defn:MA-poly}, note that Monge--Amp\`ere paraboloids are second-order Monge--Amp\`ere polynomials, but not vice versa. 
We will need the following result on polynomials. 

\begin{lem}\label{lem:polynomial-approx}
Let $U:\R^{n+1} \to \R$ be a continuous function and let $(x_0,z_0) \in \R^{n+1}$ with $z_0 \not=0$. 
Suppose that $U$ can be approximated by a classical second-order polynomial
\begin{align*}
U(x,z) = P_c(x,z) + o(|(x,z) -(x_0,z_0)|^2) \quad \hbox{near}~(x_0,z_0)
\end{align*}
where
\[
P_c(x,z) 
	= \frac{1}{2} \langle M((x,z) - (x_0,z_0)), (x,z) - (x_0,z_0) \rangle + \langle p ,(x,z) - (x_0,z_0) \rangle + U(x_0,z_0),
\]
$M$ is a symmetric matrix of size $(n+1)\times(n+1)$ and $p\in\R^{n+1}$.
Set $M_n^{ij} = M^{ij}$, $m = M^{n+1,n+1}$, and $b^i = (M^{i,n+1} + M^{n+1,i})/2$ for $1 \leq i,j \leq n$.
Then $U$ can be approximated by a second-order Monge--Amp\`ere polynomial
\begin{align*}
U(x,z) = P(x,z) + o(\delta_{\Phi}((x_0,z_0),(x,z))) \quad \hbox{near}~(x_0,z_0)
\end{align*}
where
\begin{align*}
P(x,z) 
	&= \frac{1}{2} \langle M_n(x-x_0), (x-x_0) \rangle  + m |z_0|^{2-\frac{1}{s}} \delta_h(z_0,z)  \\
	&\quad + \langle b,(x-x_0) \rangle (z-z_0)+  \langle p, (x,z) -(x_0,z_0) \rangle + U(x_0,z_0).
\end{align*}
Consequently, for all $\eps>0$, the second-order Monge--Amp\`ere polynomial
\begin{align*}
P_\eps(x,z) 
= P(x,z) - \eps \delta_{\Phi}((x_0,z_0)
\end{align*}
touches $U$ from below at $(x_0,z_0)$ in a neighborhood of $(x_0,z_0)$.
\end{lem}

\begin{proof}
Begin by writing
\[
M = \begin{pmatrix}
M_{n} & b_1 \\
b_2^T & m
\end{pmatrix} 
\]
for $m \in \R$, $b_1,b_2 \in \R^n$, and $M_{n} \in \R^n \times \R^n$. 
Letting $b = (b_1+b_2)/2$, we note that
\[
\langle M((x,z) - (x_0,z_0)),(x,z) - (x_0,z_0) \rangle
	= \langle M_{n}(x-x_0), (x-x_0) \rangle + 2\langle b,(x-x_0) \rangle (z-z_0) + m |z-z_0|^2. 
\]
Consider the quadratic term in $z$. We expand $h$ around $z_0$ to obtain
\[
h(z) = h(z_0) + h'(z_0)(z-z_0) + \frac{1}{2} h''(z_0)(z-z_0)^2 + o(|z-z_0|^2)
\]
which gives
\[
 \frac{1}{2} (z-z_0)^2  = |z_0|^{2-\frac{1}{s}} \delta_h(z_0,z) + o(|z-z_0|^2).
\]
With this, we write $P_c$ as
\begin{align*}
P_c(x,z)
	&=  \frac{1}{2} \langle M_{n}(x-x_0), (x-x_0) \rangle  + m |z_0|^{2-\frac{1}{s}} \delta_h(z_0,z)  \\
	&\quad + \langle b,(x-x_0) \rangle (z-z_0)+  \langle p, (x,z) -(x_0,z_0) \rangle + U(x_0,z_0) + o(|z-z_0|^2). 
\end{align*}

Since,
\[
\lim_{z \to z_0} \frac{\delta_h(z_0,z)}{|z-z_0|^2} = \lim_{z \to z_0} \frac{h'(z) - h'(z_0)}{2(z-z_0)} = \lim_{z \to z_0} \frac{h''(z)}{2} = \frac{h''(z_0)}{2},
\]
we have that
\[
o(|z-z_0|^2) = o(\delta_h(z_0,z)) \quad \hbox{as}~z \to z_0. 
\]
Therefore,
\[
\frac{U(x,z) - P(x,z)}{\delta_{\Phi}((x_0,z_0), (x,z))} \to 0 \quad \hbox{as}~(x,z) \to(x_0,z_0).
\]
In particular, given $\eps>0$, there is $\delta>0$ such that 
if $0 < |(x,z)-(x_0,z_0)|< \delta$, then
\[
P_\eps(x,z) := P(x,z) - \eps \delta_{\Phi}((x_0,z_0), (x,z)) < U(x,z)
\]
Therefore, $P_\eps$ touches $U$ from below at $(x_0,z_0)$. 
\end{proof}

\subsection{Point-to-measure estimate}

We prove a point-to-measure estimate for $C_s$-viscosity supersolutions which, in a sense, plays the role of the Alexandroff--Bakelman--Pucci estimate for fully nonlinear equations. 
We show that if we slide Monge--Amp\`ere paraboloids of fixed opening $a>0$ with vertices in a closed, bounded set from below until they touch the graph of $U$ for the first time, then the Monge--Amp\`ere measure of the contact points is a universal portion of the Monge--Amp\`ere measure of the set of vertices. 

We use the notation $f^+ = \max\{f,0\}$. 

\begin{thm}\label{thm:point-to-measure}
Assume that $\Omega$ is a bounded domain and that $a^{ij}(x):\Omega \to \R$ are bounded, measurable functions that satisfy \eqref{eq:ellipticity}. 
Let $Q_R = Q_R(\tilde{x},\tilde{z}) \subset \subset \Omega \times \R$, 
$f \in L^\infty(Q_R\cap \{z=0\})$, and $F \in  L^\infty(Q_R)$.
Suppose that $U\in C(\overline{Q_R})$, symmetric across $\{z=0\}$, is a $C_s$-viscosity supersolution to
\begin{equation}\label{eq:U-with-neumann}
 \begin{cases}
 a^{ij}(x) \partial_{ij}U + |z|^{2-\frac{1}{s}} \partial_{zz}U \leq F & \hbox{in}~Q_{R}  \cap \{z\not=0\} \\
\partial_{z} U(x,0) \leq f & \hbox{on}~Q_{R}  \cap \{z=0\}.
 \end{cases}
 \end{equation}
Let $B \subset \overline{Q}_R$ be a closed set. 
Fix $a>0$ and assume that 
\[
\|F\|_{L^{\infty}(Q_R)} \leq a.
\]
 For each $(x_v,z_v) \in B$, we slide Monge--Amp\`ere paraboloids of opening $a>0$ and vertex $(x_v,z_v)$ from below  until they touch the graph of $U$ for the first time. 
 Let $A$ denote the set of contact points and assume that $A \subset Q_R$. 
 Then $A$ is compact and if
 \[
 \mu_{\Phi}\left(B \cap \{(x,z) : |h'(z)| \leq \|f^+\|_{L^{\infty}(Q_R \cap \{z=0\})}/a \right)
 	\leq (1-\varepsilon_0)\mu_\Phi(B)
 \]
 then there is a constant $0 < c = c(n,\lambda,\Lambda) <1$ such that
 \[
 \mu_\Phi(A) \geq \varepsilon_0 c \mu_\Phi(B).
 \]
\end{thm}

For the proof, we first describe the setting and definition of the $\inf$-convolutions used to regularize the solutions. 
Consider an arbitrary Monge--Amp\`ere cube $ Q_R(\tilde{x},\tilde{z}) \subset \R^{n+1}$ such that $Q_R(\tilde{x},\tilde{z})^+ \not= \varnothing$. 
Let $S_{\bar{R}}(\bar{z})$ be such that $S_{\bar{R}}(\bar{z}) = S_R(\tilde{z})^+$. 
Consider the rectangles
\[
R_\rho := Q_{\rho R}(\tilde{x}) \times S_{\rho\bar{R}}(\bar{z}), \quad 0 < \rho < 1, 
\]
so that $R_\rho \subset \subset Q_R(\tilde{x},\tilde{z})^+$. 
For a fixed $0 < \rho < 1$, the $\inf$-convolution of $U \in C(Q_R(\tilde{x},\tilde{z}))$ on $R_\rho$ is given by
\begin{equation}\label{eq:defn-infconv}
U_\eps(x,z):= \inf_{(y,w) \in \overline{R}_\rho} \left\{ U(y,w) + \frac{1}{\eps} |(x,z) - (y,w)|^2 \right\} \quad \hbox{for}~(x,z) \in R_\rho. 
\end{equation}
By taking $(y,w) = (x,z)$ and using the definition of infimum, we clearly have that $U_{\eps}(x,z) \leq U(x,z)$. 
Moreover, since $U \in C(Q_R(\tilde{x},\tilde{z}))$, for each $(x_0,z_0) \in R_\rho$, there exists a point $(x_0^*, z_0^*) \in \overline{R}_\rho$ such that the infimum is attained:
\begin{equation}\label{eq:inf-attained}
U_\eps(x,z)=U(x_0^*, z_0^*) + \frac{1}{\eps}|(x_0,z_0) - (x_0^*,z_0^*)|^2.
\end{equation}
We always use the $*$ notation for such a point. 
Consequently,
\[
|(x_0,z_0) - (x_0^*,z_0^*)|
	\leq \sqrt{\eps(U(x_0,z_0) - U(x_0^*, z_0^*))} \leq \sqrt{2\eps \eta}, 
	\quad \eta:=  \|U\|_{L^{\infty}(\overline{R}_\rho)},
\]
which shows that $(x_0^*, z_0^*) \in B_{\sqrt{2\eps \eta}}(x_0,z_0)$ and $(x_0^*, z_0^*) \to (x_0,z_0)$ as $\eps \to 0$. 

We summarize the additional properties of $U_\eps$ in the next lemma. 

\begin{lem}\label{lem:inf-conv-properties}
The function $U_\eps$ in \eqref{eq:defn-infconv} satisfies the following properties. 
\begin{enumerate}[$(1)$]
\item $U_\eps \in C^{1}(R_\rho)$ and $U_\eps \nearrow U$ uniformly in $R_\rho$ as $\eps \to 0$. 
\item $U_\eps$ is semiconcave in $R_\rho$, that is, for every $(x_0,z_0) \in R_\rho$, there exists an affine function $\ell(x,z)$ such that 
\[
U_\eps(x,z) \leq \frac{1}{\eps} |(x,z) - (x_0,z_0)|^2 + \ell(x,z)
\]
with equality at $(x_0,z_0)$.
\item If $U$ satisfies
	\begin{equation}\label{eq:pucci-visco}
	\mathcal{P}^-(D^2_xU) + |z|^{2-\frac{1}{s}} \partial_{zz}U \leq a \quad \hbox{in the viscosity sense in}~R_\rho,
	\end{equation}
	and $0 < r < \rho$, then there is $\eps_1>0$ 
	such that for every $0 < \eps < \eps_1$, the function $U_\eps$ satisfies the following \emph{viscosity property} in $R_r:=Q_{rR}(\tilde{x}) \times Q_{r\bar{R}}(\bar{z})$: 
	\begin{center}
	if $(x_0,z_0) \in R_r$ 
	and $\phi \in C^2$ touches $U_\eps$ from below at $(x_0,z_0)$, then
	\begin{equation}\label{eq:Ue-visc-prop}
	\mathcal{P}^-(D^2_x\phi(x_0,z_0))+ |z_0^*|^{2-\frac{1}{s}}\partial_{zz}\phi(x_0,z_0) \leq a
	\end{equation}
	\end{center}
	for any $(x_0^*, z_0^*)$ that attains the infimum in the definition of $U_\eps(x_0,z_0)$, see \eqref{eq:inf-attained}. 
	Moreover, $(x_0^*,z_0^*) \in S_{\eps \eta}(x_0,z_0)$ satisfies
	\begin{equation}\label{eq:de}
	|h''(z_0^*) - h''(z_0)| \leq d_\eps
	\end{equation}
	for a positive constant $d_\eps$, independent of $z_0^*$, satisfying $d_\eps \to 0$ as $\eps \to 0$. 
\end{enumerate}
\end{lem}

We remark that the viscosity property \eqref{eq:Ue-visc-prop} does not necessarily mean that $U_\eps$ is a viscosity solution to an equation since
the map $z_0 \mapsto |z_0^*|^{2-\frac{1}{s}}$ is not necessarily a well-defined function. 

\begin{proof}
Properties $(1)$ and $(2)$ are classical. We only check $(3)$. 

Consider $R_r$ for a fixed $0 < r < \rho$. 
Let $(x_0,z_0) \in R_r$ and suppose that $\phi \in C^2$ touches $U_\eps$ from below at $(x_0,z_0)$. 
The function $\tilde{\phi}$ given by
\[
\tilde{\phi}(x,z)
	= \phi((x,z) + (x_0,z_0) - (x_0^*,z_0^*)) + U(x_0^*,z_0^*) - \phi(x_0,z_0)
\]
touches $U$ from below at $(x_0^*,z_0^*)$. 
By 
\eqref{eq:x-MA-euclidean}, \eqref{eq:sec-awayfrom0},  and \eqref{eq:section-cube}, we note that
\begin{equation}\label{eq:E-to-MA}
\begin{aligned}
(x_0^*, z_0^*) \in B_{\sqrt{\eps \eta}}(x_0,z_0)
	&\subset B_{\sqrt{2\eps \eta}}(x_0) \times B_{\sqrt{2\eps \eta}}(z_0)\\
	&\subset S_{\eps \eta}(x_0) \times S_{\sigma \eps \eta}(z_0)
	\subset Q_{\eps \eta}(x_0) \times S_{\sigma \eps \eta}(z_0)
\end{aligned}
\end{equation}
where $\sigma = \|h''\|_{L^\infty(S_{r \bar{R}}(\bar{z}))}$. 
By Lemma \ref{lem:Guti} and for $\eps_1 = \eps_1(n,s,r,\rho,\eta,\sigma, R,\bar{R})>0$ sufficiently small, 
\begin{equation}\label{eq:usingGuti}
\begin{aligned}
(x_0^*, z_0^*) 
	\in Q_{\eps \eta}(x_0) \times S_{\sigma \eps \eta}(z_0)
	&\subset  Q_{C_0(\rho-r)^{p_0} R}(x_0) \times S_{C_1(\rho-r)^{p_1} \bar{R}}(z_0) \\
	&\subset Q_{\rho R}(\tilde{x}) \times S_{\rho \bar{R}}(\bar{z}) = R_\rho. 
\end{aligned}
\end{equation} 
Since $U$ is a viscosity supersolution in $R_\rho$, we have
\[
\mathcal{P}^- (D^2_x\tilde{\phi}(x,z)) + |z|^{2-\frac{1}{s}} \partial_{zz}\tilde{\phi}(x,z) \bigg|_{(x,z) = (x_0^*, z_0^*)} \leq a.
\]
In particular, the viscosity property holds:
\[
\mathcal{P}^- (D^2_x\phi(x_0,z_0)) + |z_0^*|^{2-\frac{1}{s}} \partial_{zz}\phi(x_0,z_0) \leq a.
\]

Lastly, by the mean value theorem
\[
h''(z_0^*) - h''(z_0)
	= h'''(\xi)(z_0^*- z_0)
\]
for some $\xi$ between $z_0^*$ and $z_0$. Using \eqref{eq:E-to-MA} and \eqref{eq:usingGuti}, we find that
\[
|h''(z_0^*) - h''(z_0)| 
	\leq \|h'''\|_{L^{\infty}(S_{\rho \bar{R}}(\bar{z}))} \sqrt{2\eps \eta} =:d_\eps.
\]
\end{proof}

We now prove the point-to-measure estimate for regularized functions $U_\eps$ with $\eps>0$ fixed. 

\begin{lem}\label{lem:measure-estimate-Ue}
Suppose that $U_\eps$ is as in \eqref{eq:defn-infconv} and assume that $U_\eps$ satisfies the viscosity property \eqref{eq:Ue-visc-prop} in $R_r$ with \eqref{eq:de}. 
Let $B \subset \overline{R}_r$ be a closed set and fix $a>0$.
 For each $(x_v,z_v) \in B$, we slide Monge--Amp\`ere paraboloids of opening $a>0$ and vertex $(x_v,z_v)$ from below  until they touch the graph of $U_\eps$ for the first time. 
 Let $A_\eps$ denote the set of contact points and assume that $A_\eps \subset R_r$. 
Then $A_\eps$ is compact and there is $C(n,\lambda, \Lambda)  >0$ such that
 \[
C\left( \mu_{\Phi}(A_\eps) + d_\eps |A_\eps| \right)  \geq  \mu_\Phi(B)
 \]
where $d_\eps$ is the constant in \eqref{eq:de}. 
 \end{lem}

\begin{proof}
The proof that $A_\eps$ is compact follows exactly as in \cite[Theorem 7.1]{Stinga-Vaughan}. 
 
Let $(x_0,z_0) \in A_\eps$. 
There exists a Monge--Amp\`ere paraboloid $P$ of opening $a>0$ and vertex $(x_v,z_v) \in B$ that touches $U$ from below at $(x_0,z_0)$. 
By Lemma \ref{eq:MA-classical}, $U_\eps$ can be touched from below by a classical paraboloid at $(x_0,z_0)$. 
Moreover, by Lemma \ref{lem:inf-conv-properties}, $U_\eps$ is semiconcave and can thus be touched from above by a classical paraboloid at $(x_0,z_0)$. 
Consequently,  $U_\eps$ is differentiable at $(x_0,z_0)$  and the vertex $(x_v,z_v)$ is determined uniquely by 
\[
(x_v, h'(z_v)) = (x_0, h'(z_0)) + \frac{1}{a} \nabla U_\eps(x_0,z_0).
\]
Equivalently, 
\[
\nabla \Phi(x_v,z_v) = \nabla \left( \Phi + \frac{1}{a} U_\eps\right)(x_0,z_0).
\]

Let $Z$ denote the set of points $(x_0,z_0) \in R_r$ for which $U_\eps$ can be approximated by a classical second-order polynomial near $(x_0,z_0)$.
That is,
\begin{equation}\label{eq:Ueps-approx}
\begin{aligned}
U_\eps(x,z)
	&= U_\eps(x_0,z_0) + \langle \nabla U_\eps(x_0,z_0), (x,z) -(x_0,z_0) \rangle\\
	& \quad + \frac{1}{2} \langle D^2 U_\eps(x_0,z_0)((x,z) - (x_0,z_0)), (x,z) - (x_0,z_0) \rangle \\
	&\quad + o(|(x,z) - (x_0,z_0)|^2).
\end{aligned}
\end{equation}
Since $U_\eps$ is semiconcave, we have that $|R_r \setminus Z| = 0$
and $[\nabla U_\eps]_{\text{Lip}} \leq C = C(a, \eps^{-1}, R_\rho)$. 

Consider the map $T:A_\eps \to T(A_\eps) = \nabla \Phi(B)$ given by
\[
T(x_0,z_0) = \nabla \left( \Phi + \frac{1}{a} U_\eps\right)(x_0,z_0). 
\]
Since $T$ is Lipschitz and injective on $A_\eps$,  the area formula for Lipschitz maps gives
\begin{equation}\label{eq:measure-setup}
\begin{aligned}
\mu_\Phi(B) = |\nabla \Phi(B)| 
	&= \int_{T(A_\eps)} dy \, dw
	= \int_{A_\eps}|\det(\nabla T(x,z))| \,dz \, dx \\
	&= \int_{A_\eps \cap Z}\left|\det\left(D^2\Phi(x,z) + \frac{1}{a} D^2U_\eps(x,z)\right)\right| \,dz \, dx.
\end{aligned}
\end{equation}

We claim that there is a constant $C = C(n,\lambda,\Lambda)>0$ such that for all $(x_0,z_0) \in A_\eps \cap Z$,
\begin{equation}\label{eq:ueps-claim}
-aD^2\Phi(x_0,z_0) \leq D^2U_\eps(x_0,z_0) \leq Ca D^2\Phi(x_0,z_0^*)
\end{equation}
for any $z_0^*$ such that $(x_0^*,z_0^*)$ attains the infimum in the definition of $U_\eps(x_0,z_0)$. 
The first inequality is clear since $P$ touches $U_\eps$ from below at $(x_0,z_0)$. 
For the second inequality, suppose by way of contradiction that
\begin{equation}\label{eq:bwoc-eps}
D^2U_\eps(x_0,z_0) >CaD^2\Phi(x_0,z_0^*) = Ca \begin{pmatrix} I & 0 \\ 0 &(z_0^*)^{\frac{1}{s}-2} \end{pmatrix} \quad \hbox{for all}~C>0. 
\end{equation}
From \eqref{eq:Ueps-approx} and by Lemma \ref{lem:polynomial-approx},
for all $\tau>0$, the second-order Monge--Amp\`ere polynomial 
\begin{align*}
\bar{P}(x,z)
	&:=\frac{1}{2} \langle D^2_xU_\eps(x_0,z_0)(x-x_0), (x-x_0) \rangle  + \partial_{zz}U_\eps(x_0,z_0) |z_0|^{2-\frac{1}{s}} \delta_h(z_0,z)  \\
	&\quad - \tau\left(\frac{1}{2} |x-x_0|^2 + \delta_h(z_0,z)\right)\\
	&\quad + \langle \nabla_x \partial_{z}U_\eps(x_0,z_0),(x-x_0) \rangle (z-z_0)+  \langle \nabla U_\eps(x_0,z_0), (x,z) -(x_0,z_0) \rangle + U_\eps(x_0,z_0)
\end{align*}
touches $U_\eps$ from below at $(x_0,z_0)$. Since $U_\eps$ satisfies the viscosity property \eqref{eq:Ue-visc-prop}, we have
\[
\mathcal{P}^-(D^2_x \bar{P}(x_0,z_0)) + (z_0^*)^{2-\frac{1}{s}} \partial_{zz}\bar{P}(x_0,z_0) \leq a,
\]
that is, 
\[
\mathcal{P}^-(D^2_xU_\eps(x_0,z_0)  - \tau I)
 	+ (z_0^*)^{2-\frac{1}{s}}  (\partial_{zz}U_\eps(x_0,z_0) - \tau h''(z_0))\leq a. 
\]
Sending $\tau \to 0$ gives
\begin{equation}\label{eq:using-visc}
\mathcal{P}^-(D^2_xU_\eps(x_0,z_0))
 	+(z_0^*)^{2-\frac{1}{s}} \partial_{zz}U_\eps(x_0,z_0) \leq a. 
\end{equation}
On the other hand, by \eqref{eq:bwoc-eps},
\[
D^2U_\eps(x_0,z_0) 
	> Ca \begin{pmatrix} e_k\otimes e_k & 0 \\ 0 & 0 \end{pmatrix}
	>  Ca \begin{pmatrix} e_k\otimes e_k & 0 \\ 0 & 0 \end{pmatrix}
		- a \begin{pmatrix} I& 0 \\ 0 & |z_0|^{\frac{1}{s}-2}\end{pmatrix}.
\]
where $e_k$, $k=1,\dots, n$, are the standard basis vectors in $\R^n$ and $ \otimes$ denotes the usual tensor product. 
Since $\mathcal{P}^-$ is monotone increasing, 
\begin{equation}\label{eq:x-part}
\mathcal{P}^-(D^2_xU_\eps(x_0,z_0)) 
	\geq \mathcal{P}^-(Ca(e_k\otimes e_k) - aI)
	= [\lambda (C-1) - \Lambda (n-1)]a. 
\end{equation}
Also by \eqref{eq:bwoc-eps}, we have
\[
D^2U_\eps(x_0,z_0) 
	> Ca \begin{pmatrix} 0 & 0 \\ 0 & (z_0^*)^{\frac{1}{s}-2}\end{pmatrix}.
\]
By definition of positive definite matrices,
$
\partial_{zz}U_\eps(x_0,z_0)
	> Ca (z_0^*)^{\frac{1}{s}-2}.
$
Equivalently,
\begin{equation}\label{eq:z-part}
(z_0^*)^{2 - \frac{1}{s}}\partial_{zz}U_\eps(x_0,z_0)
	> Ca.
\end{equation}
Combining \eqref{eq:using-visc}, \eqref{eq:x-part}, and \eqref{eq:z-part}, we have
\begin{align*}
a &\geq \mathcal{P}^-(D^2_xU_\eps(x_0,z_0))  + (z_0^*)^{2- \frac{1}{s}}\partial_{zz}U_\eps(x_0,z_0)\\
&\geq [\lambda (C-1) - \Lambda (n-1)]a  + Ca
= [(\lambda+1) C - (\Lambda (n-1)+\lambda)]a,
\end{align*}
which is a contradiction for sufficiently large $C= C(n,\lambda, \Lambda)>0$. Therefore, \eqref{eq:ueps-claim} holds. 

Using \eqref{eq:de}, we find that
\[
h''(z_0^*) \leq h''(z_0) + d_\eps
\]
which together with \eqref{eq:ueps-claim} gives
\begin{align*}
-aD^2\Phi(x_0,z_0) \leq D^2U_\eps(x_0,z_0) &\leq Ca \begin{pmatrix} I &0 \\0 & h''(z_0^*) \end{pmatrix}
	\leq Ca \begin{pmatrix} I &0 \\0 &h''(z_0) + d_\eps \end{pmatrix} 
\end{align*}
for all $(x_0,z_0) \in A_\eps \cap Z$.
Continuing now from \eqref{eq:measure-setup}, we arrive at the desired conclusion
\begin{align*}
\mu_\Phi(B)
	&= \int_{A_\eps \cap Z}\det\left(D^2\Phi(x,z) + \frac{1}{a} D^2U_\eps(x,z)\right) \,dz \, dx\\
	&\leq \int_{A_\eps \cap Z}\det\left( \begin{pmatrix}
	I & 0 \\ 0 & h''(z)  \end{pmatrix} + C  \begin{pmatrix} I &0 \\0 &h''(z) + d_\eps \end{pmatrix}\right) \,dz \, dx \\
	&= (1+C)^n\int_{A_\eps}[(1+C)h''(z) + Cd_\eps] \, dz \, dx\\
	&\leq (1+C)^{n+1} \left(\mu_{\Phi}(A_\eps) + d_\eps |A_\eps|\right).
\end{align*}
\end{proof}

We are now ready to prove Theorem \ref{thm:point-to-measure}. 

\begin{proof}[Proof of Theorem \ref{thm:point-to-measure}]
The proof that $A$ is compact follows exactly as in \cite[Theorem 7.1]{Stinga-Vaughan}. 

Without loss of generality, assume that $Q_R(\tilde{x},\tilde{z})^+  \not= \varnothing$. If $Q_R(\tilde{x},\tilde{z})^+  = \varnothing$, then $Q_R(\tilde{x},\tilde{z})^-  \not= \varnothing$ and the proof is analogous. 

Consider the sets
\[
B_0 := B \cap \left\{(x,z): \abs{h'(z)} \leq \frac{\norm{f^+}_{ L^{\infty}(Q_R(\tilde{x}))}}{a} \right\}, \quad
B_1^+ :=  B^+ \setminus B_0, \quad  B_1^- := B^- \setminus B_0. 
\]
Note that $B_0, B_1^+, B_1^- $ are mutually disjoint and satisfy $B = B_0 \cup B_1^+ \cup B_1^-$. 
We lift paraboloids of opening $a>0$ from below with vertices in $B_0$, $B_1^+$, $B_1^-$ to form the contact sets
$A_0$, $A_1^+$, $A_1^-$, respectively. Note that $A = A_0 \cup A_1^+ \cup A_1^-$, but $A_0$, $A_1^+$, $A_1^-$ are not necessarily disjoint.

It is enough to show that $\mu_{\Phi}(B_1^+) \leq C \mu_{\Phi}(A_1^+)$ for some positive constant $C = C(n,\lambda,\Lambda) >1$. 
Indeed, first note that the proof of $\mu_{\Phi}(B_1^-) \leq C \mu_{\Phi}(A_1^-)$ will be similar.  
Together with the hypothesis on $B_0$, we have
\begin{align*}
\mu_{\Phi}(B)
	&= \mu_{\Phi}(B_0) + \mu_{\Phi}(B_1^+) + \mu_{\Phi}(B_1^-) 
	\leq (1-\varepsilon_0)\mu_{\Phi}(B) +C \mu_{\Phi}(A_1^-) + C \mu_{\Phi}(A_1^+)
\end{align*}
which implies
\[
\mu_{\Phi}(A) \geq \frac{1}{2}(\mu_{\Phi}(A_1)  + \mu_{\Phi}(A_2))  \geq \frac{\varepsilon_0}{2C}\mu_{\Phi}(B). 
\]

We now show $\mu_{\Phi}(B_1^+) \leq C \mu_{\Phi}(A_1^+)$. 
Let $\bar{R}$ and $\bar{z}$ be such that $S_{\bar{R}}(\bar{z}) = S_{R}(\tilde{z})^+$. 
For $0 < \rho < 1$, consider the rectangle $R_\rho := Q_{\rho R}(\tilde{x}) \times S_{\rho \bar{R}}(\bar{z}) \subset \subset Q_R(\tilde{z},\tilde{z})^+$. 
Let $U_\eps$ denote $\inf$-convolution of $U$ in $R_\rho$ given in \eqref{eq:defn-infconv}. 
Since $U$ is a $C_s$-viscosity supersolution to \eqref{eq:U-with-neumann}, we have that $U$ satisfies \eqref{eq:pucci-visco} in $R_\rho$, see Remark \ref{rem:pucci}. 
Fix $0 < r < \rho$. By Lemma \ref{lem:inf-conv-properties},  there is an $\eps_1>0$ such that for all $0 < \eps < \eps_1$, 
the regularized function $U_\eps$ satisfies the viscosity property \eqref{eq:Ue-visc-prop} in $R_r$ with \eqref{eq:de}. 

Define a new vertex set $B_r^+:= B_1^+ \cap \overline{R_r}$. Slide paraboloids of opening $a>0$ and vertices in $B_r^+$ from below until they touch the graph of $U_\eps$ for the first time. 
Let $A_{r,\eps}^+$ be the corresponding set of contact points for $U_\eps$ in $R_r$. 
By Lemma \ref{lem:measure-estimate-Ue},
\[
\mu_\Phi(B_r^+) \leq C\left( \mu_{\Phi}(A_{r,\eps}^+) + d_\eps |A_{r,\eps}^+|\right).
\]
One can check that
\[
\limsup_{k \to \infty} A_{r,1/k}^+ = \bigcap_{m=1}^{\infty} \bigcup_{k=m}^{\infty} A_{r,1/k}^+ \subset A_r^+
\]
where $A_r^+$ is the contact set for $U$ in $B_r^+$. Since $d_{1/k} \to 0$ as $k \to \infty$, it follows that $\mu_\Phi(B_r^+)  \leq C\mu_{\Phi}( A_{r}^+)$. 
Since $A_r^+ \subset A_1^+$, we further have that
\[
\mu_\Phi( B_1^+ \cap \overline{R_r}) 
	= \mu_\Phi(B_r^+) \leq C\mu_{\Phi}(A_{1}^+).
\]
Taking $r \to \rho$ and then $\rho \to 1$, we finally arrive at
\[
\mu_\Phi(B_1^+) \leq  (1+C)^n \mu_\Phi(A_1^+),
\]
which completes the proof. 
\end{proof}

\subsection{Proof of Theorem \ref{thm:F-harnack3}}

With Theorem \ref{thm:point-to-measure} for $C_s$-viscosity solutions in-hand, the proof of Theorem \ref{thm:F-harnack3} then follows along the same lines as in \cite{Stinga-Vaughan} under the additional assumption that  
$\|F\|_{L^\infty(Q_{KR})} \leq a$. 
For this reason, we only sketch the idea next. 

As in \cite[Lemma 8.2]{Stinga-Vaughan}, we construct explicit barriers that are used to prove a detachment lemma, like \cite[Lemma 9.2]{Stinga-Vaughan}, on how the solution $U$ separates from a touching Monge--Amp\`ere paraboloid.  
With this and the point-to-measure estimate (Theorem \ref{thm:point-to-measure}), we prove a localization lemma which morally says if $U$ can be touched from below by a Monge--Amp\`ere  paraboloid of opening $a>0$, then $U$ can be touched nearby by narrower Monge--Amp\`ere  paraboloids of opening $Ca$, for universal $C>1$, in a set of positive measure, see \cite[Lemma 9.4]{Stinga-Vaughan}. 
With these ingredients and a covering lemma \cite[Lemma 10.1]{Stinga-Vaughan}, 
we end by following the proof of \cite[Theorem 5.3]{Stinga-Vaughan}.\qed

\section{Approximation lemma}\label{sec:approx}

In this section, we prove that if the coefficients $a^{ij}(x)$ are close to $\delta^{ij}$ and both $f$ and $F$ are sufficiently small, then any $C_s$-viscosity solution $U$ to the extension problem \eqref{eq:sub} can be approximated by a harmonic function, that is, a solution to \eqref{eq:harmonic-2}.

\begin{lem}\label{lem:approximation}
For any $\varepsilon>0$, there is $\varepsilon_0 = \varepsilon_0(n,\lambda,\Lambda, s,\varepsilon)>0$ such that if $a^{ij} \in C(T_1)$ satisfies \eqref{eq:ellipticity}, $f \in C(T_1) \cap L^{\infty}(T_1)$, $F \in C(S_1 \times S_1^+) \cap L^\infty(S_1 \times S_1^+)$ with
\[
\norm{a^{ij}(\cdot) - \delta^{ij}}_{L^{\infty}(T_1)} + \norm{f}_{L^{\infty}(T_1)}+\norm{F}_{L^{\infty}(S_1 \times S_1^+)} < \varepsilon_0,
\]
and $U\in C(\overline{S_1\times S_1^+})$ is a $C_s$-viscosity solution to 
\begin{equation*}\label{eq:main}
\begin{cases}
a^{ij}(x) \partial_{ij} U+ z^{2 - \frac{1}{s}} \partial_{zz}U = F & \hbox{in}~S_1 \times S_{1}^+ \\
\partial_{z} U = f & \hbox{on}~T_1 \\
\norm{U}_{L^{\infty}(S_1 \times S_{1}^+)} \leq 1
\end{cases}
\end{equation*}
then there is a classical solution $H$ to
\begin{equation}\label{eq:H}
\begin{cases}
\Delta_x H +z^{2 - \frac{1}{s}} \partial_{zz}H = 0 & \hbox{in}~S_{3/4} \times S_{3/4}^+ \\
\partial_{z} H = 0 & \hbox{on}~T_{3/4} \\
\norm{H}_{L^{\infty}(S_{3/4} \times S_{3/4}^+)} \leq 1
\end{cases}
\end{equation}
such that
\[
\norm{U-H}_{L^{\infty}((S_{3/4} \times S_{3/4}^+) \cup T_{3/4})} \leq \varepsilon.
\]
\end{lem}

\begin{proof}
Suppose, by way of contradiction, there is a $\varepsilon>0$ such that for all $k \in \N$, there exist $a^{ij}_k, f_k \in C(T_1)$ satisfying
\[
\|a^{ij}_k(\cdot) - \delta^{ij}\|_{L^{\infty}(T_1)} + \norm{f_k}_{L^{\infty}(T_1)}+ \norm{F_k}_{L^{\infty}(S_1 \times S_1^+)} < \frac{1}{k}
\]
and $C_s$-viscosity solutions $U_k$ to 
\[
\begin{cases}
a^{ij}_k(x) \partial_{ij} U_k+ z^{2 - \frac{1}{s}} \partial_{zz}U_k = F_k & \hbox{in}~S_{1}\times S_1^+ \\
\partial_z U_k = f_k& \hbox{on}~T_1\\
\norm{U_k}_{L^{\infty}(S_1 \times S_{1}^+)} \leq 1,
\end{cases}
\]
but such that every classical solution $H$ to \eqref{eq:H} satisfies
\begin{equation}\label{eq:not-close}
\norm{U_k-H}_{L^{\infty}((S_{3/4} \times S_{3/4}^+) \cup T_{3/4})} > \varepsilon \quad \hbox{for all}~k \in \N.
\end{equation}
As a consequence of Theorem \ref{thm:F-harnack} (and recalling the notation in Section \ref{sec:Holder}), we have 
\[
\|U_k\|_{C_{\Phi}^{\alpha_1}(\overline{S_{3/4} \times S_{3/4}^+})} \leq C(\|U_k\|_{L^{\infty}(S_1 \times S_1^+)} +\norm{f_k}_{L^{\infty}(T_1)}+\norm{F_k}_{L^{\infty}(S_1 \times S_1^+)}) \leq 2C
\]
for $C = C(n,\lambda, \Lambda, s)>0$. 
Therefore, the family $(U_k)_{k \in \N}$ is uniformly bounded and equicontinuous in $\overline{S_{3/4} \times S_{3/4}^+}$. 
By Arzel\`a-Ascoli, there is a subsequence, still denoted by $(U_k)_{k \in \N}$, and a function $U_{\infty} \in C^{\alpha_1}_\Phi(\overline{S_{3/4} \times S_{3/4}^+})$ such that 
\begin{equation}\label{eq:Uk to U}
U_k \to U_{\infty} \quad \hbox{uniformly on compact subsets of $(S_{3/4} \times S_{3/4}^+) \cup T_{3/4}$ as}~k \to \infty.
\end{equation}
By Lemma \ref{lem:stability}, $U_{\infty}$ is a $C_s$-viscosity solution to \eqref{eq:H}. 
Moreover, by Proposition \ref{lem:H-classical}, $U_\infty$ is a classical solution in $(S_{3/4} \times S_{3/4}^+) \cup T_{3/4}$. 
Together with \eqref{eq:Uk to U}, this contradicts \eqref{eq:not-close}. 
\end{proof}

\begin{rem}\label{rem:approximation}
In the same way, we can show that Lemma \ref{lem:approximation} holds with Monge--Amp\`ere cylinders $S_1 \times S_\rho^+$
in place of $S_1\times S_1^+$, for any $0< \rho \leq 1$, with $\eps_0$ independent of $\rho$. 
\end{rem}

\section{Proof of Theorem \ref{thm:schauder-intro}}\label{sec:schauder}

This section is devoted to the proof of Theorem \ref{thm:schauder-intro}. 
With the extension characterization, Theorem \ref{thm:extension}, the main point is to show that $C_s$-viscosity solutions to 
\begin{equation}\label{eq:extension}
\begin{cases}
a^{ij}(x) \partial_{ij}U + z^{2-\frac{1}{s}} \partial_{zz}U = 0 & \hbox{in}~S_1 \times S_1^+\\
\partial_{z}U(x,0) = f(x) & \hbox{on}~T_1\\
\end{cases}
\end{equation}
are $C^{\alpha+2s}_{\Phi}$ at the origin. In particular, we prove the following result. 

\begin{thm}\label{thm:schauder-ext}
Fix $0 < s < 1$.  
Suppose $a^{ij} \in C(T_1) \cap L^{\infty}(T_1)$ satisfy \eqref{eq:ellipticity} 
and $a^{ij}(0) = \delta^{ij}$. 
Suppose also that $f \in L^{\infty}(T_1)$ is such that $f \in C^{\alpha}(0)$ for some $0 < \alpha < 1$ and $f(0) = 0$. 
\begin{enumerate}[start=1,label={(\arabic*)}]
	\item \label{item:schauder1}
	Suppose that $0 < \alpha +2s < 1$. There is $\varepsilon_0 = \varepsilon_0(n,s, \lambda,\Lambda)>0$ and a constant $C_0>0$ such that if
	\[
	\|a^{ij}(\cdot) - \delta^{ij}\|_{L^{\infty}(T_1)} \leq \varepsilon_0
	\]
	and $U \in C(\overline{S_1 \times S_1^+})$ is a $C_s$-viscosity solution \eqref{eq:extension}, 
	then there is a constant $c$ such that
	\[
	\|U - c\|_{L^{\infty}(S_{r^2} \times S_{r^2}^+)} \leq C_1 r^{\alpha+2s} \quad \hbox{for all}~r>0~\hbox{sufficiently small},
	\]
	where $C_1 + |c| \leq C_0(\|U\|_{L^\infty(S_1 \times S_1^+)}+\|f\|_{C^{0,\alpha}(T_1)})$.
	\item \label{item:schauder2}
	Suppose that $1 < \alpha +2s < 2$. There is $\eps_0=\eps_0(n,s, \lambda,\Lambda)>0$ and a constant $C_0>0$ such that if
	\[
	\|a^{ij}(\cdot) - \delta^{ij}\|_{L^{\infty}(T_1)} \leq \varepsilon_0
	\]
	and $U \in C(\overline{S_1 \times S_1^+})$ is a $C_s$-viscosity solution \eqref{eq:extension}, 
	then there is a linear function $\ell(x) = \langle b,x \rangle + c$ such that
	\[
	\|U - \ell\|_{L^{\infty}(S_{r^2} \times S_{r^2}^+)} \leq C_1 r^{\alpha+2s} \quad \hbox{for all}~r>0~\hbox{sufficiently small},
	\]
	where $C_1 + |b| + |c| \leq C_0(\|U\|_{L^\infty(S_1 \times S_1^+)}+\|f\|_{C^{0,\alpha}(T_1)})$.
	\item \label{item:schauder3}
	Suppose that $2 < \alpha +2s < 3$.  There is $\eps_0=\eps_0(n,s, \lambda,\Lambda)>0$ and a constant $C_0>0$ such that if $a^{ij} \in C^{\alpha+2s-2}(0)$ with
	\[
	\|a^{ij}(\cdot) - \delta^{ij}\|_{L^{\infty}(T_1)} \leq \varepsilon_0
	\]
	\[
	|a^{ij}(x) - \delta^{ij}| \leq [a^{ij}]_{C^{\alpha+2s-2}(0)} |x|^{\alpha+2s-2} \quad \hbox{for all}~x \in T_1
	\]
	and $U \in C(\overline{S_1 \times S_1^+})$ is a $C_s$-viscosity solution \eqref{eq:extension}, 
	then there is a Monge--Amp\`ere polynomial 
	\[
	P(x,z) = \frac{1}{2} \langle \A x, x \rangle + \langle b,x \rangle + c + dh(z)\]
	 such that
	\[
	\|U - P\|_{L^{\infty}(S_{r^2}\times S_{r^3}^+)} \leq C_1 r^{\alpha+2s} \quad \hbox{for all}~r>0~\hbox{sufficiently small},
	\]
	where $C_1 +|\A|+ |b| + |c| +|d| \leq C_0(\|U\|_{L^\infty(S_1 \times S_1^+)}+\|f\|_{C^{0,\alpha}(T_1)})$.
\end{enumerate}
\end{thm}

\begin{rem}\label{rem:SchauderAtPoints}
Cases \ref{item:schauder1} and \ref{item:schauder2} of Theorem \ref{thm:schauder-ext}
are Cordes--Nirenberg-type results
for the extension problem and, in particular, for the fractional problem.

In Case \ref{item:schauder3}, if $2 < \alpha+2s<3$ and $0 < \alpha <1$, then it must be that $ \frac{1}{2}<s<1$. This is precisely when \eqref{eq:extensionPDE} is degenerate near $\{z=0\}$, so we need the different scaling $r^3$ to compensate the equation. 

Recalling the definitions in Section \ref{sec:Holder}, 
the conclusion of Theorem \ref{eq:extensionPDE} is equivalent to $U \in C_{\Phi}^{k,\alpha+2s-k}(0,0)$ for $k=0,1,2$ in Cases \ref{item:schauder1},\ref{item:schauder2},\ref{item:schauder3}, respectively. If $\Omega'\subset \subset \Omega$, then by rescaling and translating the equation with respect to the $x$-variable, we can show that solutions $U$ to \eqref{eq:extension-intro} satisfy $U \in C_{\Phi}^{k,\alpha+2s-k}(x_0,0)$ for any $(x_0,0) \in \Omega' \times \{z=0\}$. 

Finally, it is clear that Theorem \ref{thm:schauder-intro} follows from Theorem \ref{thm:extension} and Theorem \ref{thm:schauder-ext}.
\end{rem}

For ease in the proof, we will often reference the following assumptions that describe 
what we mean by a normalized solution $U$.
\begin{enumerate}[start=1,label={(A\arabic*)}]
\item \label{item:A1}
	$a^{ij}(0) = \delta^{ij}$ and $f(0) = 0$,
\item  \label{item:A2}
	there is $\eps_0>0$ such that $\|a^{ij}(\cdot) - \delta^{ij}\|_{L^{\infty}(T_1)},\|f\|_{L^{\infty}(T_1)}< \eps_0$,
\item \label{item:A3}
	 the function $U$ satisfies $\|U\|_{L^{\infty}(S_1 \times S_1^+)} \leq 1$ and is a $C_s$-viscosity solution to 
	\[
\begin{cases}
a^{ij}(x) \partial_{ij}U + z^{2-\frac{1}{s}} \partial_{zz}U = 0 & \hbox{in}~S_1 \times S_1^+\\
\partial_{z}U(x,0) = f(x) & \hbox{on}~T_1, \\
\end{cases}
	\]
\item \label{item:A-new} $f \in C^\alpha(0)$ satisfies $[f]_{C^\alpha(0)}2^{\frac{\alpha}{2}} \leq \eps_0$ for $\eps_0>0$, 
\item \label{item:A4-} if $2 < \alpha+2s < 3$, 
	it holds that $a^{ij} \in C^{\alpha+2s-2}(0)$ with 
\[
|a^{ij}(x) - \delta^{ij}| \leq [a^{ij}]_{C^{\alpha+2s-2}(0)} |x|^{\alpha+2s-2}
\quad \hbox{for all}~x \in T_1
\]
and  $\bar{C}[a^{ij}]_{C^{\alpha+2s-2}(0)} \leq \eps_0$ for some $\bar{C}>0$ and $\eps_0>0$. 
\end{enumerate}
We will also need the following assumptions corresponding to the nonzero right
hand side $F$. For bounded $F = F(x)$, assume that
\begin{enumerate}[start=2,label={(A\arabic*')}]
\item  \label{item:A2-}
	given $0 < \eps_0 \leq 1$, both \ref{item:A2} and $\|F\|_{L^{\infty}(T_1)}< \eps_0$ hold,
\item \label{item:A3-}
	 given $0 < \rho  \leq 1$, the function $U$ satisfies $\|U\|_{L^{\infty}(S_1 \times S_\rho^+)} \leq 1$ and is a
	 $C_s$-viscosity solution to 
	\[
\begin{cases}
a^{ij}(x) \partial_{ij}U + z^{2-\frac{1}{s}} \partial_{zz}U = F & \hbox{in}~S_1\times S_\rho^+\\
\partial_{z}U(x,0) = f(x) & \hbox{on}~T_1.\\
\end{cases}
	\]
\end{enumerate}
To prove Theorem \ref{thm:schauder-ext}, it is enough to consider {normalized} solutions. 
Indeed, for \ref{item:A1}, we may consider an orthogonal change of variables in $x$ to assume $a^{ij}(0) = \delta^{ij}$ and if $f(0) \not=0$, we replace $U$ by $U - f(0)z$. 
We may assume \ref{item:A2-} and \ref{item:A3-}
by rescaling the equation in $x$ and considering
\[
\tilde{U}(x,z) =  \frac{U(x,z)}{\|U\|_{L^{\infty}(S_1 \times S_\rho^+)} + (\|f\|_{L^{\infty}(T_1)} + \|F\|_{L^{\infty}(T_1)})/\eps_0},
\]
and similarly for \ref{item:A2} and \ref{item:A3}. Assumptions \ref{item:A-new} and \ref{item:A4-} are also enough by rescaling. 

We now prove Theorem \ref{thm:schauder-ext} for normalized solutions by considering separately the three cases \ref{item:schauder1}, \ref{item:schauder2} and \ref{item:schauder3}. 
For each case, the desired polynomial arises as the limit of a sequence of approximating polynomials. 
The proofs rely on two main lemmas. 
The first is the inductive step in which we use the approximation lemma in Section \ref{sec:approx} to construct a suitable polynomial that is close to the solution $U$. 
In the second, we use a scaling argument to inductively build a sequence of approximating polynomials. 

\subsection{Proof of Theorem \ref{thm:schauder-ext}\ref{item:schauder1}}

\begin{lem}\label{lem:Inductive step-1}
Given $0 < \alpha +2s < 1$, there exist $0 < \varepsilon_0, \rho <1$ and a constant $c\in\R$ such that if \ref{item:A1} and \ref{item:A2} hold, then for any solution $U$ satisfying \ref{item:A3}, it holds that
\[
\|U - c\|_{L^{\infty}(S_{\rho^2} \times S_{\rho^2}^+)} \leq \rho^{\alpha +2s} \quad \hbox{and} \quad |c| \leq 2.
\]
\end{lem}

\begin{proof}
Let $0 < \eps < 1$ to be determined. 
Take $\eps_0>0$ as in Lemma \ref{lem:approximation}, so with \ref{item:A2}, there is classical solution $H$ to \eqref{eq:H} such that 
\[
\|U-H\|_{L^{\infty}((S_{3/4} \times S_{3/4}^+)\cup T_{3/4})} \leq \eps. 
\]
Note that
\[
\|H\|_{L^{\infty}((S_{3/4} \times S_{3/4}^+)\cup T_{3/4})} \leq \|U-H\|_{L^{\infty}((S_{3/4} \times S_{3/4}^+)\cup T_{3/4})} + \|U\|_{L^{\infty}((S_{3/4} \times S_{3/4}^+)\cup T_{3/4})} \leq 2.
\]
Set $c = H(0,0)$, so that $|c| \leq 2$. 
Let $\kappa = 3\min\{1, c_s\}/16$ where $c_s = 1/[2(1-s)]$. 
Recalling \eqref{eq:section-cube}, note that
\[
S_\kappa \times S_\kappa^+ \subset S_{(3/4)/4} \times S_{c_s(3/4)/4}^+ \subset \R^n \times \R^+.
\]
With this, we may apply Proposition \ref{lem:H-classical}, so that, for any $(x,z) \in (S_\kappa \times S_{\kappa}^+) \cup T_{\kappa}$, we have
\begin{align*}
|H(x,z) - c|
	&\leq |H(x,z) - H(x,0)| + |H(x,0) - H(0,0)|\\
	&\leq \|\partial_zH(x,\cdot)\|_{L_z^{\infty}((S_\kappa \times S_{\kappa}^+) \cup T_{\kappa})}z + \|\nabla_x H\|_{L_x^{\infty}((S_\kappa \times S_{\kappa}^+) \cup T_{\kappa})} |x|\\
	&\leq C(z^{\frac{1}{s}} + {|x|}) \\
	&\leq C(z^{\frac{2}{s}} + |x|^2)^{1/2}.
\end{align*}
Since $z^{\frac{1}{s}}$ is bounded in $S_\kappa^+$, we have
\[
|H(x,z) - c| 
	\leq  C(z^{\frac{1}{s}} + |x|^2)^{1/2}
	\leq C[\Phi(x,z)]^{1/2} = C[ \delta_{\Phi}((0,0), (x,z))]^{1/2}.
\]
Consequently, if $0 < \rho^2 < \kappa$, then
\begin{align*}
\|U-c\|_{L^{\infty}(S_{\rho^2} \times S_{\rho^2}^+)}
	&\leq \|U-H\|_{L^{\infty}(S_{\rho^2} \times  S_{\rho^2}^+)} + \|H-c\|_{L^{\infty}(S_{\rho^2} \times  S_{\rho^2}^+)}
	\leq \eps + C\rho
	\leq \rho^{\alpha+2s}
\end{align*}
by first choosing $\rho$ small enough to guarantee that $C \rho \leq \frac{1}{2} \rho^{\alpha+2s}$ and then
letting $\eps>0$ small so that $\eps \leq \frac{1}{2}\rho^{\alpha+2s}$. 
\end{proof}


\begin{lem}\label{lem:sequence-1}
In the setting of Lemma \ref{lem:Inductive step-1}, suppose additionally that \ref{item:A-new} holds.  
Then, there is a sequence of constants $c_k\in\R$ for $k \geq 0$ such that
\[
\|U - c_k\|_{L^{\infty}(S_{\rho^{2k}} \times S_{\rho^{2k}}^+)} \leq \rho^{k(\alpha+2s)} \quad \hbox{and} \quad |c_k - c_{k+1}| \leq 2 \rho^{k(\alpha+2s)}.
\]
\end{lem}

\begin{proof}
We prove the lemma by induction. 
Setting $c_0 = c_1 = 0$, we see that the result holds for $k=0$ since  $U$ is bounded by 1. 
Now assume that the statement holds for some $k \geq 0$. Consider the rescaled solution
\[
\tilde{U}(x,z)
	:= \frac{1}{\rho^{k(\alpha+2s)}} (U(\rho^kx, \rho^{2sk}z) - c_k), \quad (x,z) \in (S_1 \times S_1^+) \cup T_1. 
\] 
By \eqref{eq:scaling0} with $\rho^k$ in place of $\rho$, 
\begin{equation}\label{eq:scale-k}
(x,z) \in (S_1 \times S_1^+) \cup T_1 \quad \hbox{if and only if} \quad (\rho^kx,\rho^{2sk}z) \in (S_{\rho^{2k}} \times S_{\rho^{2k}}^+) \cup T_{\rho^{2k}},
\end{equation}
so $\tilde{U}$ is well-defined on $(S_1 \times S_1^+) \cup T_1$. 
Set 
\begin{equation}\label{eq:a-f-tilde}
\tilde{a}^{ij}(x) = a^{ij}(\rho^k x) \quad \hbox{and} \quad \tilde{f}(x) = \rho^{-k\alpha} f(\rho^kx).
\end{equation} 
As in Lemma \ref{lem:scaling equation}, for any $(x,z) \in S_1 \times S_1^+$
\begin{align*}
\tilde{a}^{ij}(x) &\partial_{ij}\tilde{U}(x,z) + z^{2 -\frac{1}{s}} \partial_{zz}\tilde{U}(x,z)\\
	&=  \frac{\rho^{2k}}{\rho^{k(\alpha+2s)}} \bigg[{a}^{ij}(\rho^k x) \partial_{ij}{U}(\rho^kx,\rho^{2sk}z) + (\rho^{2sk}z)^{2 -\frac{1}{s}} \partial_{zz}U(\rho^kx,\rho^{2sk}z) \bigg]
	=0
\end{align*}
and for any $(x,z) \in T_1$
\begin{align*}
\partial_z\tilde{U}(x,0)
	&=\frac{ \rho^{2sk}}{\rho^{k(\alpha+2s)}} \partial_z{U}(\rho^kx,0)
	= \frac{1}{\rho^{k\alpha}} f(\rho^kx) = \tilde{f}(x).
\end{align*}
That is, $\tilde{U}$ solves
\begin{equation}\label{eq:tildeU-eqn}
\begin{cases}
\tilde{a}^{ij}(x) \partial_{ij} \tilde{U}(x,z) + z^{2-\frac{1}{s}} \partial_{zz} \tilde{U}(x,z) = 0 & \hbox{in}~S_1 \times S_1^+ \\
\partial_z \tilde{U}(x,0) = \tilde{f}(x) & \hbox{on}~T_1. 
\end{cases}
\end{equation}
 
We now check the assumptions of Lemma \ref{lem:Inductive step-1} for $\tilde{U}$. It is easy to see that $\tilde{f}(0) = 0$ and $\tilde{a}^{ij}(0) = a^{ij}(0) = \delta^{ij}$, so \ref{item:A1} holds. 
Regarding \ref{item:A2}, we first change variables to find
\[
\|\tilde{a}^{ij}(x) - \delta^{ij}\|_{L^{\infty}_x(T_1)}
	= \|{a}^{ij}(\rho^kx) - \delta^{ij}\|_{L^{\infty}_x(T_1)}
	=  \|{a}^{ij}(y) - \delta^{ij}\|_{L^{\infty}_y(T_{\rho^{2k}})} \leq \eps_0.
\]
For $x \in T_1 = B_{\sqrt{2}}$, we use \ref{item:A-new} and the fact that $f(0) = 0$ to estimate 
\begin{align*}
|\tilde{f}(x)|
	= \frac{|{f}(\rho^kx) - f(0)|}{|\rho^{k}x|^\alpha}  |x|^\alpha
	\leq [f]_{C^\alpha(0)} |x|^\alpha 
	\leq [f]_{C^\alpha(0)} 2^{\alpha/2} \leq \eps_0.
\end{align*}
Together, we have
that \ref{item:A2} holds for $\tilde{a}^{ij}$ and $\tilde{f}$. 
Lastly, with a change of variables,  \eqref{eq:scale-k}, and by the inductive hypothesis, 
\begin{align*}
\|\tilde{U}\|_{L^{\infty}(S_1 \times S_1^+)}
	&=  \frac{1}{\rho^{k(\alpha+2s)}} \| U- c_k\|_{L^{\infty}(S_{\rho^{2k}} \times S_{\rho^{2k}}^+)}
	\leq  \frac{1}{\rho^{k(\alpha+2s)}} \rho^{k(\alpha+2s)} = 1,
\end{align*}
so that \ref{item:A3} holds. 

By Lemma \ref{lem:Inductive step-1}, there is a constant $c\in\R$ such that
\begin{equation}\label{eq:tildeUest-1}
\|\tilde{U}-c\|_{L^{\infty}(S_{\rho^2} \times S_{\rho^2}^+)} \leq \rho^{\alpha+2s} \quad \hbox{and} \quad |c| \leq 2. 
\end{equation}
Again by \eqref{eq:scaling0}, we note that
\begin{equation}\label{eq:rho-scaling}
(x,z) \in S_{\rho^2} \times S_{\rho^2}^+ \quad \hbox{if and only if} \quad   (y,w)= (\rho^kx,\rho^{2ks}z) \in S_{\rho^{2(k+1)}}\times S_{\rho^{2(k+1)}}^+,
\end{equation}
and rescale back to find
\begin{align*}
 \|\tilde{U}(x,z)-c\|_{L^{\infty}_{x,z}(S_{\rho^2} \times S_{\rho^2}^+)}
 	&= \|\rho^{-k(\alpha+2s)}(U(\rho^kx,\rho^{2sk}z) - c_k)-c\|_{L^{\infty}_{x,z}(S_{\rho^2} \times S_{\rho^2}^+)}\\
	&=\frac{1}{\rho^{k(\alpha+2s)}} \|U(y,w) - c_k-\rho^{k(\alpha+2s)}c\|_{L^{\infty}_{y,w}(S_{\rho^{2(k+1)}} \times S_{\rho^{2(k+1)}}^+)}.
\end{align*}
Consequently, setting $c_{k+1} := c_k +\rho^{k(\alpha+2s)}c$ and using \eqref{eq:tildeUest-1},
\[
\|U - c_{k+1}\|_{L^{\infty}(S_{\rho^{2(k+1)}} \times S_{\rho^{2(k+1)}}^+)} \leq \rho^{k(\alpha+2s)} \rho^{\alpha+2s} = \rho^{(k+1)(\alpha+2s)}
\]
and also
\[
|c_k - c_{k+1}|
	= \rho^{k(\alpha+2s)}|c| \leq 2 \rho^{k(\alpha+2s)},
\]
which completes the proof. 
\end{proof}


\begin{proof}[Proof of Theorem \ref{thm:schauder-ext}\ref{item:schauder1}]
Let $c_\infty$ be the limit of the Cauchy sequence $c_k$ in Lemma \ref{lem:sequence-1}. 
For any given $k \in \N$, we use  Lemma \ref{lem:sequence-1} to find
\begin{align*}
\|U - c_\infty\|_{L^{\infty}(S_{\rho^{2k}}\times  S_{\rho^{2k}}^+)}
	&\leq \|U - c_k\|_{L^{\infty}(S_{\rho^{2k}}  \times S_{\rho^{2k}}^+)} + \sum_{\ell=k}^{\infty} |c_\ell - c_{\ell+1}|\\
	&\leq \rho^{k(\alpha+2s)}+ 2\sum_{\ell=k}^{\infty} \rho^{\ell(\alpha+2s)} 
	= \(1 + \frac{2}{1-\rho^{\alpha+2s}}\) \rho^{k(\alpha+2s)}.
\end{align*}
Choose $k$ so that $\rho^{k+1} <r \leq \rho^k$. Then,
\begin{align*}
\|U - c_\infty\|_{L^{\infty}(S_{r^2} \times S_{r^2}^+)}
	&\leq \|U - c_\infty\|_{L^{\infty}(S_{\rho^{2k}} \times S_{\rho^{2k}}^+)}\\
	&\leq \(1 + \frac{2}{1-\rho^{\alpha+2s}}\) \rho^{k(\alpha+2s)} \\
	&\leq \(1 + \frac{2}{1-\rho^{\alpha+2s}}\) \rho^{-(\alpha+2s)} r^{\alpha+2s}
	=:C_1r^{\alpha+2s},
\end{align*}
as desired. 
Lastly, since
\[
|c_\infty| \leq \sum_{k=0}^{\infty}|c_k - c_{k+1}| 
	\leq 2 \sum_{k=0}^{\infty} \rho^{k(\alpha+2s)}
	= \frac{2}{1-\rho^{\alpha+2s}},
\]
there is a constant $C_0>0$ such that $C_0  \geq C_1  + |c_\infty|$. 
\end{proof}

\subsection{Proof of Theorem \ref{thm:schauder-ext}\ref{item:schauder2}}

\begin{lem}\label{lem:Inductive step-2}
Given $1 < \alpha +2s < 2$, there exist $0 < \varepsilon_0, \rho <1$, a linear function $\ell(x) = \langle b,x \rangle +c$, and a constant $D>0$ such that if \ref{item:A1} and \ref{item:A2} hold, then for any solution $U$ satisfying \ref{item:A3}, it holds that
\[
\|U - \ell\|_{L^{\infty}(S_{\rho^2} \times S_{\rho^2}^+)} \leq \rho^{\alpha +2s} \quad \hbox{and} \quad |b| + |c| \leq D
\]
and $D$ depends only on $n$ and $s$. 
\end{lem}

\begin{proof}
Let $0 < \eps < 1$ to be determined. 
Take $\eps_0$ as in Lemma \ref{lem:approximation}, so with \ref{item:A2},
there is a solution $H$ to \eqref{eq:H} such that 
\[
\|U-H\|_{L^{\infty}((S_{3/4} \times S_{3/4}^+)\cup T_{3/4})} \leq \eps. 
\]
Set
\[
\ell(x) := \langle \nabla_x H(0,0), x \rangle + H(0,0). 
\]
By Proposition \ref{lem:H-classical}, there is a constant $D = D(n,s)>0$  such that  $|\nabla_x H(0,0)| + |H(0,0)| \leq D$.
Also, by Proposition \ref{lem:H-classical}, for any $(x,z) \in (S_\kappa \times S_{\kappa}^+) \cup T_{\kappa}$ with $\kappa = \kappa(s)>0$ sufficiently small, we have
\begin{align*}
|H(x,z) - \ell(x)|
	&\leq |H(x,z) - H(x,0)| + |H(x,0) - \langle \nabla_x H(0,0), x \rangle - H(0,0)|\\
	&\leq \|\partial_zH(x,\cdot)\|_{L_z^{\infty}((S_\kappa \times S_{\kappa}^+) \cup T_\kappa)}z + \frac{1}{2}\|D^2_xH\|_{L^{\infty}((S_\kappa \times S_{\kappa}^+) \cup T_\kappa)} |x|^2\\
	&\leq C(z^{\frac{1}{s}} + |x|^2) 
	\leq C\Phi(x,z)= C\delta_{\Phi}((0,0), (x,z)).
\end{align*}
Consequently, if $0 < \rho^2 < \kappa$, then
\begin{align*}
\|U-\ell\|_{L^{\infty}(S_{\rho^2} \times S_{\rho^2}^+)}
	&\leq \|U-H\|_{L^{\infty}(S_{\rho^2} \times S_{\rho^2}^+)} + \|H-\ell\|_{L^{\infty}(S_{\rho^2} \times S_{\rho^2}^+)}\\
	&\leq \eps + C\rho^2
	\leq \rho^{\alpha+2s}
\end{align*}
by first choosing $\rho$ small enough to guarantee that $C \rho^2 \leq \frac{1}{2} \rho^{\alpha+2s}$ and then
selecting $\eps>0$ sufficiently small so that $\eps \leq \frac{1}{2}\rho^{\alpha+2s}$. 
\end{proof}


\begin{lem}\label{lem:sequence-2}
In the setting of Lemma \ref{lem:Inductive step-2}, suppose additionally that \ref{item:A-new} holds. 
Then there is a sequence of linear functions 
\[
\ell_k(x) = \langle b_k, x \rangle + c_k, \quad k \geq 0,
\]
such that
\[
\|U - \ell_k\|_{L^{\infty}(S_{\rho^{2k}} \times S_{\rho^{2k}}^+)} \leq \rho^{k(\alpha+2s)}  
\quad \hbox{and} \quad |c_k - c_{k+1}|,\rho^k |b_k - b_{k+1}| \leq D \rho^{k(\alpha+2s)}.
\]
\end{lem}

\begin{proof}
We prove the lemma by induction. Set $c_0 = c_1 = 0$ and $b_0 = b_1 = 0$, so that the lemma holds trivially for $k=0$. 
Now assume that the statement holds for some $k \geq 0$. 
Recalling \eqref{eq:scale-k}, consider 
\[
\tilde{U}(x,z)
	= \frac{1}{\rho^{k(\alpha+2s)}} (U(\rho^kx, \rho^{2sk}z) - \ell_k(\rho^kx)), \quad (x,z) \in (S_1 \times S_1^+) \cup T_1. 
\] 
Set $\tilde{a}^{ij}$ and $\tilde{f}$ as in \eqref{eq:a-f-tilde}.  
As in the proof of Lemma \ref{lem:sequence-1}, 
we can readily check that \ref{item:A1} and \ref{item:A2} hold for $\tilde{a}^{ij}$ and $\tilde{f}$ and that $\tilde{U}$ is a solution to \eqref{eq:tildeU-eqn}.
Moreover, with a change of variables, \eqref{eq:scale-k}, and by the inductive hypothesis, 
\begin{align*}
\|\tilde{U}\|_{L^{\infty}(S_1 \times S_1^+)}
	&=  \frac{1}{\rho^{k(\alpha+2s)}} \| U(\rho^kx,\rho^{2sk}z)- \ell_k(\rho^kx)\|_{L^{\infty}_{x,z}(S_1 \times S_{1}^+)}\\
	&=  \frac{1}{\rho^{k(\alpha+2s)}} \| U(y,w)- \ell_k(y)\|_{L^{\infty}_{y,w}(S_{\rho^{2k}} \times S_{\rho^{2k}}^+)}
	\leq 1,
\end{align*}
so we also have \ref{item:A3}.
In particular, the hypotheses of Lemma \ref{lem:Inductive step-2} hold for $\tilde{U}$. 

By Lemma \ref{lem:Inductive step-2}, there is a linear function $\ell(x) = \langle b,x \rangle + c$ and a constant $D$ such that
\begin{equation}\label{eq:tildeUest-2}
\|\tilde{U}-\ell\|_{L^{\infty}(S_{\rho^2} \times S_{\rho^2}^+)} \leq \rho^{\alpha+2s} \quad \hbox{and} \quad |b| + |c| \leq D. 
\end{equation}
Recalling \eqref{eq:rho-scaling}, we rescale back to find
\begin{align*}
\|\tilde{U}(x,z)-\ell(x)\|_{L^{\infty}_{x,z}(S_{\rho^2} \times S_{\rho^2}^+)}
	= \|\rho^{-k(\alpha+2s)}(\tilde{U}(\rho^{k}x, \rho^{2ks}z) - \ell_k(\rho^kx))-\ell(x)\|_{L^{\infty}_{x,z}(S_{\rho^2} \times S_{\rho^2}^+)}&\\
	=\frac{1}{\rho^{k(\alpha+2s)}} \|\tilde{U}(y,w) - \ell_k(y)-\rho^{k(\alpha+2s)}\ell(\rho^{-k}y)\|_{L^{\infty}_{y,w}(S_{\rho^{2(k+1)}}\times S_{\rho^{2(k+1)}}^+)}&.
\end{align*}
Consequently, setting $\ell_{k+1}(x) = \ell_k(x) +\rho^{k(\alpha+2s)}\ell(\rho^{-k}x)$ and using \eqref{eq:tildeUest-2}, 
\[
\|U - \ell_{k+1}\|_{L^{\infty}_{y,w}(S_{\rho^{2(k+1)}}\times S_{\rho^{2(k+1)}}^+)} \leq \rho^{k(\alpha+2s)} \rho^{\alpha+2s} = \rho^{(k+1)(\alpha+2s)}
\]
and also
\begin{align*}
|c_k - c_{k+1}|
	&= \rho^{k(\alpha+2s)}|c| \leq D \rho^{k(\alpha+2s)}\\
\rho^k|b_k - b_{k+1}|
	&= \rho^k \rho^{k(\alpha+2s)} |\rho^{-k}b| 
		\leq D \rho^{k(\alpha+2s)}
\end{align*}
which completes the proof. 
\end{proof}


\begin{proof}[Proof of Theorem \ref{thm:schauder-ext}\ref{item:schauder2}]
Let $\ell_\infty$ be the limit of the sequence $\ell_k$ in Lemma \ref{lem:sequence-2}. In particular, since the sequences $b_k, c_k$ are Cauchy, 
\[
\ell_\infty(x) := \langle b_\infty, x \rangle + c_\infty \quad \hbox{where} \quad
 \lim_{k \to \infty} c_k = c_{\infty}, \quad 
 \lim_{k \to \infty} b_k = b_{\infty}.
\]
For any given $k \in \N$, note that
\begin{equation}\label{eq:x-scaling}
\hbox{if}~x \in T_{\rho^{2k}}~\hbox{then}~|x| \leq \sqrt{2}\rho^{k},
\end{equation}
so, applying Lemma \ref{lem:sequence-2}, we find
\begin{align*}
\|U - \ell_\infty\|_{L^{\infty}(S_{\rho^{2k}}\times S_{\rho^{2k}}^+)}
	&\leq \|U - \ell_k\|_{L^{\infty}(S_{\rho^{2k}} \times S_{\rho^{2k}}^+)} + \|\ell_k-\ell_\infty\|_{L^{\infty}( T_{\rho^{2k}})}\\
	&\leq  \|U - \ell_k\|_{L^{\infty}(S_{\rho^{2k}} \times S_{\rho^{2k}}^+)}+ |b_k - b_{\infty}|\sqrt{2}\rho^{k} + |c_k - c_{\infty}|\\
	&\leq  \|U - \ell_k\|_{L^{\infty}(S_{\rho^{2k}}\times S_{\rho^{2k}}^+)}+ \sqrt{2}\rho^k\sum_{j=k}^{\infty}|b_k - b_{j+1}| + \sum_{j=k}^{\infty}|c_j - c_{j+1}|\\
	&\leq \rho^{k(\alpha+2s)}+ \sqrt{2}\rho^k\sum_{j=k}^{\infty}D \rho^{j(\alpha+2s-1)} + \sum_{j=k}^{\infty}D \rho^{j(\alpha+2s)}\\
	 &\leq \(1+ \frac{(\sqrt{2}+1)D}{1-\rho^{\alpha+2s-1}}\)\rho^{k(\alpha+2s)}.
\end{align*}
Choose $k$ so that $\rho^{k+1} <r \leq \rho^k$. Then,
\begin{align*}
\|U - \ell_\infty\|_{L^{\infty}(S_{r^2} \times S_{r^2}^+)}
	&\leq \|U - \ell_\infty\|_{L^{\infty}(S_{\rho^{2k}} \times S_{\rho^{2k}}^+)}\\
	&\leq \(1+ \frac{(\sqrt{2}+1)D}{1-\rho^{\alpha+2s-1}}\)\rho^{k(\alpha+2s)} \\
	&\leq\(1+ \frac{(\sqrt{2}+1)D}{1-\rho^{\alpha+2s-1}}\) \rho^{-(\alpha+2s)} r^{\alpha+2s}
	=:C_1r^{\alpha+2s}
\end{align*}
as desired. It remains to note that, since 
\begin{align*}
|c_\infty| &\leq \sum_{k=0}^{\infty}|c_k - c_{k+1}| 
	\leq D \sum_{k=0}^{\infty} \rho^{k(\alpha+2s)}
	\leq D \frac{1}{1-\rho^{\alpha+2s}}\\
|b_\infty| 
	&\leq \sum_{k=0}^{\infty} |b_k - b_{k+1}| 
	\leq  \sum_{k=0}^{\infty}  D \rho^{-k}\rho^{k(\alpha+2s)}
	\leq D \frac{1}{1-\rho^{\alpha+2s-1}},
\end{align*}
there is a constant $C_0>0$ such that $C_0  \geq C_1 + |b_\infty| + |c_\infty|$. 
\end{proof}

\subsection{Proof of Theorem \ref{thm:schauder-ext}\ref{item:schauder3}}

\begin{lem}\label{lem:Inductive step-3}
Given $2 < \alpha +2s < 3$, there exist $0 < \varepsilon_0, \rho <1$, a second-order Monge--Amp\`ere polynomial
\[
P(x,z) = \frac{1}{2} \langle \mathcal{A}x, x \rangle +  \langle b,x \rangle +c +dh(z),
\]
and a constant $D>0$ such that 
if \ref{item:A1} and \ref{item:A2-} hold,
 then for any solution $U$ satisfying \ref{item:A3-}, it holds that
\[
\|U - P\|_{L^{\infty}(S_{\rho^2} \times S_{\rho^{3}}^+)} \leq \rho^{\alpha +2s} \quad \hbox{and} \quad  |\mathcal{A}|+ |b| + |c|+|d| \leq D, \quad \A^{ii} +d= 0,
\]
and $D$ depends only on $n$ and $s$.
\end{lem}

\begin{proof}
Let $0 < \eps, \rho < 1$ to be determined. By Lemma \ref{lem:approximation} with Remark \ref{rem:approximation}, there is a $\varepsilon_0>0$, so with \ref{item:A2-}, there is a solution $H$ to \eqref{eq:H} in $(S_{3/4} \times S_{3\rho/4}^+) \cup T_{3/4}$ such that 
\[
\|U-H\|_{L^{\infty}((S_{3/4} \times S_{3\rho/4}^+) \cup T_{3/4})} \leq \eps. 
\]
Set
\[
P(x,z) = \frac{1}{2} \langle D^2_xH(0,0)x, x \rangle +  \langle \nabla_xH(0,0), x \rangle + H(0,0) - \Delta_xH(0,0) h(z).
\]
By Proposition \ref{lem:H-classical}, there is a constant $D = D(n,s)>0$ such that  
\[
|D^2_x H(0,0)|+|\nabla_x H(0,0)| + |H(0,0)| + |\Delta_xH(0,0)| \leq D.
\]
Also note that
\[
A^{ii} + d 
	= \Delta_xP(x,z) + |z|^{2-\frac{1}{s}}\partial_{zz}P(x,z)
	=  \Delta_xH(0,0) - \Delta_xH(0,0) = 0. 
\]
It remains to estimate $\|U-H\|_{L^{\infty}}$.  
For this, we again apply Proposition \ref{lem:H-classical} so that, for any $(x,z) \in (S_{\kappa^2} \times S_{\kappa^3}^+) \cup T_{\kappa^2}$ with $\kappa$ sufficiently small, we have
\begin{align*}
|H(x,z) - P(x,z)|
	&\leq |H(x,z) - H(x,0)|  +|\Delta_xH(0,0)| h(z)\\
	&\quad +  |H(x,0) - \frac{1}{2} \langle D^2_xH(0,0)x,x \rangle-\langle \nabla_x H(0,0), x \rangle - H(0,0)|\\
	&\leq \|\partial_zH(x,\cdot)\|_{L_z^{\infty}((S_{\kappa^2} \times S_{\kappa^3}^+) \cup T_{\kappa^2})}z  + Cz^{\frac{1}{s}}+ \frac{1}{6}\|D^3_xH\|_{L^{\infty}((S_{\kappa^2} \times S_{\kappa^3}^+) \cup T_{\kappa^2})} |x|^3\\
	&\leq C(z^{\frac{1}{s}} +|x|^3) \\
	&\leq C(\delta_h(0,z) + (\delta_\varphi(0,x))^{\frac{3}{2}}).
\end{align*}
Consequently, if $0 < \rho < \kappa$, then
\begin{align*}
\|U - P\|_{L^{\infty}(S_{\rho^2} \times S_{\rho^3}^+)}
	&\leq \|U-H\|_{L^{\infty}(S_{\rho^2} \times S_{\rho^3}^+)}+ \|H-P\|_{L^{\infty}(S_{\rho^2} \times S_{\rho^3}^+)}\\
	&\leq \eps + C\rho^3
	\leq \rho^{\alpha+2s}
\end{align*}
by first choosing $\rho$ small enough to guarantee that $C \rho^3 \leq \frac{1}{2} \rho^{\alpha+2s}$ and then
letting $\eps>0$ sufficiently small so that $\eps \leq \frac{1}{2}\rho^{\alpha+2s}$. 
\end{proof}


\begin{lem}\label{lem:sequence-3}
In the setting of Lemma \ref{lem:Inductive step-3}, assume additionally that \ref{item:A-new} and \ref{item:A4-} hold. 
Then there is a sequence of second-order Monge--Amp\`ere polynomials
\[
P_k(x,z) = \frac{1}{2} \langle \A_k x, x \rangle + \langle b_k, x \rangle + c_k +d_k h(z), \quad k \geq 0,
\]
such that

\[
\|U - P_k\|_{L^{\infty}(S_{\rho^{2k}} \times S_{\rho^{2k+1}}^+)} \leq \rho^{k(\alpha+2s)}. 
\]
and
\[
|c_k - c_{k+1}|,\rho^k |b_k - b_{k+1}|,\rho^{2k} |\A_k - \A_{k+1}|,\rho^{2k} |d_k - d_{k+1}| \leq D \rho^{k(\alpha+2s)},
\quad
\A^{ii}_k + d_k = 0.
\]
\end{lem}

\begin{proof}
We prove the lemma by induction. Set $P_0 = P_1 \equiv 0$, so that the lemma holds trivially for $k=0$. 
Now assume that the statement holds for some $k \geq 0$. 
Recalling \eqref{eq:scale-k}, consider 
\[
\tilde{U}(x,z)
	= \frac{1}{\rho^{k(\alpha+2s)}} (U(\rho^kx, \rho^{2sk}z) - P_k(\rho^kx, \rho^{2sk}z)), \quad (x,z) \in (S_1 \times S_1^+) \cup T_1. 
\] 
As in Lemma \ref{lem:scaling sections}, it is easy to check that
\begin{equation}\label{eq:scale-k-3}
(x,z) \in (S_1 \times S_\rho) \cup T_1 \quad \hbox{if and only if} \quad (\rho^kx,\rho^{2sk}z) \in (S_{\rho^{2k}} \times S_{\rho^{2k+1}})  \cup T_{\rho^{2k}},
\end{equation}
so $\tilde{U}$ is well-defined. 
Set $\tilde{a}^{ij}$ and $\tilde{f}$ as in \eqref{eq:a-f-tilde} and also
\[
\tilde{F}(x) = - \rho^{-k(\alpha+2s-2)} [\tilde{a}^{ij}(x) \mathcal{A}_k^{ij} + d_k].
\] 
Using Lemma \ref{lem:scaling equation}, for any $(x,z) \in S_1 \times S_\rho^+$,
\begin{align*}
\tilde{a}^{ij}(x) \partial_{ij}\tilde{U}(x,z) &+ z^{2-\frac{1}{s}}\partial_{zz} \tilde{U}(x,z)\\
	&= 0-\frac{1}{\rho^{k(\alpha+2s)}}\big[\rho^{2k}{a}^{ij}(\rho^kx) \partial_{ij}P_k(\rho^{k}x,\rho^{2ks}z)+ \rho^{4ks}z^{2-\frac{1}{s}}\partial_{zz} P_k(\rho^{k}x,\rho^{2ks}z)\big]\\ 
	&= -\frac{1}{\rho^{k(\alpha+2s)}}\big[\rho^{2k}{a}^{ij}(\rho^kx) \A_k^{ij}+ \rho^{4ks}z^{2-\frac{1}{s}}d_k (\rho^{2ks} z)^{\frac{1}{s}-2}\big]\\ 
	&= -\frac{1}{\rho^{k(\alpha+2s-2)}}[{a}^{ij}(\rho^kx) \A_k^{ij}+ d_k]
	=\tilde{F}(x)
\end{align*}
and for any $(x,0) \in T_1$,
\begin{align*}
\partial_{z}\tilde{U}(x,0)
	&=\frac{ \rho^{2sk}}{\rho^{k(\alpha+2s)}} \partial_{z}{U}(\rho^kx,0)-0
	= \frac{1}{\rho^{k\alpha}} f(\rho^kx) = \tilde{f}(x).
\end{align*}
That is, $\tilde{U}$ solves
\[
\begin{cases}
\tilde{a}^{ij}(x) \partial_{ij} \tilde{U}(x,z) + |z|^{2-\frac{1}{s}} \partial_{zz} \tilde{U}(x,z) = \tilde{F}(x) & \hbox{in}~S_1\times S_\rho^+ \\
\partial_z \tilde{U}(x,0) = \tilde{f}(x) & \hbox{on}~T_1. 
\end{cases}
\]

We now check the assumptions of Lemma \ref{lem:sequence-3} for $\tilde{U}$. 
As in the proof of Lemma \ref{lem:sequence-1}, 
we can readily check that \ref{item:A1} and \ref{item:A2} hold for $\tilde{a}^{ij}$ and $\tilde{f}$.  
For $x \in T_1$, we use \ref{item:A4-} to estimate
\begin{align*}
\frac{1}{\rho^{k(\alpha+2s-2)}}|a^{ij}(\rho^kx) - \delta^{ij}|
	&= \frac{|a^{ij}(\rho^kx) - \delta^{ij}|}{|\rho^kx|^{\alpha+2s-2}}|x|^{\alpha+2s-2}\\
	&\leq [a^{ij}]_{C^{\alpha+2s-2}(0)}|x|^{\alpha+2s-2} \leq C[a^{ij}]_{C^{\alpha+2s-2}(0)}.
\end{align*}
Next, note that
\begin{align*}
| \A_k| \leq \sum_{j=0}^{k-1}  |\A_j - \A_{j+1}| \leq D \sum_{j=0}^{k-1}  \rho^{j(\alpha+2s-2)} \leq \frac{D}{1-\rho^{\alpha+2s-2}} < \infty
\end{align*}
and
\[
a^{ij}(\rho^kx) \A_k^{ij} +d_k 
	= (a^{ij}(\rho^kx) - \delta^{ij})\A_k^{ij}.
\]
Consequently, we find that
\begin{align*}
\| \tilde{F}\|_{L^{\infty}(T_1)}
	&= \rho^{-k(\alpha+2s-2)} \|(a^{ij}(\rho^kx) - \delta^{ij})\A_k^{ij} \|_{L^{\infty}(T_1)}\\
	&\leq C[a^{ij}]_{C^{\alpha+2s-2}(0)}  |\A_k|\\
	&\leq \bar{C}[a^{ij}]_{C^{\alpha+2s-2}(0)}
	\leq \eps_0,
\end{align*}
so with \ref{item:A2}, we have \ref{item:A2-}. 
Lastly, with a change of variables, \eqref{eq:scale-k-3}, and by the inductive hypothesis, 
\begin{align*}
\|\tilde{U}\|_{L^{\infty}(S_1 \times S_\rho^+)}
&= \frac{1}{\rho^{k(\alpha+2s)}} \|U(\rho^kx,\rho^{2sk}z) - P_k(\rho^kx,\rho^{2sk}z)\|_{L^{\infty}_{x,z}(S_1 \times S_\rho^+)}\\
&= \frac{1}{\rho^{k(\alpha+2s)}} \|U(y,w) - P_k(y,w)\|_{L_{y,w}^{\infty}(S_{\rho^{2k}} \times S_{\rho^{2k+1}}^+)}
 \leq 1,
\end{align*}
so we also have \ref{item:A3-}. In particular, the hypotheses of Lemma \ref{lem:Inductive step-3} hold for $\tilde{U}$. 
 
By Lemma  \ref{lem:Inductive step-3}, there is a second-order Monge--Amp\`ere polynomial $P(x,z) = \frac{1}{2} \langle \A x,x \rangle + \langle b,x \rangle + c +dh(z)$ and a constant $D$ such that
\begin{equation}\label{eq:tildeUest-3}
\|\tilde{U}-P\|_{L^{\infty}(S_{\rho^2}\times S_{\rho^3}^+)} \leq \rho^{\alpha+2s} \quad \hbox{and} \quad |\A| +  |b| + |c| +|d| \leq D. 
\end{equation}
Like in \eqref{eq:scale-k-3}, it is straightforward to check that
\[
(x,z) \in S_{\rho^2} \times S_{\rho^3}^+ \quad \hbox{if and only if} \quad 
(y,w) = (\rho^kx, \rho^{2sk}z) \in S_{\rho^{2(k+1)}}  \times S_{\rho^{2(k+1)+1}}^+. 
\]
With this, we rescale back to write
\begin{align*}
&\|\tilde{U}-P\|_{L_{x,z}^{\infty}(S_{\rho^2}\times S_{\rho^3}^+)} \\
	&\qquad= \|\rho^{-k(\alpha+2s)}({U}(\rho^{k}x, \rho^{2ks}z) - P_k(\rho^kx ,\rho^{2sk}z))-P(x,z)\|_{L_{x,z}^{\infty}(S_{\rho^2}\times S_{\rho^3}^+)} \\
	&\qquad= \frac{1}{\rho^{k(\alpha+2s)}}\|{U}(y, w) - (P_k(y,w) + \rho^{k(\alpha+2s)}P(\rho^{-k}y, \rho^{-2sk}w))\|_{L_{y,w}^{\infty}(S_{\rho^{2k+2}}\times S_{\rho^{2k+3}}^+)}.
\end{align*}
Consequently, setting $P_{k+1}(x,z) = P_k(x,z) +\rho^{k(\alpha+2s)}P(\rho^{-k}x,\rho^{-2ks}z)$ and using \eqref{eq:tildeUest-3},
\[
\|U - P_{k+1}\|_{L^{\infty}(S_{\rho^{2(k+1)}}\times S_{\rho^{2(k+1)+1}}^+)} \leq \rho^{k(\alpha+2s)} \rho^{\alpha+2s} = \rho^{(k+1)(\alpha+2s)}
\]
and also
\begin{align*}
|c_k - c_{k+1}|
	&= \rho^{k(\alpha+2s)}|c| \leq D \rho^{k(\alpha+2s)}\\
\rho^k|b_k - b_{k+1}|
	&= \rho^k \rho^{k(\alpha+2s)} |\rho^{-k}b|  
	\leq D \rho^{k(\alpha+2s)}\\
\rho^{2k}|\A_k - \A_{k+1}|
	&= \rho^{2k} \rho^{k(\alpha+2s)}|\rho^{-2k}\A| 
	\leq D \rho^{k(\alpha+2s)}\\
\rho^{2k}|d_k - d_{k+1}|
	&= \rho^{2k} \rho^{k(\alpha+2s)}|(\rho^{-2sk})^{\frac{1}{s}}d| 
	\leq D \rho^{k(\alpha+2s)}.
\end{align*}
Lastly, we check that
\[
\A_{k+1}^{ii} +d_{k+1}
	= \A_k^{ii} +d_k+ \rho^{k(\alpha+2s)}\A^{ii} + \rho^{k(\alpha+2s)}d= 0.
\]
\end{proof}


\begin{proof}[Proof of Theorem  \ref{thm:schauder-ext}\ref{item:schauder3}]
Let $P_\infty$ be the limit  polynomial of the sequence $P_k$ in Lemma \ref{lem:sequence-3}. In particular, since $\mathcal{A}_k^{ij}, b_k, c_k, d_k$ are Cauchy sequences, 
\[
P_\infty(x,z)= \frac{1}{2} \langle \A_\infty x, x \rangle + \langle b_\infty, x \rangle + c_\infty +d_\infty
\]
where
\[
 \lim_{k \to \infty} c_k = c_{\infty}, \quad 
 \lim_{k \to \infty} b_k = b_{\infty}, \quad
  \lim_{k \to \infty} \A_k = \A_{\infty}, \quad
    \lim_{k \to \infty} d_k = d_{\infty}.
\]
For any given $k \geq 0$, we 
recall \eqref{eq:x-scaling} and apply Lemma \ref{lem:sequence-3} to estimate
\begin{align*}
\|&P_k-P_\infty\|_{L^{\infty}(S_{\rho^{2k}} \times S_{\rho^{2k+1}}^+)}\\
	&\leq \rho^{2k} |\A_k - \A_{\infty}| 
		+ \sqrt{2}\rho^k|b_k - b_{\infty}|
		+ |c_k - c_{\infty}|
		+ \rho^{2k+1}|d_k -d_\infty|  \\
	&\leq \rho^{2k}\sum_{j=k}^{\infty}|\A_k - \A_{j+1}| 
		+ \sqrt{2}\rho^k\sum_{j=k}^{\infty}|b_k - b_{j+1}| 
		+ \sum_{j=k}^{\infty}|c_j - c_{j+1}|
		+ \rho^{2k+1}  \sum_{j=k}^{\infty}|d_j - d_{j+1}|\\
	&\leq  \rho^{2k} \sum_{j=k}^{\infty}D \rho^{j(\alpha+2s-2)} 
		+ \sqrt{2}\rho^k\sum_{j=k}^{\infty}D \rho^{j(\alpha+2s-1)} 
		+ \sum_{j=k}^{\infty}D \rho^{j(\alpha+2s)}
		+\rho^{2k+1} \sum_{j=k}^{\infty}D \rho^{j(\alpha+2s-2)}\\
	 &=  D \rho^{2k}\frac{\rho^{k(\alpha+2s-2)}}{1-\rho^{\alpha+2s-2}} 
	 	+ \sqrt{2}D\rho^k \frac{\rho^{k(\alpha+2s-1)}}{1-\rho^{\alpha+2s-1}} 
		+ D \frac{\rho^{k(\alpha+2s)}}{1-\rho^{\alpha+2s}}
		+ D \rho^{2k+1}\frac{\rho^{k(\alpha+2s-2)}}{1-\rho^{\alpha+2s-2}} \\
	 &\leq \( \frac{(3 +\sqrt{2})D}{1-\rho^{\alpha+2s-2}}\)\rho^{k(\alpha+2s)},
\end{align*}
where we use that $\rho \leq 1$ to estimate $\rho^{2k+1} \leq \rho^{2k}$. 
Therefore, by applying again Lemma \ref{lem:sequence-3},
\begin{align*}
\|U - P_\infty\|_{L^{\infty}(S_{\rho^{2k}} \times S_{\rho^{2k+1}}^+)}
	&\leq \|U - P_k\|_{L^{\infty}(S_{\rho^{2k}} \times S_{\rho^{2k+1}}^+)}
		 + \|P_k-P_\infty\|_{L^{\infty}(S_{\rho^{2k}} \times S_{\rho^{2k+1}}^+)}\\
	 &\leq \(1+ \frac{(3 +\sqrt{2})D}{1-\rho^{\alpha+2s-2}}\)\rho^{k(\alpha+2s)}.
\end{align*}
Choose $k$ so that $\rho^{k+1} <r \leq \rho^k$. 
Since $0 < r < 1$, we have
\[
r^{3} < r^{2+ \frac{1}{k}}  \leq \rho^{2k+1}. 
\]
Therefore, we arrive at the desired estimate
\begin{align*}
\|U - P_\infty\|_{L^{\infty}(S_{r^2} \times S_{r^3}^+)}
	&\leq \|U - P_\infty\|_{L^{\infty}(S_{\rho^{2k}} \times S_{\rho^{2k+1}}^+)}\\
	&\leq \(1+ \frac{(3 +\sqrt{2})D}{1-\rho^{\alpha+2s-2}}\)\rho^{k(\alpha+2s)}\\
	&\leq \(1+ \frac{(3 +\sqrt{2})D}{1-\rho^{\alpha+2s-2}}\)\rho^{-(\alpha+2s)} r^{\alpha+2s}
	=:C_1r^{\alpha+2s}.
\end{align*}
It remains to note that, since
\begin{align*}
|c_\infty| &\leq \sum_{k=0}^{\infty}|c_k - c_{k+1}| 
	\leq D \sum_{k=0}^{\infty} \rho^{k(\alpha+2s)}
	\leq  \frac{D}{1-\rho^{\alpha+2s}}\\
|b_\infty| 
	&\leq \sum_{k=0}^{\infty} |b_k - b_{k+1}| 
	\leq  \sum_{k=0}^{\infty}  D \rho^{-k}\rho^{k(\alpha+2s)}
	\leq  \frac{D}{1-\rho^{\alpha+2s-1}}\\
|\A_\infty|
	&\leq \sum_{k=0}^{\infty} |\A_k - \A_{k+1}| 
	\leq \sum_{k=0}^{\infty}  D \rho^{-2k}\rho^{k(\alpha+2s)}
	\leq   \frac{D}{1-\rho^{\alpha+2s-2}}\\
|d_{\infty}|
	&\leq \sum_{k=0}^{\infty}|d_k - d_{k+1}| \leq \sum_{k=0}^{\infty}D\rho^{k(\alpha+2s-2)} \leq \frac{D}{1-\rho^{\alpha+2s-2}},
\end{align*}
there is a constant $C_0>0$ such that $C_0  \geq C_1 + |b_\infty| + |c_\infty|+|d_\infty|$.
\end{proof}

\section*{Acknowledgements}
The first author was supported by Simons Foundation grants
580911 and MP-TSM-00002709.
The second author was supported by Australian Laureate Fellowship FL190100081 ``Minimal surfaces, free boundaries and partial differential equations.''




\begin{thebibliography}{10}

\bibitem{Arendt} W.~Arendt and R.~M.~Sch\"atzle,
{Semigroups generated by elliptic operators in non-divergence form on $C_0(\Omega)$},
\textit{Ann. Sc. Norm. Super. Pisa Cl. Sci.}
\textbf{13} (2014), 1--18. 
\filbreak

\bibitem{Berens} H.~Berens, P.~L.~Butzer and U.~Westphal, 
{Representation of fractional powers of infinitesimal generators of semigroups}, 
\textit{Bull. Amer. Math. Soc.} 
\textbf{74} (1968), 191--196.
\filbreak

\bibitem{Biswas-Stinga} A.~Biswas and P.~R.~Stinga,
{Sharp extension problem characterizations for higher fractional power operators in Banach spaces}, 
arXiv:2309.12512
(2023), 20~pp.
\filbreak

\bibitem{Caffarelli-Annals} L.~Caffarelli,
{Interior a priori estimates for solutions of fully nonlinear equations},
\textit{Ann. of Math.}
\textbf{130} (1989), 189--213.
\filbreak

\bibitem{Caffarelli-Charro} L.~Caffarelli and F.~Charro, 
{On a fractional Monge-Amp\`ere operator,}
\textit{Ann. PDE}
\textbf{1} (2015), Art.~4, 47 pp.
\filbreak

\bibitem{Caffarelli-Gutierrez} L.~A.~Caffarelli and C.~E.~Guti\'errez,
{Properties of the solutions of the linearized Monge--Amp\`ere equation}, 
\textit{Amer. J. Math.}
\textbf{119} (1997), 423--465.
\filbreak

\bibitem{Caffarelli-Silvestre} L.~Caffarelli and L.~Silvestre,
{An extension problem related to the fractional Laplacian},
\textit{Comm. Partial Differential Equations}
\textbf{32} (2007), 1245--1260.
\filbreak

\bibitem{Caffarelli-Stinga} L.~Caffarelli and P.~R.~Stinga,
{Fractional elliptic equations, Caccioppoli estimates and regularity},
\textit{Ann. Inst. H. Poincar\'e C Anal. Non Lin\'eaire} 
\textbf{33} (2016), 767--807.
\filbreak

\bibitem{Cont} R.~Cont and P.~Tankov,
\textit{Finantial Modelling with Jump Processes},
Chapman \& Hall/CRC, 
2003.
\filbreak

\bibitem{Daskalopoulos-Savin} P.~Daskalopoulos and O.~Savin,
{On Monge--Amp\`re equations with homogeneous right-hand sides}, 
\textit{Comm. Pure Appl. Math.} 
\textbf{62} (2009), 639--676.
\filbreak

\bibitem{DeSilva-Ferrari-Salsa} D.~De~Silva, F.~Ferrari, and S.~Salsa, 
{Regularity of transmission problems for uniformly elliptic fully nonlinear equations}, 
\textit{Electron. J. Differ. Equ.} 
\textbf{25} (2018), 55--63.
\filbreak

\bibitem{Duvaut} G.~Duvaut and J.~L.~Lions,
\textit{Inequalities in Mechanics and Physics},
Springer Verlag, Berlin, 1976.
\filbreak

\bibitem{Forzani} L.~Forzani and D.~Maldonado,
{A mean-value inequality for nonnegative solutions to the linearized Monge--Amp\`ere equation},
\textit{Potential Anal.}
\textbf{30} (2009), 251--270.
\filbreak

\bibitem{Gale} J.~E.~Gal\'e, P.~J.~Miana, and P.~R.~Stinga,
{Extension problem and fractional operators: semigroups and wave equations},
\textit{J.~Evol.~Equ.}
\textbf{13} (2013), 343--386.
\filbreak

\bibitem{GT} D.~Gilbarg and N.~S.~Trudinger, 
\textit{Elliptic Partial Differential Equations of Second Order}, Springer-Verlag, 
Berlin, Heidelberg, 2001.
\filbreak

\bibitem{Grafakos} L.~Grafakos, 
\textit{Modern Fourier Analysis}, Second edition,
 Graduate Texts in Mathematics \textbf{250}.,
 Springer, New York, 2009.
 \filbreak

\bibitem{Gutierrez} C.~E.~Guti\'errez,
\textit{The Monge-Amp\`ere Equation},
Progress in Nonlinear Differential Equations and Their Applications {\bf 44}, 
Birkh\"auser, 
Basel, 2001.
\filbreak

\bibitem{Gutierrez--Nguyen} C.~E.~Guti\'errez and T.~Nguyen,
{Interior second derivative estimates for solutions to the linearized Monge--Amp\`ere equation},
\textit{Trans. Amer. Math. Soc.}
\textbf{367} (2015), 4537--4568.
\filbreak

\bibitem{Jhaveri-Stinga} Y.~Jhaveri and P.~R.~Stinga,
{The obstacle problem for a fractional Monge--Amp\`ere equation},
\textit{Comm. Partial Differential Equations}
\textbf{45} (2020), 457--482.
\filbreak

\bibitem{Le-Savin} N.~Le and O.~Savin,
{Schauder estimates for degenerate Monge--Amp\`re equations and smoothness of the eigenfunctions}, 
\textit{Invent. Math.} 
\textbf{207} (2017), 89--423.
\filbreak

\bibitem{MaldonadoI} D.~Maldonado,
{On certain degenerate and singular elliptic PDEs I: Nondivergence form operators with unbounded drifts and applications to subelliptic equations},
\textit{J.~Differential Equations}
\textbf{264} (2018), 624--678.
\filbreak

\bibitem{MaldonadoIII} D.~Maldonado,
{On certain degenerate and singular elliptic PDEs III: Nondivergence form operators and $RH_{\infty}$-weights},
\textit{J.~Differential Equations}
\textbf{280} (2021), 805--840.
\filbreak

\bibitem{Maldonado-Stinga} D.~Maldonado and P.~R.~Stinga,
{Harnack inequality for the fractional nonlocal linearized Monge-Amp\`ere equation},
\textit{Calc.~Var.~Partial Differential Equations}
\textbf{56} (2017), 56--103.
\filbreak

\bibitem{Oksendal-Sulem} B.~{\O}ksendal and A.~Sulem,
\textit{Applied Stochastic Control of Jump Diffusions},
Third edition, Universitext,
Springer, Cham, 2019.
\filbreak

\bibitem{Pazy} A.~Pazy,
\textit{Semigroups of Linear Operators and Applications to Partial Differential Equations},
{Applied Mathematical Sciences} \textbf{44},
Springer--Verlag, New York, 1983.
\filbreak

\bibitem{Savin} O.~Savin,
{Small perturbation solutions for elliptic equations},
\textit{Comm. Partial Differential Equations}
\textbf{32} (2007), 557--578.
\filbreak

\bibitem{Stinga-book} P.~R.~Stinga,
\textit{Regularity Techniques for Elliptic PDEs and the Fractional Laplacian},  
CRC Press, 2024.
\filbreak

\bibitem{Stinga-Torrea} P.~R.~Stinga and J.~L.~Torrea,
{Extension problem and Harnack's inequality for some fractional operators},
\textit{Comm. Partial Differential Equations}
\textbf{35} (2010), 2092--2122.
\filbreak

\bibitem{Stinga-Vaughan} P.~R.~Stinga and M.~Vaughan, 
{Fractional elliptic equations in nondivergence form: definition, applications and Harnack inequality},
\textit{J.~Math.~Pures Appl.}
\textbf{156} (2021), 245--306.
\filbreak

\bibitem{Stinga-Zhang} P.~R.~Stinga and C.~Zhang,
{Harnack's inequality for fractional nonlocal equations},
\textit{Discrete Contin. Dyn. Syst.}
\textbf{33} (2013), 3153--3170.
\filbreak

\bibitem{Yosida} K.~Yosida,
\textit{Functional Analysis},
Reprint of the sixth (1980) edition,
Classics in Mathematics,
Springer--Verlag, Berlin, 1995.
\filbreak

\end{thebibliography}
\end{document}